\documentclass[10pt]{amsart}
\usepackage[usenames]{color}
\usepackage{fullpage,amscd,amssymb,latexsym,url,hyperref}
\usepackage[all,2cell,ps]{xy}
\usepackage{fullpage}
\usepackage{graphicx}
\usepackage{epsfig,epstopdf}
\usepackage{float}
\usepackage{soul}
\usepackage{color}
\usepackage{cancel}

\usepackage{caption}

\theoremstyle{plain}
\newtheorem{theorem}{Theorem}[section]

\newtheorem{lemma}[theorem]{Lemma}
\newtheorem{proposition}[theorem]{Proposition}

\theoremstyle{definition}

\newtheorem{remark}[theorem]{Remark}

\newtheorem*{remark*}{Remark}

\newtheorem*{problem*}{Problem}

\def\aut#1{\mathrm{Aut}(#1)}

\def\B{\mathfrak{B}}
\def\Z{\mathbb Z}
\def\G{\mathcal G}

\def\k{\xi}

\def\h{H}

\def\setof#1#2{\{#1\, : \,#2\}}
\def\t{t}
\def\r{r}
\def\H{\mathfrak{H}}

\def\vs{\vspace{5pt}}

\usepackage{xspace}
\newcommand{\NC}{{(K)}\xspace}

\usepackage{xspace}
\newcommand{\nc}{{(R)}\xspace}

\newcommand{\Hol}{\operatorname{Hol}}
\newcommand{\Fix}{\operatorname{Fix}}

\newcommand{\dm}{\mathcal{D}}


\title{Skew braces of size  $p^2q$ II: non-abelian type}

\begin{document}

\author{E. Acri}
\author{M. Bonatto}

\address[E. Acri, M. Bonatto]{IMAS--CONICET and Universidad de Buenos Aires, 
Pabell\'on~1, Ciudad Universitaria, 1428, Buenos Aires, Argentina}
\email{eacri@dm.uba.ar}

\email{marco.bonatto.87@gmail.com}

\maketitle

\begin{abstract}
    In this paper we enumerate the skew braces of non-abelian type of size $p^2q$ for $p,q$ primes with $q>2$ by the classification of regular subgroups of the holomorph of the non-abelian groups of the same order. Since Crespo dealt with the case $q=2$, this paper completes the enumeration of skew braces of size $p^2q$ started in a previous work by the authors. In some cases, we provide also a structural description of the skew braces. As an application, we prove a conjecture posed by V. Bardakov, M. Neshchadim and M. Yadav.
\end{abstract}

\section{Introduction}

The set-theoretical Yang--Baxter equation (YBE) was introduced in \cite{MR1183474} as a discrete version of the braid equation. A pair $(X,r)$ where $X$ is a set and $r:X\times X\to X\times X$ is a map, is a \emph{set-theoretical solution to the Yang--Baxter equation} if
\begin{equation}\label{YBE}
(id_X \times r)(r\times id_X)(id_X \times r)=(r\times id_X)(id_X \times r)(r\times id_X)
\end{equation}
holds.
A solution $(X,r)$ is {\it non-degenerate} if the map $r$ is defined as
\begin{equation}\label{non-deg}
r:    X\times X\longrightarrow X\times X,\quad (x,y)\mapsto (\sigma_x(y),\tau_y(x)),
\end{equation}
where $\sigma_x,\tau_x$ are permutations of $X$ for every $x\in X$. Non-degenerate solutions have been studied intensively over the past years \cite{MR1722951,MR1637256, MR1769723,MR1809284}. 

In order to describe \emph{non-degenerate involutive solutions}, that is solutions satisfying $r^2=id_{X\times X}$, Rump introduced the ring-like binary algebraic structures called {\it (left) braces} in \cite{MR2278047}. Later \emph{skew (left) braces} have been defined by Guarnieri and Vendramin in \cite{MR3647970}, to capture also non-involutive solutions.

A \emph{skew (left) brace} is a triple $(B,+,\circ)$ where $(B,+)$ and $(B,\circ)$ are groups (not necessarily abelian) such that
\[
a\circ(b+c)=a\circ b-a+a\circ c
\]
holds for every $a,b,c\in B$. Braces are skew braces with abelian additive group.

The problem of finding non-degenerate solutions to \eqref{YBE} is closely related to the classification problem for skew braces. Indeed in \cite{MR3835326}, a canonical structure of skew brace over a permutation group related to a non-degenerate solution has been described and it has been shown how to recover all the non-degenerate solutions with a given associated skew brace. So, in a sense, the study of non-degenerate solutions to \eqref{YBE} can be reduced to the classification of skew braces.

Some recent results on the classification problem for (skew) braces are, for instance: the classification of braces with cyclic additive group \cite{MR2298848, Rump}, skew braces of size $pq$ for $p,q$ different primes \cite{skew_pq}, braces of size $p^2q$ for $p,q$ primes with $q>p+1$ \cite{Dietzel}, skew braces of size $p^2q$ with cyclic $p$-Sylow subgroup and $p>2$ \cite{Caranti}, skew braces of size $2p^2$ \cite{Crespo_2p2}, braces of order $p^2,p^3$ where $p$ is a prime \cite{p_cube}, skew braces of order $p^3$ \cite{NZ} and skew braces of squarefree size \cite{squarefree}.

This paper is the second part of the enumeration and classification of skew braces of size $p^2q$ where $p,q$ are different primes. In \cite{abelian_case}, we dealt with the left braces. In this paper, we complete the classification focusing on the skew braces with non-abelian additive group.

In \cite{Caranti} the authors partially obtained the same results, through the connection between skew braces and Hopf--Galois extensions explained in \cite{Leandro-Byott} (and they also claim that the missing cases are the subject of a second paper in preparation). The common results of \cite{Caranti} and \cite{Dietzel} agree with ours. In fact, our methods are different but we are covering and extending their results related to the enumeration of skew braces of order $p^2q$.

The main tool in this paper is the algorithm for the construction of skew braces with a given additive group developed in \cite{MR3647970} (we used the same approach in \cite{skew_pq}). We obtain all the skew braces with additive group $B$ from regular subgroups of its holomorph $\Hol(B)=B\rtimes\aut{B}$. The isomorphism classes of skew braces are parametrized by the orbits of such subgroups under the action by conjugation of the automorphism group of $B$ in $\Hol(B)$ \cite[Section 4]{MR3647970}. 

We enumerate the skew braces according to their additive group, and the structure of the paper is displayed in the following table. The classification of the groups of size $p^2q$ can be found in \cite{EnumerationGroups} and the description of their automorphism groups in \cite{auto_pq}.

\begin{center}
	\begin{tabular}{l|c|c}
		& Groups & Sections \\
		\hline
		&&\\[-1em]
		$p= 1 \pmod{q}$    & $\Z_{p^2}\rtimes_{\t} \Z_q$ & \ref{subsection:p=1(q)_non-abelian_cyclic_p-Sylow} \\
		& $\G_k$ & \ref{subsection:p=1(q)_G_k}, \ref{subsection:p=1(q)_G_0}, \ref{subsection:p=1(q)_G_-1}, \ref{subsection:p=1(q)_G_1} \\
		\hline 
		&&\\[-1em]
		$p= -1 \pmod{q}$  & $\G_F$ & \ref{section:p=-1(q)} \\
		\hline
		&&\\[-1em]
		$q = 1\pmod{p}$,   & $\Z_q\rtimes_\r \Z_{p^2}$ & \ref{subsection:q=1(p)_non-abelian_p-Sylow_cyclic} \\
		$q\neq 1\pmod{p^2}$    & $\Z_p\times (\Z_q\rtimes_\r \Z_p)$ & \ref{subsection:q=1(p)_non-abelian_non-cyclic_p-Sylow} \\
		
		\hline
		&&\\[-1em]
		$q= 1\pmod{p^2}$         & $\Z_q\rtimes_{h^p} \Z_{p^2}$ & \ref{subsection:q=1(p^2)_non-abelian_cyclic_p-Sylow1} \\
		& $\Z_p\times (\Z_q\rtimes_{h^p} \Z_p)$ & \ref{subsection:q=1(p^2)_non-abelian_non-cyclic_p-Sylow} \\
		& $\Z_q\rtimes_h \Z_{p^2}$ & \ref{subsection:q=1(p^2)_non-abelian_cyclic_p-Sylow2}
	\end{tabular}
	\bigskip
	
\end{center}

The enumeration results are collected in suitable tables in each section. The skew braces of non-abelian type of size $2p^2$ have been enumerated by Crespo, so we do not include tables for this case and we refer the reader to \cite[Section 5]{Crespo_2p2}. If $q=1\pmod{p}$ and $q\neq 1\pmod{p^2}$ there are $2$ non-abelian groups if $p^2q\neq 12$. If $p^2q=12$, there are $3$ groups. Since the skew braces of size $12$ are included in the GAP library \emph{YangBaxter}, \cite{YBE}, we omit them.

In \cite[Table 5.3]{MR3647970} we have the number of skew braces of order $n\leq 120$ with some exceptions. Due to a computational improvement of the algorithm calculating skew braces, in \cite{skew_trick} we have many tables collecting the number of skew braces of order $n\leq 868$ with some exceptions. We highlight that none of these exceptions are of the form $p^2q$ for different primes $p,q$. All this data is available in \cite{YBE}. This information was of invaluable help when working on the present paper and all of our results agree with those tables.

In \cite[Conjecture 4.1]{skew_trick}, the authors conjecture the number of skew braces of size $p^2q$ based on their tables. In Section \ref{section:conjecture}, as an application of the results obtained in \cite{abelian_case} and in the present work, we can give a positive answer to that conjecture. 
 
\section{Preliminaries}

A \emph{skew (left) brace} is a triple $(B,+,\circ)$ where $(B,+)$ and $(B,\circ)$ are groups and $$a\circ(b+c)=a\circ b-a+a\circ c$$ holds for every $a,b,c\in B$. We say that a skew brace $(B,+,\circ)$ is of $\chi$-type if $(B,+)$ has the group theoretical property $\chi$.

A {\it bi-skew brace} is a skew brace $(B,+,\circ)$ such that $(B,\circ,+)$ is also a skew brace (see \cite{biskew}). Equivalently
\begin{equation*}\label{eq for biskew}
	x+(y\circ z)=(x+ y)\circ x'\circ (x+ z) 
\end{equation*}
holds for every $x,y,z\in B$, where $x'$ denotes the inverse of $x$ in $(B,\circ)$.

Given a skew brace $(B,+,\circ)$, the mapping
\begin{equation*}\label{axiom}
	\lambda:(B,\circ)\to \aut{B,+}, \quad \lambda_a(b)=-a+a\circ b
\end{equation*}
is a homomorphism of groups. 

A subgroup $I$ of $(B,+)$ is a {\it left ideal} of $B$ if $\lambda_a(I)\leq I$ for every $a\in B$. Every left ideal is a subgroup of $(I,\circ)$. In particular, $$\Fix(B)=\setof{a\in B}{\lambda_b(a)=a, \, \text{ for every }b\in B}$$
is a left ideal. If $I$ is also a normal subgroup of both $(B,+)$ and $(B,\circ)$ then $I$ is called an {\it ideal}.

\subsection{Skew braces and regular subgroups}

Given a group $A$ we denote by $\pi_1$ and $\pi_2$ the canonical mappings
\begin{eqnarray*}
\pi_1:\Hol(A)&\longrightarrow& A,\quad (x,f)\mapsto x,\\
\pi_2:\Hol(A)&\longrightarrow& A,\quad (x,f)\mapsto f.
\end{eqnarray*}

In order to classify all skew braces of non-abelian type of size $p^2q$ where $p,q$ are primes, we are going to apply the same ideas as in \cite{abelian_case}, where we obtained a classification of skew braces of abelian type (also called braces).

A subgroup $G$ of $\Hol (A)$ is said to be \emph{regular} if the corresponding set of permutations is a regular subgroup of $\mathbb{S}_A$, the symmetric group over $A$ (i.e. the natural action is transitive and it has trivial stabilizers). We are going to exploit the connection between skew braces with given additive group $A$ and regular subgroups of the holomorph of $A$, as explained in \cite{MR3647970}. This connection is summarized in the following theorem.

\begin{theorem}\cite[Theorem 4.2, Proposition 4.3]{MR3647970}\label{thm:skew_holomorph}
	Let $(A,+)$ be a group. If $\circ$ is an operation such that $(A,+,\circ)$ is a skew brace, then $\{ (a,\lambda_a): a\in A \}$ is a regular subgroup of $\Hol(A,+)$. Conversely, if $G$ is a regular subgroup of $\Hol(A,+)$, then $A$ is a skew brace with $$a\circ b=a+f(b)$$ where $(\pi_1|_G)^{-1}(a)=(a,f)\in G$ and $(A,\circ)\cong G$.
	
	Moreover, isomorphism classes of skew braces over $A$ are in bijective correspondence with the orbits of regular subgroups of $\Hol(A)$ under the action of $\aut{A}$ by conjugation.
\end{theorem}

We can take advantage of the classification of groups of size $p^2q$ for odd primes provided in \cite[Proposition 21.17]{EnumerationGroups} and of size $4q$ in \cite{kohl_4q} and of the description of their automorphism groups given in \cite{auto_pq} to compute the orbits of regular subgroups under the action by conjugation of $\aut{A}$ on $\Hol(A)$. The enumeration of such regular subgroups will provide the enumeration of skew braces of the same size, according to Theorem \ref{thm:skew_holomorph}.

\begin{remark}\label{remark for lambdas}
	Let $(A,+)$ be a group and $G$ a regular subgroup of $\Hol(A)$. According to Theorem \ref{thm:skew_holomorph}, $(A,+,\circ)$ where 
	\begin{equation}\label{circ}
	a\circ b=a+\pi_2((\pi_1|_G)^{-1}(a))(b)
	\end{equation}
	for every $a,b\in A$ is a skew brace. In other words, $\lambda_a=\pi_2((\pi_1|_G)^{-1}(a))$ and 
	so $|\ker{\lambda}|=\frac{|G|}{|\pi_2(G)|}$.
	
As in \cite{abelian_case}, in some cases, we will use \eqref{circ} to explicitly compute the operations of the skew braces.
\end{remark}

We identify an element of the holomorph $(a,f)\in\Hol (A,+)$ with a permutation over the underlying set $A$ acting as $(a,f)\cdot x=a+f(x)$ for every $x\in A$.

Notice that if $(a,f),(b,g)\in G$ then
$$\pi_1(a,f)=\pi_1(b,g)\, \text{ if and only if } \, 
(a,f)^{-1} (b,g)\in H=G\cap (\{1\}\times \aut{A}).$$
Therefore the mapping $\pi_1$ restricted to $G$ factors through the canonical projection onto the set of cosets with respect to the subgroup $H$. If $G$ is finite, we have that $|G|=|H||\pi_1(G)|$ and $|\pi_1(G)|$ divides the size of $G$.

\begin{lemma}\label{rem for regularity}\cite[Lemma 2.5]{abelian_case}
	Let $A$ be a finite group and $G\leq \Hol(A)$. The following are equivalent:
	\begin{itemize}
		\item[(i)] $G$ is regular.
		\item[(ii)] $|A|=|G|$ and $\pi_1(G)=A$.
		\item[(iii)] $|A|=|G|$ and $G\cap \left(\{1\}\times \aut{A}\right)=1$. 
	\end{itemize}
\end{lemma}

\begin{lemma}\label{pi_1 for fix}
	Let $A$ be a group and $G=\langle u_i\alpha_i ,\, i \in I\rangle\leq \Hol(A)$ where $u_i\in A$, $\alpha_i\in \aut{A}$ for $i\in I$. Then: 
	\begin{itemize}
		\item[(1)] $\pi_1(G) \subseteq \langle h(g)\mid \, g\in U,\, h\in\pi_2(G)\rangle$, where $U=\langle u_i,\, i\in I\rangle$.
		\item[(2)] If $\alpha_i(u_i)=u_i$ then $\langle u_i\rangle \subseteq \pi_1(G)$.
	\end{itemize}
\end{lemma}
\begin{proof}
(1)	We have that
	\begin{equation}\label{PP}
	u_i \alpha_i x f=u_i \alpha_i(x) \alpha_i f \quad \text{and} \quad (u_i \alpha_i)^{-1}=\alpha_i^{-1}(u_i)^{-1}\alpha_i^{-1},
	\end{equation}
	for every $i\in I$, $x\in A$ and $f\in \aut{A}$, thus the claim follows by induction.

(2) According to \eqref{PP} we have that
$$(u_i \alpha_i)^n=u_i^n \alpha_i^n$$
and so $\pi_1((u_i \alpha_i)^n)=u_i^n$ for every $n\in \Z$.
\end{proof}

In order to show that a subgroup of the holomorph is regular or not we will make use of Lemma \ref{rem for regularity} and Lemma \ref{pi_1 for fix}. In most of the cases we will omit the explicit computations.

In \cite{abelian_case} we described a technique, inspired by \cite[Section 2.2]{NZpaper}, to obtain the conjugacy classes we are looking for. Let us briefly recall it. 

As a general strategy, we show a list of representatives of regular subgroups and then we prove that any regular subgroup is conjugate to one of the groups in the list. Given such list, the isomorphism class of the representatives of conjugacy classes can be easily obtained using the list of groups of size $p^2q$ and so we omit a proof of such isomorphisms.

To compute the list of representatives, first we need to identify some properties of regular subgroups that are invariant under conjugation by elements of the subgroup $\{1\}\times \aut{A}$.
\begin{itemize}
    \item The map $\pi_2$ is a group homomorphism and, if $G$ is a regular subgroup of $\Hol(A)$, the size of $\pi_2(G)$ divides both $|(A,+)|$ and $|\aut{A,+}|$ and it is invariant under conjugation. So we can compute the regular subgroups according to the size of their image under $\pi_2$. The unique trivial skew brace over $A$ is clearly associated to the unique regular subgroup with trivial image under $\pi_2$ (i.e. $A\times \{1\}$), so we assume that $|\pi_2(G)|>1$.
    
    \item Let us assume that the image under $\pi_2$ has size $k>1$. The conjugacy class of the image under $\pi_2$ in $\aut{A}$ is invariant under conjugation. So we can provide a list of representatives of conjugacy classes of subgroups of $\aut{A}$ of size $k$ and we assume that $\pi_2(G)$ is one of the subgroups in the list.
    
    \item Let $G$ and $G'$ be two regular subgroups such that $\pi_2(G)=\pi_2(G')$. If $hGh^{-1}=G'$ then $h$ is an element of the normalizer of their image under $\pi_2$ and moreover $h \ker{\pi_2|_G} h^{-1}=h (\ker{\pi_2|_G})=\ker{\pi_2|_{G'}}$. 
\end{itemize} 
Sometimes the invariants in the list above do not identify a conjugacy class of regular subgroups in the sense of Theorem \ref{thm:skew_holomorph}, and so we need also to employ some further invariant to provide a list of non-conjugate regular subgroups. Assume that $G$ and $G^\prime$ are regular subgroups of $\Hol(A)$ conjugate by $h\in \aut{A}$, then:
\begin{itemize}
\item the $p$-Sylow subgroups of $G$ are conjugate by $h$ to the $p$-Sylow subgroups of $G'$.
    \item Let $\H$ be a normal subgroup of $\Hol(A)$. The following diagram is commutative:  
\begin{equation}\label{diagram H}
\xymatrix{ \Hol(A) 
	\ar@{>>}[d]  \ar[r]^{\widehat{h}}&  \Hol(A)\ar@{>>}[d] \\
	\Hol(A)/\H \ar[r]^{\widehat{h\H}}& \Hol(A)/\H}
\end{equation}
where $\widehat{h}$, $\widehat{h\H}$ are the inner automorphisms and the unlabeled arrows are the canonical homomorphisms. Then their images in the quotient $\Hol(A)/\H$ are conjugate by $h\H$. So the conjugacy class of the image in the factor $\Hol(A)/\H$ is invariant up to conjugation.
\end{itemize}

Using Lemma \ref{rem for regularity} and the invariants above we can provide a list of regular subgroups for each admissible value of $k$ and claim that they are not conjugate. Usually we start by computing the values of the main invariants we already mentioned, namely:
\begin{itemize}
	\item Find a set of representatives of the conjugacy classes of subgroups of size $k$ of $\aut{A,+}$.
	\item The kernel of $\pi_2$ is a subgroup of $(A,+)$. So we need to find a set of representatives of the orbits of subgroups of size $\frac{|A|}{k}$ of $A$, under the action of the normalizer of the image of $\pi_2$ on $A$.
\end{itemize}

The next step is to prove that such list is actually a set of representatives, namely that any other regular subgroup is conjugate to one of them. First we need to construct the regular subgroups. 

Let $G$ be a regular subgroup. If the set $\setof{\alpha_i}{1\leq i\leq n}$ generates $\pi_2(G)$ and $\setof{k_j}{1\leq j\leq m}$ generates the kernel of $\pi_2|_G$ then $G$ has the following {\it standard presentation}: 
$$G=\langle k_1,\ldots, k_m, u_1\alpha_1,\ldots, u_n\alpha_n \rangle,$$
for some $u_i\in A$. In particular, $u_i\neq 1$, otherwise $\alpha_i\in G\cap (\{1\}\times \aut{A})$ and so $G$ is not regular. We can replace any generator $g$ of $G$ by the product of $g$ with any other generator of $G$, in order to obtain a nicer presentation of the same group $G$. In particular, we can choose any representative of the coset of $u_i$ with respect to the kernel, without changing the group $G$. 

The group $G$ has to satisfy the following necessary conditions, that together provide constraints over the choice of the elements $u_i$: 
\begin{itemize}	
	\item[\NC] The kernel of $\pi_2|_G$ is normal in $G$.	
	\item[\nc] The generators $\setof{u_i\alpha_i}{1\leq i\leq n}$ satisfy the same relations as $\setof{\alpha_i}{1\leq i\leq n}$ modulo $\ker{\pi_2|_G}$ (e.g. if $\alpha_i^n=1$ then $(u_i\alpha_i)^n\in \ker{\pi_2|_G}$). 
\end{itemize}

Given one such group, we can conjugate by the normalizer of $\pi_2(G)$ in $\aut{A}$ that stabilizes the kernel of $\pi_2|_G$ in order to show that $G$ is conjugate to one of the groups in the list of the chosen representatives.

\subsection{Notation} 

If $C$ is a cyclic group acting on a group $G$ by $\rho:C\longrightarrow \aut{G}$ and $\rho(1)=f$, then $G\rtimes_{f} C$ denotes the semidirect product determined by the action $\rho$. In this case the group operation on $G\rtimes_f C$ is 
\begin{eqnarray*}
	\begin{pmatrix} x_1 \\ y_1 \end{pmatrix}  \begin{pmatrix} x_2 \\ y_2 \end{pmatrix} =\begin{pmatrix} x_1 f^{y_1} (x_2) \\ y_1 y_2\end{pmatrix},
\end{eqnarray*}
for every $x_1,x_2\in G$ and $y_1,y_2\in C$.


\begin{remark}\label{subgroups of GL}\cite[Remark 3.5]{abelian_case}
	Let $p,q$ be primes such that $p=1\pmod q$ and let $g$ be an element of order $q$ of $\Z_p^\times$. As in \cite{abelian_case} we denote by $\B$ the subset of $\Z_q$ that contains $0,1,-1$ and one out of $k$ and $k^{-1}$ for $k\neq 0,1,-1$ and by $\dm_{a,b}$ the diagonal matrix with diagonal entries $g^a$ and $g^b$. We are using that, up to conjugation, the subgroups of order $q$ of $GL_2(p)$ are generated by one of the matrices
	$$\dm_{1,s}=\begin{bmatrix} g & 0\\ 0 & g^s \end{bmatrix},\quad \dm_{0,1}=\begin{bmatrix} 1 & 0\\ 0 & g \end{bmatrix},$$
	for $0\leq s\leq q-1$ and a set of representatives of conjugacy classes of such groups in $GL_2(p)$ is given by the groups generated by $\dm_{1,s}$ 
	for $s\in\B$.
	
	Up to conjugation, the unique subgroup of order $p$ of $GL_2(p)$ is generated by
	$$C=\begin{bmatrix} 1& 1\\ 0 & 1 \end{bmatrix}.$$
	
	A set of representatives of conjugacy classes of subgroups of order $pq$ of $GL_2(p)$ is given by $H_s=\langle C,\dm_{1,s}\rangle$ and $\widetilde{H}=\langle C,\dm_{0,1}\rangle$, where $0\leq s\leq q-1$. 
\end{remark}

\section{Skew braces of size $p^2q$ with $p=1\pmod{q}$}\label{section:p=1(q)}

Let $p,q$ be primes such that $p=1\pmod{q}$ and $p,q>2$. Recall that we omit the case $q=2$ since it was covered by Crespo in \cite{Crespo_2p2}. Following the same notation as in \cite{abelian_case} we denote by 
$g$ a fixed element of order $q$ in $\Z_{p}^\times$ and by $\t$ a fixed element of order $q$ in $\Z_{p^2}^\times$. The non-abelian groups of size $p^2q$ are the following:

\begin{itemize}
	\item[(i)] $\mathbb{Z}_{p^2}\rtimes_{\t} \mathbb{Z}_q=\langle \sigma, \tau\ |\ \sigma^{p^2}=\tau^q=1,\ \tau\sigma\tau^{-1}=\sigma^{\t} \rangle$.
	\item[(ii)] $\G_{k}=\langle \sigma,\tau,\epsilon\,|\, \sigma^p=\tau^p=\epsilon^q=1,\, \epsilon \sigma \epsilon^{-1}=\sigma^g,\, \epsilon \tau \epsilon^{-1}=\tau^{g^k}\rangle \cong\mathbb{Z}_{p}^2 \rtimes_{\dm_{1,k}} \mathbb{Z}_q$, 
	for $k\in\B$.
\end{itemize}

Tables \ref{table:1} and \ref{table:2} collect the enumeration of skew braces according to their additive and multiplicative groups, taking into account that the groups $\G_k$ and $\G_{k^{-1}}$ are isomorphic for $k\neq 0$.

Note that for $q=3$ we have $\B=\{0,1,-1\}$ and $2^{-1}=2=-1$ in $\Z_3^\times$. The enumeration of skew braces according to the isomorphism class of their multiplicative group is slightly effected by these facts and so we have a separate table for this case.

\begin{table}[H]
	\centering
	\small{
		\begin{tabular}{c|c|c|c|c|c|c|c|c}
			$+ \backslash \circ$ & $\Z_{p^2q}$ & $\Z_{p^2}\rtimes_{\t}\Z_q$ &$\Z_p^2\times\Z_q$ &   $\G_s$, $s\neq 0,\pm 1, 2$ & $\G_0$ & $\G_{-1}$ & $\G_1$ & $\G_{2}$\\
			\hline
			$\Z_{p^2}\rtimes_{\t}\Z_q$ & $4$ & $2(q-1)$& - &- &- &- &- &- \\
			$\G_k$, $k\neq 0,\pm 1$ & - &-& $4$  & $8(q+1)$ & $8(q+1)$ & $4(q+1)$ & $4(q-1)$ & $8(q+1)$ \\
			$\G_0$ & - &-& $2$  & $4q$ & $4q$ & $2q$ & $2(q-1)$ & $4q$ \\
			$\G_{-1}$ & - &-& $3$ & $4q+p+2$ & $4q+p+2$ & $3q+p-1$ & $2(q-1)$ & $4q+p+2$\\
			$\G_1$ & - &-& $5$  & $3(q+2)$ & $4(q+1)$ & $2(q+1)$ & $3q-1$ & $6q$
	\end{tabular}}
	\vs
	\caption{Enumeration of skew braces of size $p^2q$ according to the additive and multiplicative isomorphism class of groups where $p=1\pmod{q}$ and $q>3$.}\label{table:1}
	
	\small{    \begin{tabular}{c|c|c|c|c|c|c}
			$+ \backslash \circ$ & $\Z_{3p^2}$ & $\Z_{p^2}\rtimes_{\t}\Z_3$ &$\Z_p^2\times\Z_3$ & $\G_0$ & $\G_{-1}$ & $\G_1$ \\
			\hline
			$\Z_{p^2}\rtimes_{\t}\Z_3$ & $4$ & $4$ & - &- &- &- \\
			$\G_0$ & - &- & $2$  & $12$ & $6$ & $4$ \\
			$\G_{-1}$ & - &- & $3$ & $p+14$ & $p+8$ & $4$ \\
			$\G_1$ & - &-& $5$  & $16$ & $10$ & $8$ 
	\end{tabular}}
	\vs
	\caption{Enumeration of skew braces of size $3p^2$ according to the additive and multiplicative isomorphism class of groups where $p=1\pmod{3}$.
	}
\label{table:2}
\end{table}

\subsection{Skew braces of $\Z_{p^2}\rtimes_{\t} \Z_q$-type}\label{subsection:p=1(q)_non-abelian_cyclic_p-Sylow}

In this section we denote by $A$ the group $\Z_{p^2}\rtimes_{\t} \Z_q$. According to the description of the automorphism group of groups of order $p^2q$ given in \cite[Theorem 3.4]{auto_pq}, we have that 
$$\phi:\Hol(\Z_{p^2})=\Z_{p^2}\rtimes \Z_{p^2}^\times\longrightarrow \aut{A},\quad (i,j) \mapsto\varphi_{i,j}=\begin{cases} \tau\mapsto \sigma^i\tau,\\ \sigma \mapsto \sigma^j\end{cases}$$
is an isomorphism of groups.
In particular $|\aut{A}|=p^3(p-1)$. Thus if $G$ is a regular subgroup of $\Hol(A)$ then $|\pi_2(G)|\in \{p,q,p^2,pq,p^2q\}$. Let us show that $|\pi_2(G)|$ can not equal $p$.


\begin{lemma}
	Let $G$ be a regular subgroup of $\Hol(A)$. Then $|\pi_2(G)|\neq p$.
\end{lemma}
%
%
\begin{proof}
	Assume that $G$ is a subgroup of size $p^2q$ of $\Hol(A)$ and $|\pi_2(G)|=p$. Up to conjugation, we can 
	choose the generators of $\ker{\pi_2|_G}$ to be $\tau$ and $\sigma^p$. Hence the standard presentation of $G$ is
	$$G=\langle \tau, \sigma^p, \sigma^a \alpha\rangle$$
	where $\alpha\in \aut{A}$ is an automorphism of order $p$ that satisfies $\alpha(\tau)=\sigma^{pb}\tau$ for some $0\leq b\leq p-1$. By condition \NC we have that
	$$\sigma^a \alpha \tau \alpha^{-1}\sigma^{-a}=\sigma^{a(1-\t)+pb}\tau \in \langle \sigma^p,\tau\rangle.$$
	Then $a=0\pmod{p}$, i.e. $a=a^\prime p$. Hence $(\sigma^p)^{-a^\prime}\sigma^{a^\prime p}\alpha=\alpha\in G\cap (\{1\}\times \aut{A})$ and then $G$ is not regular by Lemma \ref{rem for regularity}.
\end{proof}

\begin{lemma}\label{case pq}
	There exists a unique conjugacy class of regular subgroups $G$ of $\Hol(A)$ with $|\pi_2(G)|=pq$. A representative is
	$$H=\langle \sigma^p,\ \sigma\varphi_{0,p+1},\ \tau^{-1}\varphi_{0,\t}\rangle\cong \mathbb{Z}_{p^2q}.$$
\end{lemma}

\begin{proof}
	According to Lemma \ref{pi_1 for fix}(2), $\langle \sigma^p,\tau\rangle \subseteq \pi_1(H)$ and clearly $\sigma\in \pi_1(H)$. So $|\pi_1(H)|>pq$ and so $|\pi_1(H)|=|H|=p^2q$. According to Lemma \ref{rem for regularity}, $H$ is regular.
		
	Let $G$ be a regular subgroup of $\Hol(A)$ with $|\pi_2(G)|=pq$. Then $K=\ker{\pi_2|_G}=\langle \sigma^p \rangle$. Up to conjugation, the subgroups of size $pq$ of $\aut{A}$ are $\langle \varphi_{p,1},\ \varphi_{0,\t}\rangle$ and $\langle \varphi_{0,\t},\ \varphi_{0,p+1}\rangle$.
		
	In the first case, we have
	$$G=\langle \sigma^p,\ \sigma^a\tau^b\varphi_{p,1},\ \sigma^c\tau^d\varphi_{0,\t}\rangle.$$
	We need to check the conditions \nc. The generators of $\pi_2(G)$ satisfy the relations $(\varphi_{p,1})^p=1$ and $\varphi_{0,t}\varphi_{p,1}\varphi_{0,\t}^{-1}=\varphi_{p,1}^{\t}$. Thus we have
		\begin{equation*}
	(\sigma^a\tau^b\varphi_{p,1})^p K =\sigma^{a\sum_{j=0}^{p-1}t^{bj}}\tau^{bp} K \overset{\nc}{=} K
	\end{equation*}
and so $b=0$. Accordingly, $d\neq 0$, otherwise by Lemma \ref{pi_1 for fix}(1) we have that $\pi_1(G)\subseteq \langle \sigma\rangle$. Moreover	
\begin{equation*}(\sigma^c\tau^d\varphi_{0,\t}) \sigma^a\varphi_{p,1}\left(\sigma^c\tau^d\varphi_{0,\t}\right)^{-1} K = \sigma^{\t^{d+1}a}\varphi_{p,1}^{\t} K\overset{\nc}{=}(\sigma^a\varphi_{p,1})^t K=\sigma^{\t a} \varphi_{p,1}^t K
	\end{equation*}
	and then $a=pa^\prime$. Therefore, $\sigma^{-a^\prime p}\sigma^{a^\prime p}\varphi_{p,1}=\varphi_{p,1}\in G$ and so $G$ is not regular, contradiction.
		
	Let $\pi_2(G)=\langle \varphi_{0,\t},\ \varphi_{0,p+1}\rangle$. Then,
	$$G=\langle \sigma^p,\ \sigma^a\tau^b\varphi_{0,p+1},\ \sigma^c\tau^d\varphi_{0,\t}\rangle.$$
	The relations satisfied by the generators of $\pi_2(G)$ are $\varphi_{0,\t}^q=\ \varphi_{0,p+1}^p=[\varphi_{0,t},\varphi_{0,p+1}]=1$. By the conditions \nc we have that $\sigma^a\tau^b\varphi_{0,p+1}$ and $\sigma^c\tau^d\varphi_{0,\t}$ satisfy the same relations modulo $K$. Accordingly, since
	\begin{eqnarray*}
	(\sigma^a\tau^b\varphi_{0,p+1})^p K&=& \sigma^{a\sum_{j=0}^{p-1}t^{bj}} \tau^{bp}K\overset{\nc}{=}K,\\
	(\sigma^a\varphi_{0,p+1})\sigma^c\tau^d\varphi_{0,\t}K&=&\sigma^{a+c}\tau^{d} \varphi_{0,p+1}\varphi_{0,\t}K\overset{\nc }{=} (\sigma^c\tau^d\varphi_{0,\t})\sigma^a\varphi_{0,p+1} K=\sigma^{t^{d+1}a+c}\tau^{d}\varphi_{0,p+1}\varphi_{0,\t}K,
	\end{eqnarray*}
	then $b=0$ and $d=-1$. Finally,
$$		(\sigma^c\tau^{-1}\varphi_{0,\t})^q K = \sigma^{qc}K \overset{\nc}{=} K$$	and so $c=0\pmod{p}$. If $a= 0 \pmod{p}$, then according to Lemma \ref{pi_1 for fix}(2), $\pi_1(G)\subseteq \langle \sigma^p, \tau\rangle$ and so $a\neq 0\pmod {p}$. Then $G$ is conjugate to $H$ by $\varphi_{0,a^{-1}}$.
\end{proof}

\begin{lemma}
	A set of representatives of regular subgroups $G$ of $\Hol(A)$ with $|\pi_2(G)|=p^2$ is
	$$G_{s}=\langle \tau,\ \sigma^{\frac{1}{\t-1}}\varphi_{1,(p+1)^s}\rangle\cong \mathbb{Z}_{p^2q}$$
	for $s=0,1$. 
\end{lemma}

\begin{proof}
The groups $G_0$ and $G_1$ are not conjugate since their images under $\pi_2$ are not. Applying Lemma \ref{rem for regularity} and Lemma \ref{pi_1 for fix} as we did in Lemma \ref{case pq}, we can show that the groups $G_s$ are regular. Indeed, $(\sigma^{\frac{1}{\t-1}}\varphi_{1,(p+1)^s})^p=\sigma^{\frac{p}{\t-1}}\varphi_{p,1}$ and $\varphi_{p,1}(\sigma^\frac{p}{\t-1})=\sigma^{\frac{p}{\t-1}}$. So we have that $\langle \tau, \sigma^p\rangle\subseteq \pi_1(G_s)$ and $\sigma^{\frac{1}{\t-1}}\in \pi_1(G_s)$. Thus $|\pi_1(H)|=p^2q$ and so $H$ is regular. 
	
	Let $G$ be a regular subgroup of $\Hol(A)$ with $|\pi_2(G)|=p^2$. According to \cite[Lemma 2.1]{abelian_case}, the $p$-Sylow subgroup of the multiplicative group of the skew brace associated to $G$ is cyclic and then so is $\pi_2(G)$. The unique cyclic subgroups of order $p^2$ of $\aut{A}$ up to conjugation are $\langle \varphi_{1,(p+1)^s}\rangle$ for $s=0,1$. The size of the kernel of $\pi_2$ is $q$ and so, in both cases,  we can assume that $\ker{\pi_2|_G}$ is generated by $\tau$ up to conjugation by an automorphism in the normalizer of $\pi_2(G)$. Therefore
	$$G=\langle \tau,\ \sigma^b\varphi_{1,(p+1)^s}\rangle$$
	for some $b\neq 0$. The group $G$ has a normal $q$-Sylow subgroup and then it is abelian, and thus $b=\frac{1}{\t-1}$. 
\end{proof}
The group 
\begin{equation}\label{other invariant 1}
\H_1=\langle \sigma, \varphi_{1,1}\rangle    
\end{equation}
is normal in $\Hol(A)$. Let $u$ be a generator of $\Z_{p^2}^\times$, then $\Hol(A)/\H_1=\langle \tau,\,\varphi_{0,u}\rangle\cong \Z_q\times \Z_{p^2}^\times$ and in particular it is abelian.

\begin{lemma}\label{Zp2 rtimes Z_q q}
	The skew braces of $A$-type with $|\ker{\lambda}|=p^2$ are $B_s=(A,+,\circ)$ where $(B_s,+)=\Z_{p^2}\rtimes_{\t} \Z_q$ and $(B_s,\circ)= \Z_{p^2}\rtimes_{\t^{\frac{s+1}{s}}} \Z_q$ for $1\leq s\leq q-1$. In particular, $B_s$ is a bi-skew brace and
	$$(B_s,\circ)\cong \begin{cases} \mathbb{Z}_{p^2q},\, \text{ if } s={q-1},\\
	A,\, \text{ otherwise.}\end{cases}$$
\end{lemma}

\begin{proof}
	Let us consider the groups
	$$G_s=\langle \sigma,\ \tau^s\varphi_{0,\t}\rangle\cong \begin{cases} \Z_{p^2q}, \, \text{ if } s=q-1,\\ A, \, \text{otherwise} \end{cases}$$
	 for $1\leq s\leq q-1$. The subset $\pi_1(G_s)$ contains $\langle \sigma\rangle$ and $\tau^s$ and then $|\pi_1(G)|>p^2$ and it divides $p^2q$. So $\pi_1(G)=A$ and according to Lemma \ref{rem for regularity} we have that $G_s$ is regular. Let $\H_1$ be the group defined in \eqref{other invariant 1}. If $G_s$ and $G_{s'}$ are conjugate, then so are their images in $\Hol(A)/\H_1$ which is abelian. Therefore $\langle \tau^s \varphi_{0,t}\rangle=\langle \tau^{s^\prime} \varphi_{0,t}\rangle$, and so it follows that $s=s^\prime.$ 
		
	Let $G$ be a regular subgroup of $\Hol(A)$ with $|\pi_2(G)|=q$. The unique subgroup of order $p^2$ of $A$ is the unique $p$-Sylow subgroup generated by $ \sigma$ and, up to conjugation, the unique subgroup of order $q$ in $\aut{A}$ is generated by $\varphi_{0,\t}$. Then 
	$$G=\langle \sigma, \tau^s \varphi_{0,\t}\rangle=G_s$$ for some $s\neq 0$.
	
	For the skew brace $B_s$ associated to $G_s$ we have that $\langle\sigma\rangle\leq \ker{\lambda}$ and $\tau\in \Fix(B_s)$. So, using \cite[Lemma 2.2]{abelian_case} we have that $\lambda_{\sigma^n \tau^m}=\lambda_{\sigma^n \tau^{\frac{sm}{s}}}=\lambda_{\tau^s}^m$ and then the multiplicative group structure of $B_s$ is given by the formula 
	\begin{eqnarray*}
		\begin{pmatrix} x_1 \\ x_2 \end{pmatrix} \circ \begin{pmatrix} y_1 \\ y_2 \end{pmatrix} 
		=\begin{pmatrix} x_1+\t^{\frac{s+1}{s}x_2}y_1\\
			x_2+y_2 \end{pmatrix}
	\end{eqnarray*}
	for every $0\leq x_1,y_1\leq p^2-1$, $0\leq x_2,y_2\leq q-1$. So $(B_s,+)=\Z_{p^2}\rtimes_{\t}\Z_q$ and $(B_s,\circ)=\Z_{p^2}\rtimes_{\t^{\frac{s+1}{s}}}\Z_q$ and according to \cite[Corollary 1.2]{skew_pq} $B_s$ is a bi-skew brace. To compute the isomorphism class of $(B_s,\circ)$ note that the group is abelian if and only if $s=q-1$.
\end{proof}

\begin{lemma}
	A set of representatives of conjugacy classes of regular subgroups $G$ of $\Hol(A)$ with $|\pi_2(G)|=p^2q$
	is
	$$G_{d}=\langle \sigma^{\frac{1}{t-1}}\varphi_{1,1},\ \tau^d\varphi_{0,\t}\rangle\cong A$$
	where $1\leq d\leq q-1$. 
\end{lemma}

\begin{proof}
	According to Lemma \ref{pi_1 for fix}(2), $\langle \sigma,\tau \rangle\subseteq \pi_1(G_d)$ and so $\pi_1(G_d)=A$. The same argument used in Lemma \ref{Zp2 rtimes Z_q q} shows that they are not pairwise conjugate.
		
	Let $G$ be a regular subgroup with $|\pi_2(G)|=p^2q$. According to \cite[Lemma 2.1]{abelian_case}, the $p$-Sylow subgroup of the multiplicative group of the skew brace associated to $G$ is cyclic, and then so is the $p$-Sylow subgroup of $\pi_2(G)$. 
	Up to conjugation we can assume that $\pi_2(G)$ is $\langle \varphi_{1,1},\ \varphi_{0,\t}\rangle\cong A$, 
	i.e.
	$$G=\langle \sigma^a\tau^b\varphi_{1,1},\ \sigma^c\tau^d\varphi_{0,\t}\rangle.$$
	From the conditions \nc we have
	\begin{align}
	    (\sigma^a\tau^b\varphi_{1,1})^{p^2}&= 1 \label{p2q case eq 1}, \\
	    (\sigma^c\tau^d\varphi_{0,\t})^q &=1 \label{p2q case eq 1 bis}, \\
	  (\sigma^c\tau^d\varphi_{0,\t})  \sigma^a\tau^b\varphi_{1,1}(\sigma^c\tau^d\varphi_{0,\t})^{-1}&=(  \sigma^a\tau^b\varphi_{1,1})^\t.\label{p2q case eq 2}
	\end{align}
	The equality in \eqref{p2q case eq 1} implies that the same relation is true modulo $\H_1$ and so $b=0$. If $d=0$, then $G$ is not regular by Lemma \ref{pi_1 for fix}(1). Accordingly $d\neq 0$ and from \eqref{p2q case eq 2} we have $a=\frac{1}{\t-1}$ and so the standard presentation of $G$ is 
	$$G=\langle \sigma^{\frac{1}{\t-1}}\varphi_{1,1},\ \sigma^c\tau^d\varphi_{0,\t}\rangle.$$ 

	%
	If $d=-1$ then from \eqref{p2q case eq 1 bis} it follows that $c=0$. Otherwise, $G$ is conjugate to $G_d$ by $\varphi_{1,1}^n$ where $n=\frac{c(\t-1)}{1-\t^{d+1}}$.
\end{proof}

We summarize the content of this subsection in the following table:
\begin{table}[H]
	\centering
	\small{
		\begin{tabular}{c|c|c}
			$|\ker{\lambda}|$ &  $\mathbb{Z}_{p^2q}$  &  $\Z_{p^2}\rtimes_{\t}\Z_q$  \\
			\hline
			$1$ &- & $q-1$ \\
			$p$  & $1$ &- \\
			$q$ & $2$ &- \\
			$p^2$ & $1$ & $q-2$ \\
			$p^2q$ &- & $1$
	\end{tabular}}
	\vs
	\caption{Enumeration of skew braces of $\Z_{p^2}\rtimes_{\t}\Z_q$-type for $p=1\pmod{q}$.}
	\label{table:3.3}
\end{table}

\subsection{Skew braces of $\G_k$-type for $k\neq 0,\pm 1$}\label{subsection:p=1(q)_G_k}

In this section we assume that $k\in \B\setminus\{0,1,-1\}$ and accordingly $q>3$.
Recall that a presentation of the group $\G_k$ is the following
\begin{equation}\label{presentation_G_k}
\G_k=\langle \sigma,\tau,\epsilon\ \big\lvert\ \, \sigma^p=\tau^p=\epsilon^q=[\sigma,\tau]=1,\ \epsilon \sigma \epsilon^{-1}=\sigma^g, \,\epsilon \tau \epsilon^{-1}=\tau^{g^k}\rangle  
\end{equation}
where $g$ is a fixed element of order $q$ in $\Z_p^\times$ as before.
An automorphism of $\G_k$ is determined by its image on the generators, i.e. by its restriction to $\langle\sigma, \tau\rangle$ given by a matrix and by the image on $\epsilon$. 
According to \cite[Subsections 4.1, 4.3]{auto_pq}, the mapping
\begin{eqnarray*}
	\phi:\mathbb{Z}_p^2 \rtimes_{\rho} \left(\mathbb{Z}_{p}^\times\times \mathbb{Z}_p^\times\right) & \longrightarrow & \aut{\G_k}, \\
	\left[(n,m),(a,b)\right] & \mapsto &	h = 	\begin{cases}
		h|_{\langle\sigma,\tau\rangle}=\begin{bmatrix}
			a& 0\\
			0 & b
		\end{bmatrix}
		,\\ \epsilon\mapsto   \sigma^n\tau^m \epsilon   ,    
	\end{cases}
\end{eqnarray*}
is a group isomorphism (the action $\rho$ is defined by setting $\rho(a,b)(n,m)=(an,bm)$).	In particular, $|\aut{\G_{k}}|=	p^2(p-1)^2$ and the unique $p$-Sylow subgroup of $\aut{\G_k}$ is generated by $\alpha_1=[(1,0),(1,1)]$ and $\alpha_2=[(0,1),(1,1)]$. 

Let $G$ be a regular subgroup of $\Hol(\G_k)$. Since $p^2q$ divides $|\aut{G_k}|$ we need to discuss all possible cases for the size of $\pi_2(G)$.

In the group $\Hol(\G_k)$ we have that
\begin{eqnarray}
(\sigma^a \tau^b \alpha_1)\epsilon(\sigma^a\tau^b \alpha_1)^{-1} &=& \sigma^{1+(1-g)a}\tau^ {(1-g^k)
b}\epsilon,\label{formulas for G_k}\\
(\sigma^c \tau^d \alpha_2)\epsilon(\sigma^c\tau^d\epsilon \alpha_2)^{-1} &=& \sigma^{(1-g)c}\tau^ {1+(1-g^k)d}\epsilon\label{formulas for G_k 2},
\end{eqnarray}
hold for every $0\leq a,b,c,d\leq p-1$.

\begin{lemma}\label{G_k_sub_p}
	A set of representative regular subgroups $G$ of $\Hol(\G_k)$ with $|\pi_2(G)|=p$ is
	$$H_i=\langle \epsilon, \tau, \sigma^{\frac{1}{g-1}}\alpha_1\alpha_2^i\rangle\cong\G_0,\quad K_i=\langle \epsilon, \sigma, \tau^{\frac{1}{g^k-1}}\alpha_1^i\alpha_2\rangle\cong\G_0,$$
	for $i=0,1$.
\end{lemma}

\begin{proof}
	Up to conjugation, the subgroups of order $p$ of $\aut{\G_k}$ are $\langle \alpha_1\rangle, \langle \alpha_2\rangle$ and $\langle\alpha_1\alpha_2\rangle$. The groups in the statement are not conjugate since either their images or their kernels with respect to $\pi_2$ are not. Each of them has the form
	$$G=\langle \epsilon, u, v \theta\rangle$$
	with $\theta(v)=v$ and $\G_k=\langle \epsilon, u,v\rangle$. According to Lemma \ref{pi_1 for fix}(2), $\pi_1(G)=\G_k$, i.e. $G$ is regular.
	
	Let $G$ be a regular subgroup of $\Hol(\G_k)$ with $|\pi_2(G)|=p$. The kernel of $\pi_2$ has size $pq$ and therefore it has an element of order $q$ of the form $u\epsilon$ where $u\in \langle \sigma,\tau\rangle$. The subgroup $L=\langle \alpha_1,\alpha_2\rangle$ normalizes $\pi_2(G)$, so we can conjugate $G$ by a suitable element of $L$ and according to \eqref{formulas for G_k} and \eqref{formulas for G_k 2} we can assume that $u=0$. Then the kernel has the form $\langle \epsilon,\ v\rangle$ for $v\in \langle \sigma,\tau\rangle$. The $p$-Sylow subgroup of $\ker{\pi_2}$ is normal and therefore $v=\sigma$ or $\tau$. Therefore the group $G$ has the following form
	$$G=\langle \epsilon, v, w \theta\rangle$$
	where $\theta\in\{\alpha_1,\alpha_2,\alpha_1\alpha_2\}$ and, either $v=\sigma$ and $1\neq w\in \langle \tau\rangle$ or $v=\tau$ and $1\neq w\in \langle\sigma \rangle$. By condition \NC we need to check that
	$$w\theta \epsilon \theta^{-1} w^{-1}=w\theta(\epsilon) w^{-1}\in \ker{\pi_2|_G}.$$
	 Using \eqref{formulas for G_k} and \eqref{formulas for G_k 2} it follows that $G$ is either $H_i$ or $K_i$ for $i=0,1$. 
\end{proof}
\begin{lemma}\label{G_k_sub_pp}
	The unique skew brace of $\G_k$-type with $|\ker{\lambda}|=q$ is $(B,+,\circ)$ where $(B,+)=\Z_p^2\rtimes_{\dm_{1,k}} \Z_q$ and
		\begin{eqnarray*} 
		\begin{pmatrix} x_1 \\ x_2 \\ x_3 \end{pmatrix} \circ \begin{pmatrix} y_1 \\ y_2\\ y_3 \end{pmatrix} 
		=\begin{pmatrix}x_1 g^{y_3}+y_1 g^{x_3}\\  x_2 g^{k y_3}+y_2 g^{kx_3} \\ x_3+y_3\end{pmatrix}.
	\end{eqnarray*}
	for every $0\leq x_1,x_2,y_1,y_2\leq p-1$ and $0\leq x_3,y_3\leq q-1$. In particular, $(B,\circ)\cong \Z_p^2\times \Z_q$. 
\end{lemma}

\begin{proof}
	According to \eqref{formulas for G_k} and \eqref{formulas for G_k 2}, the subgroup 
	$$H=\langle \epsilon,\,\sigma^{\frac{1}{g-1}}\alpha_1,\,\tau^{\frac{1}{g^k-1}}\alpha_2\rangle$$
	is isomorphic to $\Z_p^2\times \Z_q$ and $|\pi_2(G)|=p^2$. According to Lemma \ref{pi_1 for fix}(2), $\pi_1(H)=\langle \epsilon,\sigma,\tau\rangle$ and so $H$ is regular.
	
	Let $G$ be a regular subgroup with $|\pi_2(G)|=p^2$. Then the image of $\pi_2$ is the unique $p$-Sylow subgroup of $\aut{\G_k}$. The kernel is a subgroup of order $q$ of $G$, and, according to \eqref{formulas for G_k} and \eqref{formulas for G_k 2}, we can assume that it is generated by $\epsilon$. The group $G$ has a standard presentation as
	$$G=\langle \epsilon,\ \sigma^a\tau^b\alpha_1,\ \sigma^c\tau^d\alpha_2\rangle$$
	and it is abelian, since it has a normal $q$-Sylow subgroup. Using \eqref{formulas for G_k} and \eqref{formulas for G_k 2} again, we have that $b=c=0$, $a=\frac{1}{g-1}$ and $d=\frac{1}{g^k-1}$, i.e. $G=H$.
		
	Let $B$ be the skew brace associated to $H$. Since $\epsilon\in \ker{\lambda}$ and $\sigma,\tau\in \Fix(B)$, by \cite[Lemma 2.1]{abelian_case} we can compute the formula for the $\circ$ operation of $B$ as in the statement.
\end{proof}

The subgroup
\begin{equation}\label{other invariant 2} 
\H_2=\langle \sigma,\ \tau,\ \alpha_1,\ \alpha_2\rangle   \cong \Z_p^4
\end{equation}
is normal in $\Hol(\G_k)$. Let $\mu$ be a generator of $\Z_p^\times$, then $\Hol(\G_k)/\H_2=\langle \epsilon, [(0,0),(1,\mu)], [(0,0),(\mu,1)] \rangle  \cong \Z_q\times \Z_p^\times \times \Z_p^\times$. In particular, it is abelian. 


The formula
\begin{equation}\label{conj by alphas}
    \alpha_1^n \alpha_2^m u\epsilon^d \theta \alpha_2^{-m} \alpha_1^{-n}=u\sigma^{x n\frac{g^d-1}{g-1}}\alpha_1^{(1-x)n}\tau^{y  m\frac{g^{kd}-1}{g^k-1}}\alpha_2^{(1-y)m}\epsilon^d \theta,\tag{$F_k$}
\end{equation}
where $u\in \langle \sigma,\tau\rangle$ and $\theta=[(0,0),(x,y)]$ holds in $\Hol(\G_k)$. 

Let us denote $\beta_s=[(0,0),(g,g^s)]$ and $\widetilde\beta=[(0,0),(1,g)]$ for $0\leq s\leq q-1$.

\begin{proposition}\label{G_k_sub_q}
	The skew braces of $\G_k$-type with $|\ker{\lambda}|=p^2$ are $(B_{a,b},+,\circ)$ where $(B_{a,b},+)=\Z_p^2  \rtimes_{\dm_{1,k}} \Z_q$ and
	$$(B_{a,b},\circ)=\Z_p^2\rtimes_{\dm_{a+1,b+k}} \Z_q\cong \begin{cases}
	\Z_p^2\times \Z_q, \, \text{ if } a=q-1, \, b=-k \pmod{q},\\
	\G_0,\, \text{ if } a=q-1, \, b\neq -k  \pmod{q},\\
	\G_{\frac{b+k}{a+1}},\, \text{otherwise, }\\
	\end{cases}$$
	for $0\leq a,b\leq q-1$ and $(a,b)\neq (0,0)$. In particular, there are $q^2-1$ such skew braces and they are bi-skew braces. 
\end{proposition}

\begin{proof}
	The subgroups $$G_{a,b}=\langle \sigma,\ \tau,\ \epsilon \beta_0^a \widetilde \beta^b\rangle$$ are regular.
	If $G_{a,b}$ and $G_{c,d}$ are conjugate, then so are their images in the quotient $\Hol(\G_k)/\H_2$, which is abelian. Then $\langle \epsilon \beta_0^a\widetilde \beta^b\rangle=\langle \epsilon \beta_0^c\widetilde \beta^d\rangle$ and so it follows that $(a,b)=(c,d)$.
		
	Let $G$ be a regular subgroup of $\Hol(\G_k)$ such that $|\pi_2(G)|=q$. Then the kernel of $\pi_2$ is the unique $p$-Sylow subgroup of $\G_k$. The elements of order $q$ of $\aut{\G_k}$ are of the form $\alpha_1^s\alpha_2^t \beta_0^a \widetilde \beta^b$ where $s=0$ (resp. $t=0$) whenever $a=0$  (resp. $b=0$). According to formula \eqref{conj by alphas} for $u=1$ and $d=0$, up to conjugation by an element in $\alpha_1\alpha_2$, we can assume that $\pi_2(G)$ is generated by $\beta_0^{a}\widetilde \beta^{b}$ where $(a,b)\neq (0,0)$. Therefore
	$$G=\langle \sigma,\ \tau,\ \epsilon^n \beta_0^{a}\widetilde \beta^{b}\rangle$$
and so $n\neq 0$ since $G$ is regular. Using that $(\epsilon^n\beta_0^a\widetilde \beta^b)^{n^{-1}}=\epsilon \beta_0^{an^{-1}}\widetilde \beta^{bn^{-1}}$ we can conclude that $G=G_{an^{-1},bn^{-1}}$.
		
	Let $B_{a,b}$ be the skew brace associated to the regular subgroup $G_{a,b}$. Note that $\epsilon\in \Fix(B_{a,b})$ and $\sigma,\tau\in \ker{\lambda}$. Therefore we can apply \cite[Lemma 2.2]{abelian_case} and it follows that
	\begin{eqnarray}\label{formula_sub_G_k_q}
	\begin{pmatrix} x_1 \\ x_2 \\x_3\end{pmatrix} \circ \begin{pmatrix} y_1 \\ y_2 \\y_3 \end{pmatrix} 
	=\begin{pmatrix} x_1+g^{(a+1)x_3}y_1 \\ x_2+g^{(b+k)x_3}y_2\\ x_3+y_3 \end{pmatrix}
	\end{eqnarray}
	for every $0\leq x_1,x_2,y_1,y_2\leq p-1,\, 0\leq x_3,y_3\leq q-1$. So $(B_{a,b},+)= \Z_p^2\rtimes_{\dm_{1,k}} \Z_q$, $(B_{a,b},\circ)=\Z_p^2\rtimes_{\dm_{a+1,b+k}} \Z_q$ and the images of the two actions commute. According to \cite[Proposition 1.1]{skew_pq}, $B_{a,b}$ is a bi-skew brace. 
\end{proof}

To compute the regular subgroups of $\Hol(\G_k)$ with image under $\pi_2$ of size $pq$ and $p^2q$ we are computing the conjugacy classes of subgroups of size $pq$ and $p^2q$ of $\aut{\G_k}$.

\begin{lemma}\label{pi_2 pq G_k}
	A set of representatives of conjugacy classes of subgroups of $\aut{\G_k}$ of order $pq$ and order $p^2q$ are
	\begin{center}
		\small{
			\begin{tabular}{c|l|c|l}\label{table for case pq G_k} 
				Size &    $G$ & Parameters & Class \\
				\hline
				&&&\\[-1em]
				$pq$    & $\mathcal H_{1,s}=\langle \alpha_1, \beta_s\rangle$  & $0\leq s\leq q-1$ & $\mathbb{Z}_p\rtimes_{\t} \mathbb{Z}_q$\\
				\cline{2-4}
				&&&\\[-1em]
				&   $\mathcal H_{2,s}=\langle \alpha_2, \beta_s\rangle$  &  $0\leq s\leq q-1$ & $\Z_{pq}$, if $s=0$,\\  
				&&&\\[-1em]
				& & & $\Z_p\rtimes_{\t}\Z_q$, otherwise. \\
				\cline{2-4}
				&&&\\[-1em]
				&  $\mathcal K_i=\langle \alpha_i, \widetilde\beta\rangle$  &  $i=1,2$ & $\Z_{pq}$, if $i=1$,\\
				&&&\\[-1em]
				& & & $\Z_p\rtimes_{\t}\Z_q$, if $i=2$. \\
				\cline{2-4}
				&&&\\[-1em]
				&   $\mathcal W=\langle \alpha_1\alpha_2, \beta_1\rangle$  & - & $\Z_p\rtimes_{\t} \Z_q$\\
				\hline
				&&&\\[-1em]
				$p^2q$ &    $\mathcal T_s=\langle \alpha_1,\alpha_2,\beta_s\rangle$ &  $0\leq s\leq q-1$ & $\G_s$ \\
				\cline{2-4}
				&&&\\[-1em]
				& $\mathcal U=\langle \alpha_1,\alpha_2,\widetilde\beta\rangle$   & - & $\G_0$
		\end{tabular}}
	\end{center}
\end{lemma}
\begin{proof}
	Let $G$ be a subgroup of size $pq$ of $\aut{\G_k}$. Then $G$ is generated by an element of order $p$ and an element of order $q$. The elements of order $p$ belong to the subgroup generated by $\alpha_1$ and $\alpha_2$ and up to conjugation there are three elements of order $p$, namely $ \alpha_1,\alpha_2,\alpha_1\alpha_2$. The elements of order $q$ can be chosen of the form $\alpha_1^n\alpha_2^m\beta_s$ for $s\neq 0$, $\alpha_1^n\beta_0$ or  $\alpha_2^m \widetilde\beta$.
	
	If the $p$-Sylow subgroup of $G$ is generated by $\alpha_i$, $i=1,2$, using that
	\begin{align}
		[(a,b),(x,y)]\alpha_1^n\alpha_2^m\beta_s[(a,b),(x,y)]^{-1}&=\alpha_1^{xn+(1-g)a}\alpha_2^{ym+(1-g^s)b}\beta_s \label{conjugation in aut G_k},\\
		[(a,b),(x,y)]\alpha_2^m\widetilde\beta [(a,b),(x,y)]^{-1}&=\alpha_2^{ym+(1-g)b}\widetilde\beta ,\notag
	\end{align}
	we have that $G$ is conjugate to $\mathcal H_{i,s}$ or to $\mathcal K_i$. If the $p$-Sylow subgroup of $G$ is generated by $\alpha_1\alpha_2$ then necessarily the element of order $q$ has form $\alpha_1^n\alpha_2^m\beta_1$, since the $p$-Sylow subgroup has to be normal. In this case, then $G=\langle \alpha_1\alpha_2,\ \alpha_1^n\alpha_2^m\beta_1\rangle=\langle \alpha_1\alpha_2,\ \alpha_1^{n-m}\beta_1\rangle$ and according to \eqref{conjugation in aut G_k}, $G$ is conjugate to $\mathcal W$ by a suitable power of $\alpha_1$.
	
	If $G$ is a subgroup of order $p^2q$, then $G$ is generated by $\alpha_1,\alpha_2$ and an element of order $q$, which can be chosen to be $\beta_s$ for $0\leq s\leq q-1$ or $\widetilde\beta$. Such groups are not conjugate, since their restrictions to $\langle \sigma,\tau\rangle$ are not.
\end{proof}

In the following we are employing the same notation as in Lemma \ref{pi_2 pq G_k} and the following formula:
\begin{equation}\label{formula for powers}
    (\sigma^a \tau^b \epsilon^c\theta)^n=\sigma^{a\sum_{i=0}^{n-1} g^{i(c+x)}}\tau^{a\sum_{i=0}^{n-1} g^{i(ck+y)}}\epsilon^{nc}\theta^n\tag{$P_k$}
\end{equation}
where $\theta=[(0,0),(g^x,g^y)]$.

\begin{lemma}\label{lemma G_k pq 1}
	A set of representatives of conjugacy classes of subgroups $G$ of $\Hol(\G_k)$ with $|\pi_2(G)|=pq$ is
	
	\bigskip
	\begin{center}
		\small{
			\begin{tabular}{c|l|c|l|c}
				$\pi_2(G)$ & Subgroups & Parameters & Isomorphism class & \#\\
				\hline
				&&&&\\[-1em]
				$\mathcal H_{1,s}$  &  $\widetilde H_{1,s,d}=\langle \tau, \sigma^{\frac{1}{g-1}}\alpha_1,\epsilon^d\beta_s\rangle$  & $0\leq s\leq q-1$, & $\G_{dk+s}$ & $q(q-1)$\\
				&&&&\\[-1em]
				& & $1\leq d\leq q-1$ & & \\
				&&&&\\[-1em]
				\cline{2-5}
				&&&&\\[-1em]
				&    $\widehat H_{1,s}=\langle \sigma, \tau\alpha_1,\epsilon^{\frac{1-s}{k}}\beta_s\rangle$ &  $2\leq s\leq q$ &$\G_{\frac{k-s+1}{k}}$ & $q-1$\\
				&&&&\\[-1em]
				\cline{2-5}
				&&&&\\[-1em]
				&  $\overline H_{1,s}=\langle \sigma\tau, \sigma^{\frac{1}{g-1}}\alpha_1,\epsilon^{\frac{s-1}{1-k}}\beta_s\rangle$ &  $2\leq s\leq q$ & $\G_{\frac{s-k}{1-k}}$ & $q-1$\\
				&&&&\\[-1em]
				\hline
				&&&&\\[-1em]
				$\mathcal H_{2,s}$    & $\widetilde H_{2,s,d}=\langle \sigma, \tau^{\frac{1}{g^k-1}}\alpha_2,\epsilon^d\beta_s\rangle $ & $0\leq s\leq q-1$, & $   \Z_p^2\times \Z_q$, if $s= 0, d= -1$,
				&    $q(q-1)$\\
				&&&&\\[-1em]
				& & $1\leq d\leq q-1$  & $\G_{0}$, if $s\neq 0,\, d=-1$, &  \\
				&&&&\\[-1em]
				& & &    $\G_{\frac{s}{d+1}}$,  otherwise.  &  \\
				&&&&\\[-1em]
				\cline{2-5}
				&&&&\\[-1em]
				&     $\widehat H_{2,s}=\langle \tau, \sigma\alpha_2,\epsilon^{s-1}\beta_s\rangle $  & $2\leq s\leq q$ &  $\G_0$,  if $s=0$, & $q-1$\\
				&&&&\\[-1em]
				& & & $\G_{\frac{s(k+1)-k}{s}}$, otherwise. &  \\
				&&&&\\[-1em]
				\cline{2-5}
				&&&&\\[-1em]
				& $\overline H_{2,s}=\langle \sigma\tau, \tau^{\frac{1}{g^k-1}}\alpha_2,\epsilon^{\frac{s-1}{1-k}}\beta_s\rangle $ & $2\leq s\leq q$ & $\G_0$ if $s=0$, & $q-1$ \\
				& & & $\G_{\frac{s-k}{s(1-k)}}$, otherwise.   & \\
				&&&&\\[-1em]
				\hline
				&&&&\\[-1em]
				
				$\mathcal K_1$    & $\widetilde K_{1,d}=\langle \tau, \sigma^{\frac{1}{g-1}}\alpha_1,\epsilon^d\widetilde\beta\rangle $ &  $1\leq d\leq q-1$ & $\Z_p^2\times \Z_q$ if $d=-k^{-1}$, &  $q-1$\\
				&&&&\\[-1em]
				& & & $\G_0$, otherwise.  & \\
				&&&&\\[-1em]
				\cline{2-5}
				&&&&\\[-1em]
				& $\widehat K_1=\langle \sigma, \tau\alpha_1,\epsilon^{-k^{-1}}\widetilde\beta\rangle$ &  - & $\G_0$ &$1$\\
				&&&&\\[-1em]
				\cline{2-5}
				&&&&\\[-1em]
				& $\overline K_1=\langle \sigma\tau, \sigma^{\frac{1}{g-1}}\alpha_1,\epsilon^{\frac{1}{1-k}}\widetilde\beta\rangle$&  - &$\G_0$ & $1$\\
				&&&&\\[-1em]
				\hline
				&&&&\\[-1em]
				$\mathcal K_{2}$ & $\widetilde K_{2,d}=\langle \sigma, \tau^{\frac{1}{g^k-1}}\alpha_2,\epsilon^d\widetilde\beta\rangle $ &$1\leq d\leq q-1$ & $\G_d$ &   $q-1$\\
				&&&&\\[-1em]
				\cline{2-5}
				&&&&\\[-1em]
				& $\widehat K_2=\langle \tau, \sigma\alpha_2,\epsilon\widetilde\beta\rangle$  & - & $\G_{k+1}$& $1$ \\
				&&&&\\[-1em]
				\cline{2-5} 
				&&&&\\[-1em]
				
				& $\overline K_2=\langle \sigma\tau, \tau^{\frac{1}{g^k-1}}\alpha_2,\epsilon^{\frac{1}{1-k}}\widetilde\beta\rangle$& - &$\G_{\frac{1}{1-k}}$ & $1$ \\
				&&&&\\[-1em]
				\hline
				&&&&\\[-1em]
				$\mathcal W$ & $\widetilde W_d=\langle \sigma, \tau^{\frac{1}{g^k-1}}\alpha_1\alpha_2,\epsilon^{d}\beta_1\rangle $ &  $1\leq d\leq q-1$&$\G_{d+1}$ & $q-1$\\
				&&&&\\[-1em]
				\cline{2-5}
				&&&&\\[-1em]
				& $\widehat W_{d}=\langle \tau, \sigma^{\frac{1}{g-1}}\alpha_1\alpha_2,\epsilon^d\beta_1\rangle $ & $1\leq d\leq q-1$&  $\G_{dk+1}$& $q-1$\\
		\end{tabular}}
	\end{center}
	
\end{lemma}

\begin{proof}
Any of the subgroups $G$ in the table have the form 
	\begin{equation}\label{shape}
	G=\langle u , w\delta, \epsilon^c h\rangle, 
	\end{equation}
for $\langle u,w\rangle=\langle \sigma,\tau\rangle$, $\delta\in \langle\alpha_1,\alpha_2\rangle$, $c\neq   0$, $\delta(u)=u$, $\delta(w)=w$, $h(\epsilon)=\epsilon$. According to Lemma \ref{pi_1 for fix}(2), we have $\langle \sigma,\tau\rangle \subseteq \pi_1(G)$. Since $\epsilon^c\in \pi_1(G)$, then $|\pi_1(G)|>p^2$ and so $\pi_1(G)=\G_k$.
	
	A set of representatives of the orbits of the elements of $\langle \sigma,\tau\rangle$ under the action of $\aut{\G_k}$ is $\{\sigma,\tau,\sigma\tau\}$. Then the subgroups with the same image under $\pi_2$ and belonging to different rows of the tables are not conjugate, since their kernels with respect to $\pi_2$ are not. 
	
Assume that two groups in the same row are conjugate. Then their images in the abelian quotient $\Hol(\G_k)/\H_2$ coincide. Then according to \eqref{shape}, their images are generated respectively by $\epsilon^c h $ and $\epsilon^d h$ and so it follows that $c=d$. 
	
	
	We focus on the case when the image under $\pi_2$ is $\mathcal H_{1,s}$, for the other cases we can apply the same ideas, so we omit the computation. 
		
	Let $G$ be a regular subgroup of $\Hol(\G_k)$ with $\pi_2(G)=\mathcal H_{1,s}$. The kernel of $\pi_2$ is a subgroup of order $p$ of $\G_k$ and we can choose it up to the action of the normalizer of $\mathcal H_{1,s}$ on the subgroups of $\langle \sigma,\tau\rangle$. Since $\setof{[(0,0),(a,b)]}{1\leq a,b\leq p-1}\leq N_{\aut{\G_k}}(\mathcal H_{1,s})$, we can assume that the kernel is generated by  $\sigma$, by $\tau$ or by $\sigma\tau$. The group $G$ has the form
	$$G=\langle u, \, v \epsilon^b\alpha_1, \, w\epsilon^d \beta_s\rangle$$
	where $u\in \{\sigma,\tau,\sigma\tau\}$ and $v,w\in \langle \sigma,\tau\rangle$. From the \nc condition $(v \epsilon^b\alpha_1)^p\in \ker{\pi_2} =\langle u\rangle$ we have that $(v \epsilon^b\alpha_1)^p \H_2=\epsilon^{bp}\H_2=\H_2$ and so $b=0$. If $d=0$ then according to Lemma \ref{pi_1 for fix}(1), $\pi_1(G)\subseteq \langle \sigma,\tau\rangle$ and so $G$ is not regular. Thus $d\neq 0$. From the \nc condition $\beta_s \alpha_1\beta_s^{-1}=\alpha_1^g$ we have that
	\begin{equation}\label{condition}
	  (w\epsilon^d\beta_s)v\alpha_1 (w\epsilon^d \beta_s)^{-1}=(v\alpha_1)^g=v^g \alpha_1^g \pmod {\langle u\rangle}.   
	\end{equation}

	(i) Assume that $\ker{\pi_2|_G}=\langle \sigma\rangle$. Then $G$ has a standard presentation as
	$$G=\langle \sigma, \, \tau^a\alpha_1,\, \tau^c \epsilon^d \beta_s\rangle$$
	for $a,d\neq 0$ and so up to conjugation by $[(0,0),(1,a^{-1})]$ we can assume that $a=1$. Hence $G$ has the form
	$$G=\langle \sigma, \tau \alpha_1,\tau^c\epsilon^d\beta_s\rangle.$$ By \eqref{condition} we have
	\begin{equation}\label{needed for g_0}
	(\tau^{c}\epsilon^d\beta_s)\tau \alpha_1(\tau^{c}\epsilon^d\beta_s)^{-1}=\sigma^{g\frac{g^c-1}{g-1}}\tau^{g^{dk+s}}\alpha_1^{g}=\tau^g\alpha_1^g\pmod{\langle \sigma\rangle} 
	\end{equation}
	and it follows that $d=\frac{1-s}{k}$ and so $s\neq 1$. If there exists $h\in \aut{\G_k}$ that normalizes $\langle \sigma, \tau\alpha_1\rangle$ and that $h \tau^c \epsilon^{\frac{1-s}{k}}\beta_s h^{-1}\in \widehat H_{1,s} $, the group $G$ is conjugate to $\widehat H_{1,s}$. Hence, according to \eqref{conj by alphas} we can choose $h$ to be a suitable power of $\alpha_1$.
	%
	%

	(ii) Assume that $\ker{\pi_2|_G}=\langle\tau\rangle$. Then 
	$$G=\langle \tau, \sigma^a\alpha_1,\sigma^c\epsilon^d\beta_s\rangle$$
	where $a,d\neq 0$. Similarly to the previous case, equation \eqref{condition} implies that 
	$a=\frac{1}{g-1}$. The relation $(\beta_s)^q=1$ holds and so $(\sigma^c\epsilon^{-1}\beta_s)^q\in \langle \tau\rangle$. According to \eqref{formula for powers} if $d=-1$ then $(\sigma^c\epsilon^d\beta_s)^q=\sigma^{qc}$, and so $c=0$. Otherwise, according to \eqref{conj by alphas}, the group $G$ is conjugate to $\widetilde H_{1,s,d}$ by a suitable power of $\alpha_1$.

	(iii) Assume that $\ker{\pi_2|_G}=\langle \sigma\tau\rangle$. Then 
	$$G=\langle \sigma\tau, \sigma^a\alpha_1,\sigma^c\epsilon^d\beta_s\rangle.$$
	Arguing as in (ii) we have $a=\frac{1}{g-1}$, $d=\frac{s-1}{1-k}\neq 0$, $s\neq 1$ and if $d=-1$ then $c=0$. If $d\neq -1$, according to \eqref{conj by alphas} we have that $G$ is conjugate to $\overline H_{1,s}$ by a suitable power of $\alpha_1$.
\end{proof}

\begin{lemma}\label{G_k_sub_ppq}
	A set of representatives of conjugacy classes of regular subgroups $G$ of $\Hol(\G_k)$ with $|\pi_2(G)|=p^2q$ is
	
	\bigskip
	\begin{center}
		\small{
			\begin{tabular}{c|l|c|c|c}
				$\pi_2(G)$ & Subgroups & Parameters  & Isomorphism class & \#\\
				\hline
				&&&&\\[-1em]
				$\mathcal T_{s}$    &    $\widetilde{T}_{c,s}=\langle \sigma^{\frac{1}{g-1}}\alpha_1,\tau^{\frac{1}{g^k-1}}\alpha_2,\epsilon^{c}\beta_s\rangle$ & $0\leq s\leq q-1$ &$\G_s$ &  $q(q-1)$\\
				&&&&\\[-1em]
				&  & $1\leq c\leq q-1$ & & \\
				&&&&\\[-1em]
				\cline{2-5}
				&&&&\\[-1em]
				&  $\widehat{T}_{s}=\langle  \sigma^{\frac{1}{g-1}}\alpha_1, \sigma\tau^{\frac{1}{g^k-1}}\alpha_2,\epsilon^{s-1}\beta_s \rangle$  & $2\leq s\leq q$& $\G_{s}$ & $q-1$\\
				&&&&\\[-1em]
				\cline{2-5}
				&&&&\\[-1em]
				&  $\overline{T}_{s}=\langle \tau\sigma^{\frac{1}{g-1}}\alpha_1,  \tau^{\frac{1}{g^k-1}}\alpha_2,\epsilon^{\frac{1-s}{k}}\beta_s \rangle$ & $2\leq s\leq q$ & $\G_s$ & $q-1$\\
				&&&&\\[-1em]
				\hline
				&&&&\\[-1em]
				$\mathcal U$ &     $\widetilde{U}_{c}=\langle \sigma^{\frac{1}{g-1}}\alpha_1, \tau^{\frac{1}{g^k-1}}\alpha_2,\epsilon^{c}\widetilde\beta \rangle$  & $1\leq c\leq q-1$ &  $\G_0$ & $q-1$\\
				&&&&\\[-1em]
				\cline{2-5}
				&&&&\\[-1em]
				& $\widehat{U}=\langle \tau \sigma^{\frac{1}{g-1}}\alpha_1, \tau^{\frac{1}{g^k-1}}\alpha_2,\epsilon^{-k^{-1}}\widetilde\beta \rangle $ & -& $\G_0$ & $1$ \\
				&&&&\\[-1em]
				\cline{2-5}
				&&&&\\[-1em]
				& $\overline{U}=\langle \sigma^{\frac{1}{g-1}}\alpha_1,\tau^{\frac{1}{g^k-1}}\sigma\alpha_2, \epsilon\widetilde\beta \rangle $ & -& $\G_0$ &   $1$  
		\end{tabular}}
	\end{center}
\end{lemma}

\begin{proof}
	The groups in different rows of the tables are not conjugate since their $p$-Sylow subgroups are not. The same argument used in Lemma \ref{lemma G_k pq 1} shows that the groups in the table are regular and that the groups in the same row and with the same image under $\pi_2$ are not conjugate (e.g. if $\widetilde{T}_{c,s}$ and $\widetilde{T}_{d,s}$ are conjugate then $c=d$).
		
	Let $G$ be a regular subgroup of $\Hol(\G_k)$ with $\pi_2(G)=\mathcal T_s$. Then
	$$G=\langle u \epsilon^a \alpha_1,v\epsilon^b\alpha_2,w\epsilon^c\beta_s\rangle$$
	for $u,v,w\in \langle \sigma,\tau\rangle$ and $0\leq a,b,c\leq q-1$. We need to check the \nc conditions for the generators of $G$. Since
\begin{align*}
(u\epsilon^a\alpha_1)^p&=(v\epsilon^b\alpha_2)^p\overset{\nc}{=}1, 
\end{align*}
	then we have that $(u\epsilon^a\alpha_1)^p\H_2=\epsilon^a\H_2=(v\epsilon^b\alpha_2)^p\H_2=\epsilon^b\H_2$ and so $a=b=0$. If $c=0$, according to Lemma \ref{pi_1 for fix}(1), $\pi_1(G)\subseteq \langle \sigma,\tau\rangle$ and so $G$ is not regular. Hence $c\neq 0$ and 
\begin{align}
(w\epsilon^c\beta_s)u\alpha_1(w\epsilon^c\beta_s)^{-1}&\overset{\nc}{=}(u\alpha_1)^g=u^g \alpha_1^g,\label{eqeq1}\\
(w\epsilon^c\beta_s)v\alpha_2(w\epsilon^c\beta_s)^{-1}&\overset{\nc}{=}(v\alpha_2)^{g^s}=v^{g^s}\alpha_2^{g^s}.\label{eqeq2}
\end{align}
Equation \eqref{eqeq1} implies that $u=\sigma^{\frac{1}{g-1}}\tau^n$ where either $n=0$ or $c= k^{-1}(1-s)$ and equation \eqref{eqeq2} that $v=\sigma^m\tau^{\frac{1}{g^k-1}}$ where either $m=0$ or $c= s-1$. The equality $c=k^{-1}(1-s)=s-1$ implies that $k=-1$ or $s=1$ and so $c=0$, contradiction. Thus, we have to discuss the following cases:
\begin{itemize}
    \item[(i)] $n=m=0$.
    \item[(ii)] $m=0$, $c=k^{-1}(1-s)$, $n\neq 0$ and $s\neq 1$. 
    \item[(iii)] $n=0$, $c=1-s$, $n\neq 0$ and $s\neq 1$. 
\end{itemize}
	
(i)	If $n=m=0$, then we have that $$G=\langle \sigma^{\frac{1}{g-1}}\alpha_1, \tau^{\frac{1}{g^k-1}}\alpha_2,\sigma^a\tau^b\epsilon^c\beta_s\rangle$$
	for some $a,b$. Let $h=\alpha_1^r \alpha_2^t$. To see that $hGh^{-1}=\widetilde{T}_{c,s}$ for some $r,t$, it is sufficient to show that $h\sigma^a\tau^b\epsilon^c\beta_s h^{-1}\in \widetilde{T}_{c,s}$. Using \eqref{conj by alphas}, this is equivalent to have that $(r,t)$ is a solution of the following system of linear equations: 
	\begin{equation*}
		\begin{cases}{}
			r(g^{c+1}-1)=a(1-g)\\
			t(1-g^{ck+s})=b(1-g^k).
		\end{cases}
	\end{equation*}    
	Using \eqref{formula for powers} and the \nc condition $(\sigma^a\tau^b\epsilon^c\beta_s)^q=1$, if $c+1=0$ then $a=0$ and if $ck+s=0$ then $b=0$. So the system admits a solution for every $a,b$.
	
(ii) Let $m=0$ and $c=1-s$.  Up to conjugation by the automorphism $h=[(0,0),(n^{-1},1)]$ we can assume that $n=1$. Accordingly $G$ has the following form
	\begin{align*}
		G&=\langle \tau\sigma^{\frac{1}{g-1}} \alpha_1,\tau^{\frac{1}{g^k-1}}\alpha_2, \sigma^a\tau^b\epsilon^{\frac{1-s}{k}}\beta_s\rangle
	\end{align*}
	for some $a,b$. Let $(r,t)$ be a solution of the following system of linear equations: 
	\begin{equation*}
			\begin{cases}{}
			r(g^{\frac{1-s+k}{k}}-1)=a(1-g)\\
			t(1-g)=b(1-g^k)+r(1-g^k)(1-g).
		\end{cases}
	\end{equation*}    
	We have $(\sigma^a\tau^b\epsilon^{\frac{1-s}{k}}\beta_s)^q=1$. According to \eqref{formula for powers}, if $\frac{1-s}{k}+1=\frac{1-s+k}{k}=0$ then $a=0$. Hence the system admits a solution for every $a,b$ and $\alpha_1^r \alpha_2^t G(\alpha_1^r \alpha_2^t)^{-1}=\widehat{T}_{s}$. 
	%
	
	(iii) Let $n=0$ and $c=s-1$. Up to conjugation by $[(0,0),(1,m^{-1})]$ we can assume that $m=1$, i.e. 
		\begin{align*}
		G&=\langle \sigma^{\frac{1}{g-1}} \alpha_1,\sigma\tau^{\frac{1}{g^k-1}}\alpha_2, \sigma^a\tau^b\epsilon^{s-1}\beta_s\rangle
	\end{align*}
	for some $a,b$. Let $(r,t)$ be a solution of the following system of linear equations: 
	\begin{equation*}
		\begin{cases}{}
		r(g^s-1)=(1-g^s)t+(1-g)a\\
t(g^{sk-k+s}-1)=(1-g^k)b.
		\end{cases}
	\end{equation*}    
	From the condition $(\sigma^a\tau^b\epsilon^{s-1}\beta_s)^q=1$ and \eqref{formula for powers} we have that if $sk-k+s=0$ then $b=0$ and if $s=0$ then $a=0$. Accordingly the system has solution and $\alpha_1^r \alpha_2^t G(\alpha_1^r \alpha_2^t)^{-1}=\overline{T}_{s}$.

	We can use the same ideas when $\pi_2(G)=\mathcal U$. In such case, $G$ is conjugate to $\widetilde{U}_c$, $\widehat{U}$ or $\overline{U}$.
\end{proof}

We summarize the content of this subsection in the following table:

\begin{table}[H]
	\centering
	
	\small{
		\begin{tabular}{c|c|c|c|c|c|c}
			$|\ker{\lambda}|$ & $\Z_p^2\times\Z_q$ & $\G_k$ & $\G_s,\, s\neq 0,\pm  1, k$& $\G_0$ & $\G_{-1}$ & $\G_1$ \\
			\hline
			$1$ &- & $2(q+1)$ & $2(q+1)$ & $2(q+1)$ & $q+1$ & $q-1$  \\
			$p$ & $2$ & $4(q+2)$ & $4(q+2)$ & $4(q+1)$ & $2(q+2)$ &    $2(q-1)$   \\
			$q$ & $1$ &- &- &- &- &-\\
			$pq$ & - & - & -  & $4$ & - & -\\
			$p^2$ & $1$ & $2q-3$ & $2(q-1)$ & $2(q-1)$ & $q-1$  & $q-1$  \\
			$p^2q$ & - & $1$ & - & - &- & -
	\end{tabular}}
	\vs
	\caption{Enumeration of skew braces of $\G_k$-type for $p=1\pmod{q}$.}
\end{table}

\subsection{Skew braces of $\G_0$-type}\label{subsection:p=1(q)_G_0}

Let us consider the group $\G_0$ with presentation
$$\G_0=\langle \sigma,\tau,\epsilon\ |\ \sigma^p=\tau^p=\epsilon^q=[\sigma,\tau]=[\tau,\epsilon]=1,\, \epsilon \sigma\epsilon^{-1}=\sigma^{g}\rangle.$$
According to the description of the automorphisms of $\G_0$ provided in \cite[Theorem 3.1, 3.4]{auto_pq}, the map
$$\phi: \mathbb{Z}_p \rtimes_{\rho} \left(\mathbb{Z}_{p}^\times\times \mathbb{Z}_p^\times\right)\longrightarrow \aut{\G_0}, \quad [n,(a,b)]\mapsto 	h = 	\begin{cases}
h|_{\langle \sigma,\tau\rangle}=\begin{bmatrix}
a& 0\\
0 & b
\end{bmatrix}
,\quad \epsilon\mapsto   \sigma^n \epsilon       
\end{cases}$$
where $\rho(a,b)(n)=an$, is a group isomorphism. In particular, $|\aut{\G_{0}}|=	p(p-1)^2$ and its $p$-Sylow subgroup is generated by $\alpha=[1,(1,1)]$. 

If $G$ is a regular subgroup of $\Hol(\G_0)$, then $|\pi_2(G)|$ divides $p^2q$ and $|\aut{\G_0}|$, so it divides $pq$. 

In this section we can employ the same computations we did in Section \ref{subsection:p=1(q)_G_k}, taking into account that in this case the $p$-Sylow subgroup of $\aut{\G_0}$ is cyclic and generated by $\alpha$. 

%

\begin{lemma}\label{G_0_sub_p}
	A set of representative of regular subgroups $G$ of $\Hol(\G_0)$ with $|\pi_2(G)|=p$ is
	\begin{align*}
		H=\langle \epsilon, \tau, \sigma^{\frac{1}{g-1}}\alpha\rangle\cong \Z_p^2\times \Z_q, \quad
		K=\langle \epsilon, \sigma,\tau\alpha\rangle\cong \G_0. 
	\end{align*}
\end{lemma}

\begin{proof}
   We can argue as in Lemma \ref{G_k_sub_p} taking into account that the subgroups of order $q$ of $\G_0$ are generated by elements of the form $\sigma^n \epsilon$.
\end{proof}
%
%
%

The subgroup
\begin{equation}\label{H3}
 \H_3=\langle \sigma,\tau,\alpha\rangle
\end{equation}
is a normal subgroup of $\Hol(\G_0)$ with $\Hol(\G_0)/\H_3\cong \Z_q \times \Z_p^\times \times \Z_p^\times$, formula \eqref{formula for powers} holds in $\Hol(\G_0)$ setting $k=0$ and the formula 
 \begin{equation}\label{conj by alphas for G_0}
    \alpha^n u\epsilon^d \theta  \alpha^{-n}=u\sigma^{x n\frac{g^d-1}{g-1}}\alpha^{(1-x)n}\epsilon^d \theta,\tag{$F_0$}
\end{equation}
where $u\in \langle \sigma,\tau\rangle$ and $\theta=[(0,0),(x,y)]$ holds is $\G_0$. Then, using the same argument of Proposition \ref{G_k_sub_q}, we have the following.

\begin{proposition}\label{Formula G0 q}
	The skew braces of $\G_0$-type with $|\ker{\lambda}|=p^2$ are $(B_{a,b},+,\circ)$, where $(B_{a,b},+)=\Z_p^2\rtimes_{\dm_{1,0}} \Z_q$ and
	$$(B_{a,b},\circ)=\Z_p^2\rtimes_{\dm_{a+1,b}} \Z_q\cong \begin{cases}
	\Z_p^2\times \Z_q, \, \text{ if } a=q-1, \, b=0,\\
	\G_0,\, \text{ if } a=q-1, \, b\neq 0,\\
	\G_{\frac{b}{a+1}},\, \text{otherwise, }\\
	\end{cases} $$
	for $0\leq a,b\leq q-1$ and $(a,b)\neq (0,0)$. In particular, there are $q^2-1$ such skew braces and they are bi-skew braces.
\end{proposition}

Similarly to Lemma \ref{pi_2 pq G_k} we obtain that a set of representative of conjugacy classes of subgroups of order $pq$ of $\aut{\G_0}$ is 
$$\mathcal H_s=\langle \alpha,\beta_s\rangle, \quad \mathcal K=\langle \alpha,\widetilde\beta\rangle$$
where $\beta_s=[0,(g,g^s)]$ and $0\leq s\leq q-1$ and $\widetilde{\beta}=[0,(1,g)]$. 

\begin{lemma}\label{G_0_pq}
	A set of representatives of regular subgroups $G$ of $\Hol(\G_0)$ with $|\pi_2(G)|=pq$ is

	\bigskip
	\begin{center}
		\small{
			\begin{tabular}{c|l|c|c|c}
				$\pi_2(G)$ & Subgroups& Isomorphism class & Parameters   & \#\\
				\hline
				&&&&\\[-1em]
				$\mathcal H_{s}$    &    $\widetilde H_{s,c}=\langle \tau, \sigma^{\frac{1}{g-1}}\alpha,\epsilon^{c}\beta_s\rangle$ & $\G_s$ & $0\leq s\leq q-1$,  & $q(q-1)$\\
				&&&&\\[-1em]
				& & &  $1\leq c\leq q-1$ &  \\
				&&&&\\[-1em]
				\cline{2-5}
				&&&&\\[-1em]
				&  $\widehat H_{s}=\langle \sigma\tau, \sigma^{\frac{1}{g-1}}\alpha,\epsilon^{s-1}\beta_s\rangle$& $\G_s$  & $2\leq s\leq q$ & $q-1$\\
				&&&&\\[-1em]
				\hline
				&&&&\\[-1em]
				$\mathcal H_1$  &  $  \overline H_{1,c}=\langle \sigma, \tau\alpha,\epsilon^c\beta_1\rangle$ & $\G_{c+1}$ &  $1\leq c\leq q-1$ &$q-1$\\
				&&&&\\[-1em]
				\hline
				&&&&\\[-1em]
				$\mathcal K$ &     $\widetilde K_{c}=\langle \tau, \sigma^{\frac{1}{g-1}}\alpha,\epsilon^{c}\widetilde\beta\rangle$ &$\G_0$ &   $1\leq c\leq q-1$ & $q-1$\\
				&&&&\\[-1em]
				\cline{2-5}
				&&&&\\[-1em]
				& $\widehat K =\langle \sigma\tau, \sigma^{\frac{1}{g-1}}\alpha,\epsilon\widetilde\beta\rangle  $ &$\G_0$  & -&  $1$
		\end{tabular}}
	\end{center}
	
	\bigskip
\end{lemma}

    \begin{proof}
        %
	By the same argument used in Lemma \ref{lemma G_k pq 1} we can show that the groups in the table are regular and they are not conjugate (in this case we are going to use the group $\H_3$ as in \eqref{H3}).         
In order to prove that the groups in the table form a set of representatives, we can employ the very same computations we performed in Lemma \ref{lemma G_k pq 1} for the groups with image under $\pi_2$ equal to $\langle \alpha_1, \beta_s \rangle$ and $\langle \alpha_1,\widetilde \beta\rangle$. 

If we assume that $\pi_2(G)=\mathcal{H}_s$ and $\ker{\pi_2|_G}=\langle \sigma\rangle$ the \nc condition \eqref{needed for g_0} implies that $s=1$. For the other cases we can just set $k=0$ and the same computations will do. 

If we assume that $\pi_2(G)=\mathcal{K}$ and $\ker{\pi_2|_G}=\langle \sigma\rangle$ then 
$$G=\langle \sigma,\tau^a \alpha, \tau^b \epsilon^c\widetilde \beta\rangle$$ 
and the same condition implies that $a=0$, i.e. $G$ is not regular. The other cases are basically unchanged.

Moreover we can use \eqref{conj by alphas for G_0} instead of \eqref{conj by alphas} to prove that any regular subgroup $G$ is conjugate to one in the table by a suitable power of $\alpha$.
    \end{proof}

\begin{remark}\label{direct of G_0 type}
	A skew brace $B$ is a direct product if and only if there exists $I,J$ ideals of $B$ such that $I+J=B$ and $I\cap J=0$. The direct product of a skew brace of size $p^2$ and a skew brace of size $q$ is of abelian type. The following are the skew braces of $\G_0$-type which are direct products of the trivial skew brace of size $p$ and a skew brace of size $pq$. Since $\G_0$ is the direct product of $\Z_p$ and $\Z_p\rtimes_g \Z_q$ we need to consider just the skew braces of size $pq$ with non-abelian additive group. According to the classification provided in \cite{skew_pq}, there are $2q-1$ non trivial such skew braces, then we need to find $2q-1$ skew braces that are a direct product:
	\begin{itemize}
		\item[(i)] The skew braces $B_{a,0}$ for $a\neq 0$ as defined in Proposition \ref{Formula G0 q} are direct product of the trivial skew brace of size $p$ and a skew brace of size $pq$ with $|\ker{\lambda}|=p$. 
		\item[(ii)] The skew brace associated to the group $H$ of Lemma \ref{G_0_sub_p} is the direct product of the trivial skew brace of size $p$ and the unique skew brace of size $pq$ with $|\ker{\lambda}|=q$.
		\item[(iii)] The skew brace associated to $\widetilde T_{0,c}$ for $c\neq 0$ as defined in Lemma \ref{G_0_pq} is the direct product of the trivial skew brace of size $p$ and a skew brace of size $pq$ with $|\ker{\lambda}|=1$.
	\end{itemize}
\end{remark}

We summarize the content of this subsection in the following table:
\begin{table}[H]
	\centering
	
	\small{
		\begin{tabular}{c|c|c|c|c|c}
			$|\ker{\lambda}|$ &  $\Z_p^2\times\Z_q$ & $\G_k, k\in \B\setminus\{0,\pm 1\} $ & $\G_0$ & $\G_{-1}$ & $\G_1$ \\
			\hline
			$p$ &-  & $2(q+1)$ & $2q+1$ & $q+1$ & $q-1$  \\
			$pq$  & $1$ & - & $1$ & - & - \\
			$p^2$  & $1$  & $2(q-1)$ & $2q-3$ & $q-1$ & $q-1$ \\
			$p^2q$  &- &- & $1$ &- &- \\
	\end{tabular}}
	\vs
	\caption{Enumeration of skew braces of $\G_0$-type for $p=1\pmod{q}$.}
	\label{table:G_0}
\end{table}

\subsection{Skew braces of $\G_{-1}$-type}\label{subsection:p=1(q)_G_-1}

Let us consider the group $\G_{-1}$ with presentation
$$\G_{-1}=\langle \sigma,\tau,\epsilon\ |\ \sigma^p=\tau^p=\epsilon^q=[\sigma,\tau]=1,\, \epsilon \sigma\epsilon^{-1}=\sigma^{g},\ \epsilon \tau\epsilon^{-1}=\tau^{g^{-1}}\rangle.$$
Let 
$$\mathfrak G=\left(\Z_p^2\rtimes_\rho (\Z_p^\times\times \Z_p^\times)\right)\rtimes_{\rho^\prime} \Z_2$$ where $\rho(a,b)(n,m)=(an,bm)$, $\rho^\prime(1)[(n,m),(a,b)]=[(-gm ,-g^{-1} n),(b,a)]$ for every $0\leq n,m\leq p-1$, $1\leq a,b\leq p-1$. According to \cite[Subsections 4.1, 4.3]{auto_pq}, the mapping
\begin{eqnarray*}
	\phi:\mathfrak G &\longrightarrow &\aut{\G_{-1}},\quad \left[(n,m),(a,b), i\right]  \mapsto 	
	h = 	\begin{cases}
		h |_{\langle\sigma,\tau\rangle}=H_{i}(a,b),\\ \epsilon\mapsto   \sigma^n\tau^m \epsilon^{(-1)^i},
	\end{cases} 
\end{eqnarray*}
where $$H_{0}(a,b)=\begin{bmatrix} a& 0 \\
0& b \end{bmatrix}, \quad H_{1}(a,b)=\begin{bmatrix} 0& a \\
b& 0 \end{bmatrix},$$
is an isomorphism. The map
$$\nu:\aut{\G_{-1}}\longrightarrow \Z_2,\quad [(n,m),(a,b),i]\mapsto  i$$
 is a group homomorphism and the group $\aut{\G_{-1}}$ is generated by $\aut{\G_{-1}}^+=\ker{\nu}$ and by the involution $\psi=[(0,0),(1,1),1]$, defined by
\begin{equation}\label{psi}
\psi=\begin{cases} \sigma\mapsto\tau,\quad     \tau\mapsto \sigma,\quad     \epsilon\mapsto\epsilon^{-1}.
\end{cases}
\end{equation}

Note that $\aut{\G_{-1}}^+\cong \aut{\G_k}$ for $k\neq 0,\pm 1$ and it contains the $p$-Sylow subgroup of $\aut{\G_{-1}}$, generated by $\alpha_1=[(1,0),(1,1),0]$ and $\alpha_2=[(0,1),(1,1),0]$, and 
the elements of odd order of $\aut{\G_{-1}}$. Since $p>2$, then $\pi_2(G)\leq \aut{\G_{-1}}^+$ and $G\leq \G_{-1}\rtimes \aut{\G_{-1}}^+\leq \Hol(\G_{-1})$ for every regular subgroup $G\leq \Hol(\G_{-1})$. Nevertheless the conjugacy classes of regular subgroups might be different, since we need to take into account the conjugation by $\psi$.

In some cases, the computations of the conjugacy classes can be obtained from those in Subsection \ref{subsection:p=1(q)_G_k} just by setting $k=-1$, as for the following that follows directly from Lemma \ref{G_k_sub_pp}.  
\begin{lemma}\label{sub_G_-1_pp}
	The unique skew brace of $\G_{-1}$-type with $|\ker{\lambda}|=q$ is $(B,+,\circ)$ where $(B,+)=\Z_p^2\rtimes_{\dm_{1,-1}} \Z_q$ and
		\begin{eqnarray*} 
		\begin{pmatrix} x_1 \\ x_2 \\ x_3 \end{pmatrix} \circ \begin{pmatrix} y_1 \\ y_2\\ y_3 \end{pmatrix} 
		=\begin{pmatrix}x_1 g^{y_3}+y_1 g^{x_3}\\  x_2 g^{- y_3}+y_2 g^{-x_3} \\ x_3+y_3\end{pmatrix}
	\end{eqnarray*}
	for every $0\leq x_1,x_2,y_1,y_2\leq p-1$ and $0\leq x_3,y_3\leq q-1$. In particular, $(B,\circ)\cong \Z_p^2\times \Z_q$.
\end{lemma}

We claim that the computations in Lemma \ref{G_k_sub_p}, Proposition \ref{G_k_sub_q}, Lemma \ref{pi_2 pq G_k} and  Lemma \ref{lemma G_k pq 1} depend just on the condition $k\neq 0,1$ and so we can employ them in the proofs of the next results. Note also that the subgroup
$$\H_4=\langle \sigma,\tau,\alpha_1,\alpha_2\rangle$$
is a normal subgroup of $\Hol(\G_{-1})$ and setting $k=-1$ in \eqref{conj by alphas} we get a formula that holds in $\Hol(\G_{-1})$.

\begin{lemma}\label{G_-1_sub_p}
	A set of representatives of regular subgroups $G$ of $\Hol(\G_{-1})$ with $|\pi_2(G)|=p$ is
	\begin{equation*}
		H_i=\langle \epsilon, \tau, \sigma^{\frac{1}{g-1}}\alpha_1\alpha_2^i\rangle\cong \G_0
	\end{equation*}
	for $i=0,1$.
\end{lemma}

\begin{proof} The groups $H_0$ and $H_1$ are not conjugate since their images under $\pi_2$ are not. Arguing as in Lemma \ref{G_k_sub_p} we can show that every regular subgroup of $\Hol(\G_{-1})$ with $|\pi_s(G)|=p$ is conjugate to the groups $H_i, K_i$ for $i=0,1$ defined by setting $k=-1$ in the groups in Lemma \ref{G_k_sub_p}. It is easy to see that $H_0$ is conjugate to $K_0$ by the automorphism $\psi$ and that $H_1$ is conjugate to $K_1$ by the automorphism $[(0,0),(g^2,1),0]$. 
\end{proof}

Let us denote $\beta_s=[(0,0),(g,g^s),0]$ and $\widetilde \beta=[(0,0),(1,g),0]$ for $0\leq s\leq q-1$.

\begin{proposition}\label{G_{-1}_sub_q}
	The skew braces of $\G_{-1}$-type with $|\ker{\lambda}|=p^2$ are $(B_{a,b},+,\circ)$ where $(B_{a,b},+)=\Z_p^2\rtimes_{\dm_{1,-1}}\Z_q$ and 
	$$(B_{a,b},\circ)=\Z_p^2\rtimes_{\dm_{a+1,b-1}} \Z_q\cong \begin{cases}
	\Z_p^2\times \Z_q, \, \text{ if } a=q-1, \, b=1,\\
	\G_0,\, \text{ if } a=q-1, \, b\neq 1,\\
	\G_{\frac{b-1}{a+1}},\, \text{otherwise},\\
	\end{cases}$$
	for $0\leq a,b\leq q-1$ and $(a,b)\neq (0,0)$. In particular, they are bi-skew braces and $B_{a,b}\cong B_{c,d}$ if and only if $(c,d)=(-b,-a)$ and so there are $\frac{(q-1)(q+2)}{2}$ such skew braces.
\end{proposition}

\begin{proof}
	Let $G$ be a regular subgroup of $\Hol(\G_{-1})$ such that $|\pi_2(G)|=q$. Arguing as in Proposition \ref{G_k_sub_q} we can show that every such group is conjugate to a subgroup
	$$G_{a,b}=\langle \sigma,\tau, \epsilon \beta_0^a\widetilde \beta^b\rangle$$
 by an element of $\aut{\G_{-1}}^+$. Since $\psi G_{a,b}\psi =G_{-b,-a}$, 
	the groups of the form $G_{a,-a}$ are normalized by $\psi$ and the other orbits have length $2$. Therefore there are
	$$q-1+\frac{q(q-1)}{2}=\frac{(q-1)(q+2)}{2}$$
	orbits under the action of $\psi$. The statement about the structure of the associated skew braces follows by the same argument of Proposition \ref{G_k_sub_q}.
\end{proof}

\begin{lemma}\label{pi_2 pq for k=-1}
	A set of representatives of conjugacy classes of subgroups of $\aut{\G_{-1}}$ of order $pq$ and $p^2q$ is
	\begin{center}
		
		\small{
			\begin{tabular}{c|l|c|l}\label{subgroups of size pq of aut G_-1}
				Size & $G$ & Parameters & Isomorphism class\\  
				\hline
				&&&\\[-1em]
				$pq$ & $\mathcal H_{1,s}=\langle \alpha_1,\beta_s\rangle$ & $0\leq s\leq q-1$ & $\mathbb{Z}_{p}\rtimes_{g}\Z_q$\\
				&&&\\[-1em]
				\cline{2-4}
				&&&\\[-1em]
				& $\mathcal H_{2,0}=\langle \alpha_2,\beta_0\rangle$  & - & $\Z_{pq}$\\
				&&&\\[-1em]
				\cline{2-4}
				&&&\\[-1em]
				&    $\mathcal W=\langle \alpha_1\alpha_2,\beta_1\rangle$  & - & $\mathbb{Z}_{p}\rtimes_{g}\Z_q$\\
				&&&\\[-1em]
				\hline
				&&&\\[-1em]
				$p^2q$ & $\mathcal T_{s}=\langle \alpha_1,\alpha_2,\beta_s\rangle\cong \G_s$  & $s\in\B$ & $\G_s$
		\end{tabular}}
		
	\end{center}
	where $\B$ is as in Remark \ref{subgroups of GL}.
\end{lemma}

\begin{proof}
	The subgroups of size $pq$ of $\aut{\G_{-1}}^+$ up to conjugation by elements of $\aut{\G_{-1}}^+$ are collected in Lemma \ref{pi_2 pq G_k}. We need to compute the orbits of such groups under the action by conjugation of $\psi$. It is easy to see that 
	\begin{itemize}
	    \item 	$\psi  \mathcal H_{1,s}\psi=\mathcal H_{2,s^{-1}}$ and $\psi \mathcal T_s \psi=\mathcal T_{s^{-1}}$ for $s\neq 0$; 
	    \item $\psi \mathcal H_{1,0}\psi=\mathcal K_2$, $\psi \mathcal H_{2,0}\psi=\mathcal K_1$ and $\psi \mathcal U\psi =\mathcal T_0$; 
	    \item $\psi \mathcal W\psi =\langle \alpha_1^{-g} \alpha_2^{-g^{-1}}, \beta_1  \rangle$ is not conjugate to any other group in the table, since their $p$-Sylow subgroups are not.\qedhere
	\end{itemize}
\end{proof}

\begin{lemma}\label{G_1 sub aut ppq}
	A set of representatives of conjugacy classes of regular subgroups $G$ of $\Hol(\G_{-1})$ with $|\pi_2(G)|=pq$ is the following:
	\bigskip
	\begin{center}
		\small{
			\begin{tabular}{c|l|c|l|c}
				$\pi_2(G)$ & Subgroups & Parameters& Isomorphism class  & \#\\
				\hline
				&&&&\\[-1em]
				$\mathcal H_{1,s}$  &  $\widetilde H_{1,s,d}=\langle \tau, \sigma^{\frac{1}{g-1}}\alpha_1,\epsilon^d\beta_s\rangle$ & $0\leq s\leq q-1$,  &$\G_{s-d}$ &  $q(q-1)$\\
				&&&&\\[-1em]
				& & $1\leq d\leq q-1$ & & \\
				&&&&\\[-1em]
				\cline{2-5}
				&&&&\\[-1em]
				&    $\widehat H_s=\langle \sigma, \tau\alpha_1,\epsilon^{s-1}\beta_s\rangle$ &  $2\leq s\leq q$ &$\G_{s}$ & $q-1$\\
				&&&&\\[-1em]
				\cline{2-5}
				&&&&\\[-1em]
				&  $\overline H_s=\langle \sigma\tau, \sigma^{\frac{1}{g-1}}\alpha_1,\epsilon^{\frac{s-1}{2}}\beta_s\rangle$ &$2\leq s\leq q$ & $\G_{\frac{s+1}{2}}$ &  $q-1$\\
				&&&&\\[-1em]
				\hline
				&&&&\\[-1em]
				$\mathcal H_{2,0}$ & $\widetilde H_{2,0,d}=\langle \sigma, \tau^{\frac{1}{g^{-1}-1}}\alpha_2,\epsilon^d\beta_0\rangle $ & $1\leq d\leq q-1$ & $\mathbb{Z}_p^2\times \mathbb{Z}_q$,  if $d=-1$ & $q-1$\\
				&&&&\\[-1em]
				& & & $\G_{0}$,  otherwise  & \\
				&&&&\\[-1em]
				\cline{2-5}
				&&&&\\[-1em]
				&     $\widehat H_{2,0}=\langle \tau, \sigma\alpha_2,\epsilon^{-1}\beta_0\rangle $ &   - &$ \G_0$ & $1$\\
				&&&&\\[-1em]     
				\cline{2-5}
				&&&&\\[-1em]
				& $\overline H_{2,0}=\langle \sigma\tau, \tau^{\frac{1}{g^{-1}-1}}\alpha_2,\epsilon^{-\frac{1}{2}}\beta_0\rangle $ &   - &$\G_0$ & $1$ \\
				&&&&\\[-1em]
				\hline
				&&&&\\[-1em]
				$\mathcal W$ & $\widetilde W_d=\langle \sigma, \tau^{\frac{1}{g^{-1}-1}}\alpha_1\alpha_2,\epsilon^{d}\beta_1\rangle $ &  $1\leq d\leq q-1$&$\G_{d+1}$ & $q-1$
		\end{tabular}}
	\end{center}
	
\end{lemma}

\begin{proof}
%
Let $G$ be a regular subgroup of $\Hol(\G_{-1})$ and assume that $\pi_2(G)=\mathcal{H}_{i,s}$ for $s$ and $i$ as in Lemma \ref{pi_2 pq for k=-1}. Using the same computations as in Lemma \ref{lemma G_k pq 1} we can show that $G$ is conjugate to one of the groups in the table with the same image under $\pi_2$ by some element of $\aut{\G_{-1}}^+$. Such groups are not conjugate. Indeed, if $\pi_2(G)\neq \mathcal W$, then its normalizer is contained in $\aut{\G_{-1}}^+\cong \aut{\G_k}$ for $k\neq 0,\pm 1$ and so we can prove that the corresponding groups in the table are not conjugate as we did in Lemma \ref{lemma G_k pq 1}. 

	If $\pi_2(G)=\mathcal{W}$, the groups with $\pi_2(G)=\mathcal{W}$ are conjugate to $\widetilde W_d$ for some $d\neq 0$ by some element in $\aut{\G_{-1}}^+$. If $\widetilde W_c$ and $\widehat W_{d}$ are conjugate in $\aut{\mathcal{G}_{-1}}$ by $h\psi$ for $h\in \aut{\G_{-1}}^+$ so are their images modulo the subgroup $\H_4$. Hence we have
	$$h\psi\epsilon^c  \beta_1 \psi h^{-1}\H_4=\epsilon^{ -c}  \beta_1\H_4\in \langle \epsilon^d \beta_1\H_4\rangle,$$
and so $c=-d$. It is easy to see that the groups $\widetilde W_d$ and $\widehat W_{-d}$ are conjugate by $[(0,0),(1,g^2),0]$.
	
	Therefore the table collects a set of representatives of regular subgroups with the desired properties.
\end{proof}

\begin{lemma}\label{label}	A set of representatives of conjugacy classes of regular subgroups $G$ of $\Hol(G_{-1})$ with $|\pi_2(G)|=p^2q$ is 
		
	\bigskip
	\begin{center}
		\small{
			\begin{tabular}{c|l|c|c|c}
				$\pi_2(G)$ & Subgroups & Class & Parameters & \#\\
				\hline
				&&&&\\[-1em]
				$\mathcal T_{s}$    &    $\widetilde T_{d,s}=\langle \sigma^{\frac{1}{g-1}}\alpha_1, \tau^{\frac{1}{g^{-1}-1}}\alpha_2,\epsilon^{d}\beta_s\rangle$ & $\G_{s}$ & $s\in\B\setminus\{1\}$, & $\frac{q^2-1}{2}$\\
				&&&&\\[-1em]
				& & &  $1\leq d\leq q-1$  & \\
				&&&&\\[-1em]
				\cline{2-5}
				&&&&\\[-1em]
				& $\widehat T_{m,s}=\langle \tau\sigma^{\frac{1}{g-1}}\alpha_1, \sigma^m\tau^{\frac{1}{g^{-1}-1}}\alpha_2,\epsilon^{s-1}\beta_s\rangle$  & $\G_s$ & $s\in\B\setminus\{1,-1\}$,
				& $\frac{(p-1)(q-1)}{2}$\\
				&&&&\\[-1em]
				& & & $0\leq m\leq p-1,\, m\neq -\frac{g}{(g-1)^2} \pmod{p}$ & \\
				&&&&\\[-1em]
				\cline{2-5}
				&&&&\\[-1em]
				& $\overline T_s=\langle \sigma^{\frac{1}{g-1}}\alpha_1, \sigma\tau^{\frac{1}{g^{-1}-1}}\alpha_2,\epsilon^{s-1}\beta_s\rangle$  & $\G_s$ & $s\in\B\setminus\{1,-1\}$
				& $\frac{q-1}{2}$\\ 
				&&&&\\[-1em]
				\hline
				&&&&\\[-1em]
				$\mathcal T_1$    & $\widetilde T_{d,1}=\langle \sigma^{\frac{1}{g-1}}\alpha_1, \tau^{\frac{1}{g^{-1}-1}}\alpha_2,\epsilon^{d}\beta_1\rangle$ & $\G_1$ & $d\in \mathfrak{A}$ & $\frac{q-1}{2}$\\
				&&&&\\[-1em]
				\hline
				&&&&\\[-1em]
				$\mathcal T_{-1}$ & $\widehat T_{c,-1}=\langle \tau\sigma^{\frac{1}{g-1}}\alpha_1, \sigma^n\tau^{\frac{1}{g^{-1}-1}}\alpha_2,\epsilon^{-2}\beta_{-1}\rangle$  & $\G_{-1}$ &  $0\leq n\leq p-1$, $n\neq -\frac{g}{(g-1)^2}  \pmod{p}$  & $p-1$
		\end{tabular}}
	\end{center}
where $\mathfrak A\subseteq\Z_q^\times$ containing one out of $d$ and $-d$ for each $d\in\Z_q^\times$. 
\end{lemma}
\begin{proof}
Using the same argument as in Lemma \ref{lemma G_k pq 1} we can show that the groups in the list are regular. Let us show that the groups with the same image under $\pi_2$ in the table are not conjugate. Let $\delta=h\psi$ for $h\in \aut{\G_{-1}}^+$.
	\begin{itemize}
		\item if $\pi_2(G)=\mathcal{T}_s$ and $s\neq \pm 1$ then $N_{\aut{\G_1}}(\mathcal T_s)=\aut{\G_{-1}}^+$. So the groups in different rows and the groups in the family $\widehat T_{m,s}$ are not pairwise conjugate, since their $p$-Sylow subgroups are not conjugate by elements in $\aut{\G_{-1}}^+$. 
		
		\item If $s=1$, then $\mathcal{T}_1$ is normal in $\aut{\G_{-1}}$ and $$\delta\epsilon^c  \beta_{ 1} \delta^{-1} \H_4=\psi\epsilon^c  \beta_{ 1} \psi \H_4=\epsilon^{ -c}  \beta_{ 1}\H_4.$$
		So, if $\widetilde T _{d,1}$ and $\widetilde T _{c,1}$ are conjugate by $\delta$ then $c=-d$. 
		\item If $s=-1$, then $\mathcal{T}_{-1}$ is normal in $\aut{\G_{-1}}$ and $$\delta\epsilon^c  \beta_{ -1} \delta^{-1} \H_4=\psi\epsilon^{-c}  \beta_{ -1} \psi \H_4=\epsilon^{ -c}  \beta_{ -1}^{-1}\H_4.$$ 
		If $\widetilde T_{c,-1}$ and $\widetilde T_{d,-1}$ are conjugate by $\delta$, then $c=d$. The groups $\widetilde{T}_{c,-1}$ and $\widehat{T}_{n,-1}$ are not conjugate since their $p$-Sylow subgroups are not.
		 
        
	\end{itemize}


Let $G$ be a regular subgroup of $\Hol(\G_{-1})$. According to Lemma \ref{pi_2 pq for k=-1} we can assume that $\pi_2(G)=\mathcal{T}_s$ for $s\in \mathcal{B}$. From the same computations as in the beginning of Lemma \ref{G_k_sub_ppq}, in which we set $k=-1$, we have that 
$$G=\langle \sigma^{\frac{1}{g-1}}\tau^n\alpha_1, \, \sigma^m\tau^{\frac{1}{g^{-1}-1}}\alpha_2, w \epsilon^c\beta_s\rangle$$
for $c\neq 0$ and some $w\in \langle \sigma,\tau\rangle$. So we need to discuss two cases:
\begin{itemize}
    \item $n=m=0$;
    \item $(n,m)\neq (0,0)$, $c=s-1$ and so $s\neq 1$.
\end{itemize}
In the second case, if $n=0$ we can set $m=1$ up to conjugation by $[(0,0),(m^{-1},m^{-1}),0]$, otherwise we can set $n=1$ conjugating by $[(0,0),(n^{-1},n^{-1}),0]$. Then we reduce the cases to $(n,m)=(0,1)$ or $n=1$ and $0\leq m\leq p-1$. If $n=1$ the elements $\tau \sigma^{\frac{1}{g-1}}$ and $\sigma^{m}\tau^{\frac{1}{g^{-1}-1}}$ need to be linearly independent, i.e.
$$det\begin{bmatrix}\frac{1}{g-1} & m\\ 1 & \frac{1}{g^{-1}-1}\\ \end{bmatrix}=-\frac{g}{(g-1)^2}-m\neq 0\pmod{p}.$$ 
In both cases, up to conjugation by an element in $\langle \alpha_1,\alpha_2\rangle$ (using \eqref{conj by alphas} with $k=-1$) we can assume that $w=1$ and so the regular groups are identified by the parameters $(s,c,n,m)$. 

The groups identified by $(1,c,0,0)$ and $(1,-c,0,0)$ are conjugate by $\psi$. The groups identified by $(-1,-2,0,1)$ and $(-1,-2,1,0)$ are conjugate by $[(0,0),(1,-g),0]\psi$. Accordingly, $G$ is conjugate to one of the groups in the table.
\end{proof}
We summarize the content of this subsection in the following table:

\begin{table}[H]
	\centering
	
	\small{    \begin{tabular}{c|c|c|c|c|c}
			$|\ker{\lambda}|$  & $\Z_p^2\times\Z_q$ & $\G_k$ & $\G_0$ & $\G_{-1}$ & $\G_1$ \\
			\hline
			$1$ &- & $p+q-1$ & $p+q-1$ & $p+q-2$ & $\frac{q-1}{2}$ \\
			$p$ & $1$ & $2(q+2)$ & $2(q+1)$ & $q+2$ & $q-1$ \\
			$q$ & $1$ &- &- &- &- \\
			$pq$ &- &- & $2$ &- &- \\
			$p^2$ & $1$ & $q-1$ & $q-1$ & $q-2$ & $\frac{q-1}{2}$ \\
			$p^2q$ &- &- &- & $1$ &- \\
	\end{tabular}}
	\vs
	\caption{Enumeration of skew braces of $\G_{-1}$-type for $p=1\pmod{q}$.}
\end{table}

\subsection{Skew braces of $\G_1$-type}\label{subsection:p=1(q)_G_1}

A presentation of the group $\G_1$ is the following 
$$\G_1=\langle \sigma,\tau,\epsilon\ |\ \sigma^p=\tau^p=\epsilon^q=[\sigma,\tau]=
1,\, \epsilon \sigma\epsilon^{-1}=\sigma^{g},\, \epsilon \tau\epsilon^{-1}=\tau^{g}\rangle.$$

According to \cite[Subsections 4.1, 4.2]{auto_pq}, the mapping
$$\phi:  \mathbb{Z}_p^2\rtimes GL_2(p)\longrightarrow\aut{\G_1},\quad  [(n,m),H]\mapsto h= 	\begin{cases}
h|_{\langle \sigma,\tau\rangle}=H,\\ \epsilon\mapsto   \sigma^n\tau^m \epsilon 
\end{cases}$$
is a group isomorphism.
In particular, $|\aut{\G_{1}}|=	p^3(p-1)^2(p+1)$, a $p$-Sylow subgroup of $\aut{\G_1}$ is generated by $\alpha_1=[(1,0),Id],\,\alpha_2=[(0,1),Id]$, $\gamma=[(0,0),C]$, where $C$ is defined as in Remark \ref{subgroups of GL} and $\alpha_1$ and $\alpha_2$ generate a normal subgroup in $\aut{\G_1}$.

Note that for $k\neq 0,-1$ the map 
$$\iota:\aut{\G_k}\longrightarrow \aut{\G_1},\quad [(n,m),(a,b)]\mapsto \left[(n,m),\begin{bmatrix} a & 0 \\ 0 & b \end{bmatrix}\right]$$ is an injective group homomorphism. Moreover, the formula \eqref{conj by alphas} with $k=1$ holds in $\Hol(\G_1)$.

The procedure to check the regularity of a subgroup has already been established in the previous subsections, so we are not showing that part of the classification strategy explicitly in this subsection.

In the following we are using the formulas
\begin{align}
\alpha_2\gamma\alpha^{-1}\alpha_2^{-1}& =\alpha_1^{-1}\gamma \label{F1},\\
(\tau^a \alpha_2^b\gamma)^n &=\sigma^{\frac{an(n-1)}{2}}\tau^{an}\alpha_1^{\frac{bn(n-1)}{2}}\alpha_2^{bn}\gamma^n \label{F2}
\end{align}
for $a,b,n\in \mathbb Z$.

\begin{lemma}\label{G_1_sub_p}
	A set of representatives of conjugacy classes of regular subgroups $G$ of $\Hol(\G_{1})$ with $|\pi_2(G)|=p$ is
	\begin{equation*}
		H=\langle \epsilon, \tau, \sigma^{\frac{1}{g-1}}\alpha_1\rangle\cong \G_0,\quad W=\langle \epsilon,\sigma, \tau^{\frac{1}{g-1}}\alpha_2\gamma\rangle\cong \G_0.
	\end{equation*}
\end{lemma}

\begin{proof} 
	Up to conjugation the subgroups of order $p$ of $\aut{\G_{1}}$ are $\langle\alpha_1\rangle$, $\langle\alpha_2^i\gamma\rangle$ for $i=1,2$. So $H$ and $W$ are not conjugate.

	If $\pi_2(G)=\langle \alpha_1\rangle$, arguing as in Lemma \ref{G_k_sub_p} we can show that $G$ is conjugate to $H$.
	
	Let $G$ be a regular subgroup of $\Hol(\G_{1})$ with $\pi_2(G)=\langle \alpha_2^i\gamma\rangle$ for $i=0,1$. Then $K=\ker{\pi_2|_G}$ contains an element of order $q$ and so, up to conjugation by a power of $\alpha_1$ we can assume that $K$ is generated by an element of the form $\tau^n \epsilon$ and by some $v\in \langle \sigma,\tau\rangle$. So $G$ has the following standard presentation
	$$G=\langle \tau^n \epsilon, v, u\alpha_2^i \gamma\rangle$$
	where $u,v\in \langle \sigma,\tau\rangle$. The $p$-Sylow subgroup of $K$ is characteristic in $K$ and $K$ is normal in $G$ then $\langle v   \rangle$ is normal in $G$. Thus
	$$w\alpha_2^i \gamma(u)w^{-1}=\gamma(u)\in \langle u\rangle.$$
	Therefore we can assume that $u=\sigma$ and so $u=\tau^c$ for $c\neq 0$.
   According to condition \NC, we have
	\begin{equation*}\label{condition on c}
		\tau^c\alpha_2^i\gamma\tau^n\epsilon\gamma^{-1}\alpha_2^{-i}\tau^{-c}=\sigma^n\tau^{(1-g)c+i+n}\epsilon\in K.	\end{equation*}
	Hence, $(1-g)c+i=0$. If $i=0$ then $c=0$ and so $G$ is not regular. Therefore $i=1$ and $c=\frac{1}{g-1}$. Finally, using \eqref{F1}, we have that $G$ is conjugate to $W$ by $\alpha_2^{-n}\gamma^{-n}$. 
\end{proof}
\begin{lemma}
	A set of representatives of conjugacy classes of regular subgroups $G$ of $\Hol(\G_{1})$ with $|\pi_2(G)|=p^2$ is
	\begin{equation*}
		H_i=\langle \epsilon, \, \sigma^{\frac{1}{g-1}}\alpha_1,\, \tau^{\frac{1}{g-1}}\alpha_2\gamma^i\rangle\cong \mathbb{Z}_p^2\times \mathbb{Z}_q
	\end{equation*}
	for $i=0,1$.
\end{lemma}

\begin{proof}
	The subgroups of $\aut{\G_1}$ of size $p^2$ are $\langle\alpha_1,\alpha_2\rangle$, $\langle \alpha_1,\gamma\rangle$ and $\langle \alpha_1,\alpha_2\gamma\rangle$, up to conjugation. 
	So, the groups in the statement are not conjugate since their image under $\pi_2$ are not.
	
	Let $G$ be a regular subgroup of $\Hol(\G_{1})$ such that $|\pi_2(G)|=p^2$. The kernel of $\pi_2$ is the normal $q$-Sylow subgroup of $G$ and so $G$ is abelian. Moreover, up to conjugation by the normalizer of $\pi_2(G)$ we can assume that it is generated by $\epsilon$.
	
	If $\pi_2(G)=\langle \alpha_1,\alpha_2\rangle$ then $G$ is conjugate to $H_0$ by an element of $\iota(\aut{\G_k})\leq\aut{\G_1}$ (see Lemma \ref{G_k_sub_pp}). Otherwise we have that
	$$G=\langle \epsilon, u\alpha_1,v\alpha_2^i\gamma\rangle$$
	for some $u,v\in \langle \sigma,\tau\rangle$. From abelianness of $G$ it follows that $u=\sigma^{\frac{1}{g-1}}$ and $v=\tau^{\frac{i}{g-1}}$ by \eqref{F1}. Since $v\neq 1$ then $i=1$ and so $G=H_1$.
%
%
%
\end{proof}   
 
Employing the notation in Remark \ref{subgroups of GL} let us define $\beta_s=[(0,0),\dm_{1,s}]$ and $\widetilde{\beta}=[(0,0),\dm_{0,1}]$ for $1\leq s\leq q-1$. 
\begin{proposition}\label{G_{1}_sub_q}
	The skew braces of $\G_1$-type with $|\ker{\lambda}|=p^2$ are $(B_{a,b},+,\circ)$ where $(B_{a,b},+)=\Z_p^2 \rtimes_{\dm_{1,1}} \Z_q$ and	
	$$(B_{a,b},\circ)=\Z_p^2\rtimes_{\dm_{a+1,b+1}} \Z_q\cong \begin{cases}
	\Z_p^2\times \Z_q, \, \text{ if } a=b=q-1,\\
	\G_0,\, \text{ if } a=q-1, \, b\neq q-1,\\
	\G_{\frac{b+1}{a+1}},\, \text{otherwise, }\\
	\end{cases}$$
	for $0\leq a,b\leq q-1$ and $(a,b)\neq (0,0)$. Moreover, they are bi-skew braces and $B_{a,b}\cong B_{c,d}$ if and only if $(c,d)=(b,a)$ and so there are $\frac{(q-1)(q+2)}{2}$ such skew braces.
\end{proposition}

\begin{proof}
Let $G$ be a regular subgroup of $\Hol(\G_1)$ such that $|\pi_2(G)|=q$. By Remark \ref{subgroups of GL} we have that, up to conjugation, an element of order $q$ of $GL_2(p)$ is a diagonal matrix with entries a power of $g$. Arguing as in Proposition \ref{G_k_sub_q}, also in this case we can show that every such group is conjugate to a subgroup of the form
	$$G_{a,b}=\langle \sigma,\tau, \epsilon \beta_0^a\widetilde\beta^b\rangle.$$
	Hence $G_{a,b}$ and $G_{c,d}$ are conjugate by $h\in \aut{\G_1}$ if and only if 
	$$h(\epsilon)h \beta_0^a\widetilde \beta^b h^{-1}=\epsilon h \beta_0^a \widetilde \beta^b h^{-1}=\epsilon \beta_0^c\widetilde \beta^d\pmod{\langle \sigma,\tau\rangle},$$ i.e. $\beta_0^a\widetilde \beta^b|_{\langle \sigma,\tau\rangle}=\dm_{a,b}$ and $\beta_0^c \widetilde \beta^d|_{\langle\sigma,\tau\rangle}=\dm_{c,d}$ are conjugate. And so either $(a,b)=(c,d)$ or $(a,b)=(d,c)$. In particular there are $\frac{(q-1)(q+2)}{2}$ such classes. The other claims follow as in Proposition \ref{G_k_sub_q}.
\end{proof}

The following lemma collects the conjugacy classes of subgroups of size $pq$ and $p^2q$ of $\aut{\G_1}$. 

\begin{lemma}
	A set of representatives of conjugacy classes of subgroups of size $pq$ and $p^2q$ of $\aut{\G_1}$ is
	\bigskip
	\begin{center}
		\small{
			\begin{tabular}{c|l|c| l}\label{subgroups of size pq of aut G_1}
				Size &    $G$ & Parameters & Class\\
				\hline
				&&&\\[-1em]
				$pq$ &      $\mathcal H_{s}=\langle \alpha_1,\beta_s\rangle$ & $0\leq s\leq q-1$ & $\Z_p\rtimes_g \Z_q$\\
				&&&\\[-1em]
				\cline{2-4}
				&&&\\[-1em]
				&           $\mathcal U=\langle \alpha_2,\beta_0\rangle$ & - & $\mathbb{Z}_{pq}$\\ 
				&&&\\[-1em]
				\cline{2-4}
				&&&\\[-1em]
				&     $\mathcal K_s=\langle \gamma,\beta_s\rangle$ & $0\leq s\leq q-1$ &  $\Z_{pq}$, if $s=1$,\\
				& & &  $\Z_p\rtimes_g \Z_q$, otherwise \\
				&&&\\[-1em]
				\cline{2-4}
				&&&\\[-1em]
				&   $\mathcal M=\langle \alpha_1\alpha_2^{2 }\gamma,\beta_{2^{-1}}\rangle$ & - & $\Z_p\rtimes_g \Z_q$\\
				&&&\\[-1em]
				\cline{2-4}
				&&&\\[-1em]
				& $\mathcal V=\langle \gamma,\widetilde{\beta}\rangle$ & - & $\Z_p\rtimes_g \Z_q$\\
				&&&\\[-1em]
				\hline
				&&&\\[-1em]
				$p^2q$ &  $\mathcal T_{s}=\langle \alpha_1,\alpha_2,\beta_s\rangle$ & $s\in \B$  & $\G_s$\\
				&&&\\[-1em]
				\cline{2-4}
				&&&\\[-1em]
				& $\mathcal R_{s}=\langle \alpha_1,\gamma,\beta_s\rangle$ & $0\leq s\leq q-1$  & $\G_{1-s}$ \\
				&&&\\[-1em]
				\cline{2-4}
				&&&\\[-1em]
				& $\mathcal N=\langle \alpha_1,\gamma,\widetilde{\beta}\rangle$ & - & $\G_{0}$ \\
				&&&\\[-1em]
				\cline{2-4}
				&&&\\[-1em]
				&     $\mathcal L=\langle \alpha_1,\alpha_2\gamma,\beta_{2^{-1}}\rangle$ & - & $\G_{2}$
		\end{tabular}}
	\end{center}
	where $\B$ is as in Remark \ref{subgroups of GL}.
\end{lemma}

\begin{proof}
The groups in the table are not pairwise conjugate. Indeed if $h=[(n,m),M]$ conjugate a pair of such groups then their $p$-Sylow subgroups are conjugate by $h$ and their restrictions to $\langle \sigma,\tau\rangle$ are conjugate by $M$. A case-by-case discussion shows that these two conditions can not be satisfied by any pair of groups in the table.

	

	Let $H$ be a subgroup of order $pq$ or $p^2q$ of $\aut{\G_1}$ and let $H|_{\langle \sigma,\tau\rangle}$ denotes the restriction of the action of $H$ to ${\langle \sigma,\tau\rangle}$. If $H|_{\langle \sigma,\tau\rangle}$ has size $q$ then, according to Remark \ref{subgroups of GL} we have that $H|_{\langle \sigma,\tau\rangle}$ is generated by $\dm_{1,s}$ for $s\in \mathcal{B}$. Hence, up to conjugation, $H\leq \iota(\aut{\G_k})$. Using the same notation as in Lemma \ref{pi_2 pq G_k}, we have that $H$ is conjugate to one of the groups $\setof{\mathcal{H}_{1,s},\mathcal{H}_{2,s},\mathcal{W},\mathcal{T}_s}{s\in \mathcal{B}}$ by an element of $\iota(\aut{\G_k})$. Moreover: 
	\begin{itemize}
	   
 	    \item if $s\neq 0, \pm 1$, $\psi \mathcal{H}_{1,s}\psi= \mathcal{H}_{2,{s}^{-1}}$, where $\psi$ is the automorphism swapping $\sigma$ and $\tau$.
	   
	   	    \item The group $\mathcal{W}$ is conjugate to $\mathcal{H}_{1,1}$ by a suitable automorphism of the form $[(0,0),M]$, since $\dm_{1,1}$ is central in $GL_2(p)$. 
	\end{itemize}
	Hence $H$ is conjugate to one of the groups $\setof{\mathcal{H}_{s},\,\mathcal{U}}{0\leq s\leq q-1}$ if $|H|=pq$ or to one of $\setof{\mathcal{T}_{s}}{0\leq s\leq q-1}$ if $|H|=p^2q$ as in the statement.
	%
	
	If $H|_{\langle \sigma,\tau\rangle}$ has size $pq$ then, according to Remark \ref{subgroups of GL}, up to conjugation we have that $H|_{\langle\sigma,\tau\rangle}=\langle C, \dm_{1,s}\rangle$ or $H|_{\langle\sigma,\tau\rangle}=\langle C, \dm_{0,1}\rangle$ for $0\leq s\leq q-1$.
	Therefore 
	$$H=\langle \alpha_1^n\alpha_2^m \gamma,\ \alpha_1^r \alpha_2^t\theta\rangle$$
	for $\theta\in \{\beta_{s}, \widetilde{\beta} : 0\leq s\leq q-1\}$. If $\theta=\beta_{0}$ then $t=0$ and if $\theta=\widetilde \beta$ then $r=0$. Using \eqref{conj by alphas} for $k=1$, up to conjugation we can assume that $r=t=0$. The $p$-Sylow subgroup of $H$ has to be normal, i.e. if $\theta=[(0,0),\dm_{t,s}]$, using \eqref{F2} we have
	\begin{equation*}
		\theta\alpha_1^n\alpha_2^m\gamma\theta^{-1} = \alpha_1^{n g^t }\alpha_2^{m g^s }\gamma^{g^{t-s}}=(\alpha_1^n\alpha_2^m\gamma)^r=\alpha_1^{nr+mr\frac{r-1}{2}}\alpha_2^{mr}\gamma^{r}
	\end{equation*}
for some $r\in \mathbb{Z}$. Thus, $r=g^{s-t}$ and accordingly $n,m$ solve the following system of linear equations:
	\begin{equation*}
		\begin{cases}
 2(g^{t-s}-g^t)n+g^{t-s}(g^{t-s}-1)m=0\\
			(g^{t-s}-g^s)m=0.
		\end{cases}
	\end{equation*}
	Hence: 
	\begin{itemize}
		\item  if $\theta=\beta_{s},\widetilde \beta$ with $s\neq 0, 2^{-1}$ then $n=m=0$, i.e. $H=\mathcal K_s$ or $H=\mathcal V$. 
		\item  If $\theta=\beta_{2^{-1}}$ then $m=2n$ and so, either $n=m=0$ or up to conjugation by $[(0,0),n^{-1}Id]$, $n=1, m=2$, i.e. $H=\mathcal K_{2^{-1}}$ or $H$ is conjugate to $\mathcal M$.
		\item   If $\theta=\beta_{0}$ then either $n=m=0$ or, up to conjugation by $[(0,0),n^{-1}Id]$, $m=0$ and $n=1$. Since $\alpha_2\langle \alpha_1 \gamma, \beta_0\rangle \alpha_2^{-1}=\mathcal K_0$ then $H$ is conjugate to $\mathcal K_0$.
	\end{itemize}
	
		Let $H$ be a subgroup of $\aut{\G_1}$ of size $p^2q$ and let $H|_{\langle \sigma,\tau\rangle}$ has size $pq$. Then, up to conjugation, 
	$$H=\langle \alpha_1^n\alpha_2^m,\ \alpha_1^r\alpha_2^t\gamma,\ \theta\rangle$$
	where $\theta$ is either $\beta_s$ or $\widetilde{\beta}$. Using \eqref{F1} to check the abelianness of the $p$-Sylow subgroup of $H$ it follows that $m=0$ and therefore we can assume that $n=1$ and $r=0$.
	From the normality of the $p$-Sylow subgroup it follows that:
	\begin{itemize}
		\item if $\theta=\beta_s$ for $s\neq 2^{-1}$ then
		$t(g^s-g^{1-s})=0$ and so $t=0$, i.e. $H=\mathcal R_s$. 
		\item If $\theta=\widetilde{\beta}$ then $(1-g)t=0$ and so $t=0$, i.e. $H=\mathcal N$.  
		\item If $\theta= \beta_{2^{-1}}$ then up to conjugation by an element of the form $[(0,0),xId]$ we have $t\in \{0,1\}$, i.e. $H$ is conjugate either to $\mathcal R_{2^{-1}}$ or  to $\mathcal L$. \qedhere
	\end{itemize}
\end{proof} 

The group 
\begin{equation}\label{H_5}
    \H_5=\langle \sigma,\tau,\alpha_1,\alpha_2\rangle
\end{equation}
is a normal subgroup of $\Hol(\G_1)$ and $\Hol(\G_1)/\H_5\cong \Z_q\times GL_2(p)$.
\begin{lemma}\label{G_1 pq}
	A set of representatives of conjugacy classes of regular subgroups $G$ of $\Hol(\G_{1})$ with $|\pi_2(G)|=pq$ is 
	
	\bigskip
	\begin{center}
		\small{
			\begin{tabular}{c|l|c|l|c}
				$\pi_2(G)$ & Subgroups &  Parameters &Class & \#\\
				\hline
				&&&\\[-1em]
				
				$\mathcal H_{s}$ & $\widetilde H_{s,c}=\langle \tau, \sigma^{\frac{1}{g-1}}\alpha_1,\epsilon^{c}\beta_s\rangle$    &  $0\leq s\leq q-1$, & $\G_{c+s}$
				& $q(q-1)$\\
				&&&\\[-1em]
				& &  $1\leq c\leq q-1$ & &\\
				&&&\\[-1em]
				\cline{2-5}
				&&&\\[-1em]
				&    $\widehat H_{s}=\langle \sigma, \tau\alpha_1,\epsilon^{1-s}\beta_s\rangle$  & $2\leq s\leq q$& $\G_{2-s}$ & $q-1$\\
				&&&\\[-1em]
				\hline
				&&&\\[-1em]
				$\mathcal K_s$ & $\widetilde K_s=\langle \sigma, \tau\gamma, \epsilon^{1-2s}\beta_s \rangle$ & $0\leq s\leq q-1$, & $\mathbb{Z}_p^2\times \mathbb{Z}_q$, if $s=1$,  &     $q-1$ \\
				&&&\\[-1em]
				& &  $s\neq 2^{-1}$ & $\G_2$, otherwise.  &\\
				&&&\\[-1em]
				\hline
				&&&\\[-1em]
				$\mathcal U$ &   $\widetilde U_c=\langle \sigma, \tau^{\frac{1}{g-1}}\alpha_2, \epsilon^c\beta_0 \rangle$  &  $1\leq c\leq q-1$  & $\mathbb{Z}_p^2\times \mathbb{Z}_q$, if $c=-1$,  &   $q-1$ \\
				&&&\\[-1em]
				& & & $\G_0$, otherwise.  &\\
				&&&\\[-1em]
				\cline{2-5}
				&&&\\[-1em]
				& $\widetilde{U}=\langle \tau, \sigma\alpha_2, \epsilon^{-1}\beta_0 \rangle$ & - & $\G_0$ & $1$ \\
				&&&\\[-1em]
				\hline
				&&&\\[-1em]
				$\mathcal M$  & $\widetilde M_{c}=\langle \sigma,\tau^{\frac{2}{g-1}} \alpha_1\alpha_2^2\gamma,\epsilon^c\beta_{2^{-1}}\rangle $   & $1\leq c\leq q-1$ & $\G_{2(c+1)}$   &  $q-1$ \\
				&&&\\[-1em]
				\hline
				&&&\\[-1em]
				$\mathcal V$ & $\widetilde V=\langle \sigma, \tau\gamma, \epsilon^{-2}\widetilde{\beta} \rangle$   & - & $\G_2$ &   $1$
		\end{tabular}}
	\end{center}
	
\end{lemma}

\begin{proof}	
The groups with the same image under $\pi_2$ but in different rows of the table are not pairwise conjugate because their $p$-Sylow subgroups are not conjugate by the normalizer of their image under $\pi_2$. The same argument used in Lemma \ref{lemma G_k pq 1} shows that the groups in the same row are not pairwise conjugate (employing the group $\H_5$).
%

Let $G$ be a regular subgroup of $\Hol(\G_1)$ with $\pi_2(G)\in \setof{\mathcal{H}_s, \, \mathcal{U}}{0\leq s\leq q-1}$, i.e. with $\pi_2(G)\leq \iota(\aut{\G_k})$. Then $G\leq \G_{1}\rtimes \iota(\aut{\G_k})$ and we can assume that 
$$G=\langle u, v\alpha_i, w\epsilon^c\beta_s\rangle$$
for $c\neq 0$ and $u,v,w\in \langle \sigma,\tau\rangle$. Up to conjugation by an element in the normalizer of $\pi_2(G)$ we can assume that $u\in \{\sigma,\tau,\sigma\tau\}$.

If $\ker{\pi_2}=\langle \sigma\tau\rangle$, the condition \NC forces $s=1$ and $i=1$. In such case, up to conjugation by $h\in N_{\aut{\G_1}}(\mathcal{H}_1)$, we can assume that the kernel is generated by $\sigma$. 

Therefore, setting $k=1$ in the proof of Lemma \ref{lemma G_k pq 1}, it follows that $G$ is conjugate to one of the groups in the table corresponding to the same image under $\pi_2$ by an element of $\iota(\aut{\G_k)}$ (we are using that \eqref{conj by alphas} and \eqref{formula for powers} holds for $k=1)$.

	In all the other cases we can assume that $G$ has the form 
	$$G=\langle u,v \alpha_1^n\alpha_2^m\gamma, w\epsilon^c\theta\rangle$$
	for $c\neq 0$ and $u,v,w\in \langle \sigma,\tau\rangle$. From condition \NC we have that $\ker{\pi_2|_G}=\langle \sigma\rangle$ and so $v,w\in \langle \tau\rangle$. 
	
	Let $\pi_2(G)=\mathcal{K}_s$. Then	
	\[
	G=\langle \sigma,\, \tau^a\gamma,\, \tau^b\epsilon^c\beta_s \rangle
	\]
	for $a\neq 0$ and so up to conjugation by $h=[(0,0),a^{-1}Id]$ we can assume that $a=1$. In $\mathcal{K}_s$ the relation $\beta_s\gamma\beta_s^{-1}=\gamma^{g^{1-s}}$ holds. Then the condition \nc and formula \eqref{F2} imply that 
	$$(\tau^b\epsilon^c\beta_s)\tau\gamma(\tau^b\epsilon^c\beta_s)^{-1}  =\tau^{g^{s+c}}\gamma^{g^{1-s}} \overset{\nc}{=}(\tau \gamma)^{g^{1-s}}=\tau^{g^{1-s}}\gamma^{g^{1-s}} \pmod{\langle \sigma\rangle}$$
	and so $c=1-2s$ and then $s\neq 2^{-1}$.  Since $(\tau^b\epsilon^{1-2s}\beta_s)^q\in \ker{\pi_2}$, if $s=1$ then $b=0$ (see \eqref{formula for powers}). Otherwise, $G$ is conjugate to $\widetilde K_s$ by $\gamma^n$ where $n=\frac{b}{1-g^{1-s}}$.
	
	The other cases are similar and so we omit the computations.
\end{proof}

\begin{lemma}
	A set of representatives of conjugacy classes of regular subgroups $G$ of $\Hol(\G_{1})$ with $|\pi_2(G)|=p^2q$ is 
	
	\bigskip
	\begin{center}
		\small{
			\begin{tabular}{c|l|c|c|c}
				$\pi_2(G)$ & Subgroups & Parameters  & Class & \#\\
				\hline
				&&&\\[-1em]
				$\mathcal  T_{s}$    &    $\widetilde T_{s,d}=\langle \sigma^{\frac{1}{g-1}}\alpha_1, \tau^{\frac{1}{g-1}}\alpha_2,\epsilon^{d}\beta_s\rangle$  & $s\in\B\setminus\{-1\}$,  & $\G_{s}$& $\frac{(q-1)(q+1)}{2}$\\
				&&&\\[-1em]
				& &  $1\leq d\leq q-1$ & & \\
				&&&\\[-1em]
				\cline{2-5}
				&&&\\[-1em]
				&    $\widetilde T_{-1,d}=\langle \sigma^{\frac{1}{g-1}}\alpha_1, \tau^{\frac{1}{g-1}}\alpha_2,\epsilon^{d}\beta_{-1}\rangle$ &  $d\in \mathfrak{A}$ &$\G_{-1}$ & $\frac{q-1}{2}$\\
				&&&\\[-1em]
				\cline{2-5}
				&&&\\[-1em]
				& $\widehat T_{s}=\langle \sigma^{\frac{1}{g-1}}\alpha_1, \sigma\tau^{\frac{1}{g-1}}\alpha_2,\epsilon^{s-1}\beta_s\rangle$  & $s\in\B\setminus\{1\}$ & $\G_s$ &   $\frac{q+1}{2}$\\
				&&&\\[-1em]
				\cline{2-5}
				&&&\\[-1em]
				& $\overline T_{s}=\langle \tau\sigma^{\frac{1}{g-1}}\alpha_1, \tau^{\frac{1}{g-1}}\alpha_2,\epsilon^{1-s}\beta_s\rangle$  &  $s\in\B\setminus\{1,-1\}$ & $\G_s$ & $\frac{q-1}{2}$\\
				&&&\\[-1em]
				
				
				\hline
				&&&\\[-1em]
				$\mathcal R_s$    & $\widetilde R_{s}=\langle \sigma^{\frac{1}{g-1}}\alpha_1, \tau\gamma,\tau^{\frac{1-g^{1-s}}{2}}\epsilon^{1-2s}\beta_s\rangle$ &  $0\leq s \leq q-1$, $s\neq 2^{-1}$ &$\G_{1-s}$ & $q-1$\\
				&&&\\[-1em]
				\hline
				&&&\\[-1em]
				$\mathcal N$ & $\widetilde N=\langle \sigma^{\frac{1}{g-1}}\alpha_1,\tau\gamma,\tau^{\frac{1-g^{-1}}{2}} \epsilon^{-2}\widetilde{\beta}\rangle$  &  -  & $\G_{0}$ & $1$\\ 
				&&&\\[-1em]
				\hline
				&&&\\[-1em]
				$\mathcal L$ & $\widetilde L_d=\langle \sigma^{\frac{1}{g-1}}\alpha_1,\tau^{\frac{1}{g-1}}\alpha_2\gamma,\epsilon^{d}\beta_{2^{-1}}\rangle$  &   $1\leq d\leq q-1$  & $\G_{2}$ &$q-1$
		\end{tabular}}
	\end{center}
where $\mathfrak{A}$ is as in Lemma \ref{label}.
\end{lemma}
\begin{proof}
	Let us show that the groups in the table are not pairwise conjugate:
	\begin{itemize}
	\item If $s\neq -1$, if $\widehat T_{s,d}$ (resp. $\widetilde L_{d}$) and $\widehat T_{s,c}$ (resp. $\widetilde L_{c}$) are conjugate by an element in the normalizer of their image under $\pi_2$ then, by looking at their images in $\Hol(\G_1)/\H_5$ it follows that $c=d$ (and the same argument shows that the groups $\widehat T_s$ and $\overline T_s$ for $s\neq -1$ and the groups $\widetilde L_d$ are not pairwise conjugate). If $s=-1$ the same argument shows that $c=-d$.
	
		\item If $\pi_2(G)= \mathcal{T}_{s}$ for $s\neq \pm 1$ then the action of $h\in N_{\aut{\G_1}}(\pi_2(G))$ restricted to $\langle \sigma,\tau\rangle$ is a diagonal matrix. Therefore the groups $\widetilde T_{s,d}$, $\widehat{T}_s$ and $\overline T_s$ are not conjugate, since their $p$-Sylow subgroups are not. 
	
		\item If $s\neq -1$ and $\widetilde T_{s,d}$ and $\widetilde T_{s,c}$ are conjugate then, by looking at their images in $\Hol(\G_1)/\H_5$ it follows that $c=d$ (and the same argument shows that the groups $\widehat T_s$ and $\overline T_s$ for $s\neq -1$ and the groups $\widetilde L_d$ are not pairwise conjugate).

		\item If $\pi_2(G)= \mathcal{T}_{-1}$ then the action of $h\in N_{\aut{\G_1}}(\pi_2(G))$ restricted to $\langle \sigma,\tau\rangle$ is either a diagonal matrix or it has the form
		$$h|_{\langle \sigma,\tau\rangle}=\begin{bmatrix} 0 & b \\ a& 0 \end{bmatrix}.$$
		The groups $\widehat T_{-1}$ and $\overline T_{-1}$ and the groups $\widetilde T_{-1,c}$  and $\widetilde T_{-1,-c}$ are conjugate by the automorphism $\psi$ swapping $\sigma$ and $\tau$. 
		The other groups are not conjugate, since their $p$-Sylow subgroups are not. 
	\end{itemize}
		
	Let $G$ be a regular subgroup and let $\pi_2(G)=\mathcal{T}_s\leq \iota(\aut{\G_k})$ for $s\in \mathcal{B}$. Then we can argue as in Lemma \ref{G_k_sub_ppq} to show that $G$ is conjugate to one of the groups in the table (formulas \eqref{conj by alphas} and \eqref{formula for powers} hold for $k=1$).	
		
		Otherwise, as in Lemma \ref{G_1 pq} we have that $G$ has the form
		$$G=\langle \sigma^{\frac{1}{g-1}} \alpha_1 ,\ u   \alpha_2^i \gamma,\ w\epsilon^c \theta\rangle$$
		for $c\neq 0$, $i=0,1$ and $\theta\in\setof{\beta_s, \, \widetilde\beta}{0\leq s\leq q-1}$.
\begin{itemize}
    \item If $\pi_2(G)=\mathcal{R}_s$, up to conjugation by a power of $\alpha_1$ we can assume that $w\in \langle \tau\rangle$ and that $u\in \langle \tau\rangle$ up to conjugation by a power of $\gamma$. Conjugating by an automorphism of the form $[(0,0),xId]$ we can assume that $u=\tau$. Since 
		$$\beta_{s}\gamma \beta_{s}^{-1}=(\gamma)^{g^{1-s}}$$
	 the condition \nc implies that the last two generators of $G$ satisfy the same relation. Accordingly $c=1-2s$ and $w=\tau^{\frac{1-g^{1-s}}{2}}$. Therefore $G$ is conjugate to $\widetilde {R}_s$. If $\pi_2(G)=\mathcal{N}$ we can similarly conclude that $G$ is conjugate to $\widetilde N$.
    
    \item If $\pi_2(G)=\mathcal{L}$, up to conjugation by an element of the form $$h=   \left[(0,0),\begin{bmatrix} x^2 & y \\ 0 & x \end{bmatrix}\right]$$ we can assume that $u\in \langle \tau\rangle$. Using \eqref{F2} we have that
		$$\beta_{2^{-1}}\alpha_2\gamma (\beta_{2^{-1}})^{-1}=\alpha_1^{-g^{2^{-1}}\frac{g^{2^{-1}}-1}{2}}(\alpha_2 \gamma)^{g^{2^{-1}}},$$
		so, as in the previous case, condition \nc implies that $w\in \langle \tau\rangle$ and $u=\tau^{\frac{1}{g-1}}$. Finally $G$ is conjugate to $\widetilde{L}_c$ by a suitable power of $\alpha_1$.\qedhere
\end{itemize}

\end{proof}

The following tables summarize the content of this subsection.

\begin{table}[H]
	\centering
	
	\small{
		\begin{tabular}[t]{c|c|c|c|c|c|c}
			$\ker{\lambda}$ & $\Z_p^2\times\Z_q$ &  $\G_k$, $k\neq 0,\pm 1, 2$ & $\G_2$ & $\G_0$ & $\G_{-1}$ & $\G_1$ \\
			\hline
			$1$ & - & $q+3$ & $2q+1$ & $q+3$ & $\frac{q+3}{2}$  & $q$\\
			$p$ & $2$ & $2(q+1)$ & $3q$ & $2q$ & $q+1$ & $q$  \\
			$q$ & $2$ & - &- &- &- &-  \\
			$pq$ &- &- & - & $2$ &- &- \\
			$p^2$ & $1$ & $q-1$ & $q-1$ & $q-1$ & $\frac{q-1}{2}$ & $q-2$ \\
			$p^2q$ &- &- & - &- &- & $1$ \\
		\end{tabular}
		\qquad
		\begin{tabular}[t]{c|c|c|c|c}
			$\ker{\lambda}$ & $\Z_p^2\times\Z_3$ &  $\G_0$ & $\G_{-1}$ & $\G_1$ \\
			\hline
			$1$ & - & $6$ & $4$  & $3$ \\
			$p$ & $2$ & $6$ & $5$ & $3$  \\
			$3$ & $2$ &- &- &-  \\
			$3p$ &- & $2$ &- &- \\
			$p^2$ & $1$ & $2$ & $1$ & $1$  \\
			$3p^2$ &-  &- &- & $1$
	\end{tabular}}
	\vs
	\caption{Enumeration of skew braces of $\G_1$-type for $p=1\pmod{q}$: the first table assumes $q>3$ and the second table assumes $q=3$.}
	\label{table:G_1}
\end{table}

\section{Skew braces of size $p^2q$ with $p=-1 \pmod{q}$}\label{section:p=-1(q)}

In this section we assume that $p$ and $q$ are odd primes such that $p=-1\pmod{q}$ (recall that we are omitting the case $q=2$ and $p^2q=12$) and that $\h(x)=x^2+\k x+1$ is an irreducible polynomial over $\Z_p$ such that its companion matrix 
\begin{equation}\label{def of F}
F=\begin{bmatrix}
0 & -1 \\
1 & -\xi
\end{bmatrix}    
\end{equation}
has order $q$. 

According to \cite[Proposition 21.17]{EnumerationGroups}, the unique non-abelian group of size $p^2q$ is:
$$\G_F=\langle \sigma,\tau,\epsilon\,|\, \sigma^p=\tau^p=\epsilon^q=1,\, \epsilon \sigma \epsilon^{-1}=\tau,\, \epsilon \tau \epsilon^{-1}=\sigma^{-1}\tau^{-\k}\rangle \cong \mathbb{Z}_p^2\rtimes_F \mathbb{Z}_q.$$
An automorphism of $\G_F$ is determined by its image on the generators, i.e. by its restriction to $\langle\sigma, \tau\rangle$ given by a matrix and by the image on $\epsilon$.
According to \cite[subsections 4.1, 4.4]{auto_pq}, the map
$$\phi:\mathbb{Z}_p^2\rtimes_{\rho} (N_{GL_2(p)}(F))\longrightarrow \aut{\G_F}, \quad [(n,m),M]\mapsto 	
h_{\pm} = 	\begin{cases}
h|_{\langle \sigma,\tau\rangle}=M_{\pm }\\
\epsilon \mapsto  \sigma^n\tau^m \epsilon^{\pm 1} 
\end{cases}
$$
where $\phi([(n,m),M])=h_+$ if $M\in C_{GL_2(p)}(F)$ and $h_-$ otherwise, is a group isomorphism. The form of $M_{\pm}=h_{\pm}|_{\langle \sigma,\tau\rangle}$ is the following:
\begin{equation}\label{normalizer of F}
M_+=\begin{bmatrix}
x & -y \\
y & x-\xi y
\end{bmatrix},\quad 
M_-=\begin{bmatrix}
x & y-\xi x\\
y & -x
\end{bmatrix}
\end{equation}
where $x,y\in \Z_p$ and $x^2+y^2-\xi xy\neq 0$.
The $p$-Sylow subgroup of $\G_F$ is characteristic, so the map
$$\nu:\aut{\G_F}\longrightarrow \aut{\G_F/\langle \sigma,\tau\rangle},\quad h_{\pm}\mapsto \pm 1$$
is a group homomorphism. The kernel of $\nu$ is $\aut{\G_F}^+=\Z_p^2\rtimes C_{GL_2(p)}(F)$ and it contains the $p$-Sylow subgroup of $\aut{\G_F}$, generated by $\alpha_1=[(1,0),Id]$ and $\alpha_2=[(0,1),Id]$, and the elements of odd order of $\aut{\G_F}$.

\begin{proposition}
	Let $G$ be a regular subgroup of $\Hol(\G_F)$. Then $|\pi_2(G)|\neq p, pq$.
\end{proposition}
\begin{proof}
Let $G$ be a regular subgroup of $\Hol(\G_F)$. The group $\G_F$ has no subgroups of order $pq$ and then $|\pi_2(G)|\neq p$. The group $\G_F$ is the unique non-abelian group of size $p^2q$ and  has no normal subgroup of order $p$. If $\pi_2(G)=pq$ then $G$ has a normal subgroup of order $p$ and so it has to be abelian. Accordingly, $\pi_2(G)$ is an abelian subgroup of order $pq$ of $\aut{\G_F}$. On the other hand, $\aut{\G_F}$ has no such subgroup, contradiction.
%
%
	%
	%
	%
	%
\end{proof}
The subgroup
\begin{equation}\label{sub H p=-1 (q)}
\H_F=\langle \sigma,\tau, \alpha_1,\alpha_2\rangle
\end{equation}
is normal in $\Hol(\G_F)$. 

\begin{lemma}\label{G_A_sub_q}
	The skew braces of $\G_F$-type with $|\ker{\lambda}|=p^2$ are $(B_a,+,\circ)$ where $(B_a,+)=\Z_p^2\rtimes _F \Z_q$ and 
	$$(B_a,\circ)=\Z_p^2\rtimes _{F^{\frac{a+1}{a}}} \Z_q\cong \begin{cases}\mathbb{Z}_p^2\times \mathbb{Z}_q,\,\text{ if }a=q-1\\
	\G_F,\, \text{ otherwise}\end{cases}$$
	for $1\leq a\leq q-1$. In particular they are bi-skew braces.
\end{lemma}


\begin{proof}
	The groups 
	$$G_{a}=\langle \sigma,\tau,\epsilon^a f\rangle$$
	where $f=[(0,0),F]$ have size $p^2q$ and they are regular. Let $\H_F$ be as in \eqref{sub H p=-1 (q)} and assume that $G_a$ and $G_b$ are conjugate by $h$, then 
	$$h(\epsilon)^a h f h^{-1}\H_F=\epsilon^{\pm a} f^{\pm 1}\H_F=(\epsilon ^b f)^n \H_F$$ 
	for some $n$. So $n=\pm 1$ and $a=b$.
	
	Let $G$ be a regular subgroup of $\Hol(\G_F)$ such that $|\pi_2(G)|=q$. Up to conjugation, we can assume that $\pi_2(G)$ is generated by $f$. The kernel of $\pi_2$ is the $p$-Sylow subgroup of $\G_F$ and then we can assume that
	$$G=\langle \sigma,\tau, \epsilon^a f\rangle$$
	where $a\neq 0$. Arguing as in Proposition \ref{G_k_sub_q} we can provide the structure of the skew braces associated to the group $G_a$.
\end{proof}

\begin{lemma}
	There exists a unique conjugacy class of regular subgroups $G$ of $\Hol(\G_F)$ with $|\pi_2(G)|=p^2$. A representative is 
	$$H=\langle \epsilon, u\alpha_1,w\alpha_2\rangle\cong\mathbb{Z}_p^2 \times\mathbb{Z}_q$$
	where $u=(F-1)^{-1}(\sigma)$ and $w=(F-1)^{-1}(\tau)$. 
\end{lemma}

\begin{proof}
	The group $H$ has the desired properties. Assume that $G$ is a regular subgroup of $\Hol(\G_F)$ with $|\pi_2(G)|=p^2$.  Then the image of $\pi_2$ is the normal $p$-Sylow subgroup of $\aut{\G_F}^+$ generated by $\alpha_1$ and $\alpha_2$. Up to conjugation, we can assume that the kernel is generated by $\epsilon$ and so we have
	$G=\langle \epsilon, u \alpha_1,w \alpha_2\rangle.$
	The kernel of $\pi_2$ is a normal subgroup of size $q$ of $G$, and so $G$ is abelian. Thus, by abelianness, it follows that $u=(F-1)^{-1}(\sigma)$ and $w=(F-1)^{-1}(\tau)$. 
\end{proof}

\begin{remark}\label{rem on F} 
    
Let us collect some basic properties of the linear mapping defined by the matrix \eqref{def of F}.

\begin{itemize}
    \item[(i)] Let $W$ be a polynomial. Then
\begin{equation}\label{solutions of h(F)=0} 
\ker{(W(F))}\neq 0 \, \text{ if and only if } \, W(F)=0.
\end{equation}
Indeed $\ker(W(F))$ is an $F$-invariant subgroup and therefore it is either $0$ or $\Z_p^2$. Therefore, the polynomial algebra generated by $F$ is a field and moreover, if $W_1$ and $W_2$ are polynomials, then $W_1(F)=W_2(F)$ if and only if $W_1(F)(x)=W_2(F)(x)$ for some $x\neq 0$.

\item[(ii)] Let $n\in \mathbb N$. Since 
\begin{eqnarray*}
    \h(F^n)-\h(F)&=& F^{2n}-\xi F^n+1-(F^2+\xi F+1)\\
    &=&(F^n-F)(F^n+\underbrace{F+\xi}_{=-F^{-1}})=(F^n-F)(F^n-F^{-1})
\end{eqnarray*}
and the matrix $F$ satisfies the equation $\h(F)=F^2+\xi F+1=0$, we have that $\h(F^n)=0$ if and only if $n=\pm 1\pmod q$.

\item[(iii)]  The automorphism $F$ acts irreducibly on $\Z_p^2$ and so 
\begin{equation}\label{basis x F(x)}
\langle x,\, F(x)\rangle=\Z_p^2
\end{equation}
 for every $x\neq 0$.
 
\end{itemize}
\end{remark}

Let us define the map $\Psi$ by setting
$$\Psi:\Z_p^2\longrightarrow \Z_p, \quad (x,y)\mapsto x^2+y^2-x+y-\k xy.$$
It is easy to check that $\Psi$ is surjective. 

\begin{lemma}\label{level of psi}
    Let $v\in \langle \sigma,\tau\rangle$, $\widetilde{v}=F(v)-(1+F)\sigma$ and $S_v=\langle \widetilde{v}\alpha_1, v \alpha_2\rangle\leq \Hol(\G_F).$ Then
     $$|N_{\Hol(\G_F)}(S_v)\cap \aut{\G_F}|=\begin{cases} 2p^2(p^2-1)
     , \, \text{ if } v=(\frac{1}{\xi+2},-\frac{1}{\xi+2}),\\
    2p^2(p-1), \, \text{ otherwise.}\end{cases}
    $$
    Moreover, if $S_v$ and $S_w$ are conjugate by an element of $\aut{\G_F}$ then $\Psi(v)=\Psi(w)$. 
\end{lemma}

\begin{proof}
Let $h=[(n,m),M_+]\in \aut{\G_F}$ as defined in \eqref{normalizer of F}, i.e. $M_+\in C_{GL_2(p)}(F)$.
%
Then
$$h S_v h^{-1}=\langle M(\widetilde v)\alpha_1^x\alpha_2^y,\ M(v)\alpha_1^{-y}\alpha_2^{x-\xi y}\rangle.$$
The equality $h S_v h^{-1} = S_w$ is equivalent to a system of linear equations, namely:
\begin{equation}\label{system for psi 1}
\begin{cases}
    M(\tilde v)=x \tilde w+y w\\
    M(v)=- y\tilde w + (x-\xi y)w.
    \end{cases}
\end{equation}
Given $v=\sigma^n\tau^m$ and $w=\sigma^s\tau^t$, we can translate \eqref{system for psi 1} into a system of linear equations for $x$ and $y$. The initial system of $4$ equations in $x$ and $y$ is equivalent to the following homogeneous system:
\begin{equation}\label{for normalizer}
\begin{bmatrix} m-t & -\xi m+s+n-1\\
n-s& -m+\xi s-t-1\end{bmatrix}\begin{bmatrix} x \\ y \end{bmatrix}=L\begin{bmatrix} x \\ y \end{bmatrix}=0.
\end{equation}
If the system admits non-trivial solutions, i.e. $S_v$ and $S_w$ are conjugate, then  $\mathrm{det}(L)=\psi(n,m)-\psi(s,t)=0$.

Setting $v=w$, it is easy to deduce that if $v=(\frac{1}{\xi+2},-\frac{1}{\xi+2})$, every $(x,y)\in \Z_p^2$ is a solution of \eqref{for normalizer}, i.e. $\aut{\G_1}$ normalizes $S_v$. Otherwise, the solutions to \eqref{for normalizer} are given by $y=0$.

A similar argument leads to the same condition if we consider $M_-$ as in \eqref{normalizer of F}, with the property $M_- F=F^{-1}M_-$. In such case we get that $S_{(\frac{1}{\xi+2},-\frac{1}{\xi+2})}$ is normalized by any $h=[(n,m),M_-]$, otherwise the coefficients of $M_-$ satisfy $x=0$ and $y\neq 0$.
\end{proof}

\begin{lemma}
	A set of representatives of regular subgroups $G$ of $\Hol(\G_F)$ such that $|\pi_2(G)|=p^2q$ is
	
	\bigskip
	\begin{center}
		\small{
			\begin{tabular}{l|l|c|c}
				Groups & Parameters  & Class & \#\\
				\hline
				&&&\\[-1em]
				$G_{c}=\langle u_c\alpha_1, f(u_c)\alpha_2,\epsilon^c f\rangle $    & $1\leq c\leq q-1,\, c\neq -2$  & $\G_{F}$& $q-2$ \\
				&&&\\[-1em]
				$H_{a}=\langle  \widetilde {v_a}\alpha_1, v_a\alpha_2,\epsilon^{-2} f\rangle$ & $1\leq a\leq p-1$ & $\G_F$ & $p-1$
		\end{tabular}}
		
		\bigskip
		
	\end{center}
	where $f=[(0,0),F]$, $u_c=\h(F^{c+1})^{-1}(F-1)^{-1}(F^c-1)(F^{c+2}-1)(\sigma)$, and $\setof{v_a\in \Z_p^2}{a\in \Z_p}$ is a set of representatives of the level sets of $\Psi$ with $\Psi(v_a)=a$ and $\widetilde{v_a}=F(v_a)-(1+F)(\sigma)$ (as defined in Lemma \ref{level of psi}).
\end{lemma}

\begin{proof}
	The subgroups $G_c$ in the statement are regular. Indeed $\epsilon^c\in \pi_1(G_c)$ and according to Lemma \ref{pi_1 for fix}(2) we have $\langle u_c, f(u_c)\rangle\subseteq \pi_1(G_c)$ and by Remark \ref{rem on F}(iii), $\{u_c, F(u_c)\}$ is a basis of $\Z_p^2$. So $|\pi_1(G_c)|>p^2$ and then $\pi_1(G_c)=\G_F$ (and similarly for $H_a$). 
	
	If $G_c$ is conjugate to $G_d$ or to $H_a$, then arguing as in Lemma \ref{G_A_sub_q} it follows that $c=d$ or $c=-2$. If $H_a$ and $H_b$ are conjugate then their $p$-Sylow subgroups are conjugate. Thus, by Lemma \ref{level of psi} we have $a=\Psi(v_a)=\Psi(v_b)=b$. 
	
	The unique subgroup of order $p^2q$ of $\aut{\G_F}$ is generated by $\alpha_1,\alpha_2$ and $f=[(0,0),F]$ and it is isomorphic to $\G_F$. Therefore, the standard presentation of the regular subgroup $G$ of $\Hol(\G_F)$ with $|\pi_2(G)|=p^2q$ is
	$$G=\langle u\epsilon^a\alpha_1,\ v\epsilon^b\alpha_2,\ w\epsilon^c f\rangle $$
	where $u,v,w\in \langle \sigma,\tau\rangle$. The \nc conditions are:
	\begin{align}
	(u\epsilon^a \alpha_1)^p =(v\epsilon^b \alpha_2)^p &=(w\epsilon^c f)^q=1 \label{R 1}\\
	(w\epsilon^c f)u\epsilon^a \alpha_1(w\epsilon^c f)^{-1}&=v\epsilon^b \alpha_2\label{R 2}\\
	(w\epsilon^c f)v\epsilon^b \alpha_2(w\epsilon^c f)^{-1}&=(u\epsilon^a \alpha_1)^{-1}(v\epsilon^b \alpha_2)^{-\xi}\label{R 3}.
	\end{align}
	From \eqref{R 1} we have $a=b=0$. If $c=0$, then by Lemma \ref{pi_1 for fix}(1), $\pi_1(G)\subseteq \langle \sigma,\tau \rangle$ and then $G$ is not regular. Thus $c\neq 0$ and \eqref{R 2} and \eqref{R 3} are equivalent to
	\begin{equation}\label{condition on u and v}
	\begin{cases}
	v=F^{c+1}(u)-(F-1)^{-1}(F^c-1)F(\sigma)\\
	h(F^{c+1})(u)=(F-1)^{-1}(F^c-1)(F^{c+2}-1)(\sigma).
	\end{cases}
	\end{equation}
	According to Remark \ref{rem on F}(ii) and (i), if $c\neq -2$ then 
	$\h(F^{c+1})$ is an automorphism of $\Z_p^2$. 
	If $c\neq -2$ then, substituting the second equation in the first one in \eqref{condition on u and v} and using that $F^2=-\xi F-1$ we have that \eqref{condition on u and v} is equivalent to
	\begin{equation}\label{condition on u and v 2}
	\begin{cases}
	v=F(u)\\
	u=h(F^{c+1})^{-1}(F-1)^{-1}(F^c-1)(F^{c+2}-1)(\sigma).
	\end{cases}
	\end{equation}
	Therefore $u=u_c$ is uniquely determined by $c$ and 
	$$G=\langle u_c\alpha_1,\ f(u_c)\alpha_2,\ w\epsilon^c f\rangle$$
	where $c\neq -2$. If $c+1=0$, the last condition in \eqref{R 1} implies that $w=0$ and so $G=G_{-1}$. 
	
		Otherwise, we will prove that $G$ is conjugate to $G_c$ by an element of the form $h=\alpha_1^n\alpha_2^m$ for some $n,m$. Indeed, $h$  centralizes $\langle u\alpha_1, f(u)\alpha_2\rangle$ and so $hGh^{-1}=G_c$ if and only if the conjugate of the last generator is in $G_c$, i.e.
	\begin{equation}\label{condition_G_c}
	 	h w \epsilon^c f h^{-1}=w'\alpha_1^{n+m}\alpha_2^{-n+(1+\xi)m}\epsilon^c f=(u_c \alpha_1)^{n+m}(f(u_c)\alpha_2)^{-n+(1+\xi )m}\epsilon^c f\in G_c.
	\end{equation} 
	Equation \eqref{condition_G_c} is equivalent to have that $w=nx+mF(x)$, where $$x=\frac{(F^c-1)(F^{c+1}-1)(F^{c+2}-1)}{(F-1)H(F^{c+1})}\neq 0.$$
		According to Remark \eqref{rem on F}(iii), $\{x,F(x)\}$ is a basis of $\Z_p^2$ and therefore $w\in \langle x, F(x)\rangle$.


	Let $c=-2$. In such case, the second condition in \eqref{condition on u and v} is trivial according to \eqref{solutions of h(F)=0} and the first one is equivalent to
	\begin{equation}\label{u and v c=-2}
	u=F(v)-(F+1)(\sigma)=\widetilde v.
	\end{equation}
	If $v=(x,y)$ then $\widetilde v=(-y-1,x-\xi y-1)$. Since $G$ is regular, then $u$ and $v$ are linearly independent, i.e.
	\begin{equation}\label{determinant for v}
	\det\begin{bmatrix}v & \widetilde v \end{bmatrix}=\det\begin{bmatrix}x & -y-1 \\ y & x-\xi y-1\end{bmatrix}=\Psi(x,y)\neq 0.
	\end{equation}
	The number of solutions of the equation $\Psi(x,y)=0$ is $p+1$, so there are $p^2-p-1$ choices for $v$.
	
	
The $p$-Sylow subgroup of $G$ is the group $S_{v}=\langle \widetilde v \alpha_1, v\alpha_2\rangle$. According to Lemma \ref{level of psi}, if $v =(\frac{1}{\k+2},-\frac{1}{\k+2})$ then $S_v$ is normalized by $\aut{\G_F}$, otherwise its normalizer has index $p+1$ in $\aut{\G_1}$. Accordingly, there are 
$$\frac{p^2-p-2}{p+1}+1=\frac{(p-2)(p+1)}{p+1}+1=p-1$$
choices for $v$ up to conjugation and the representatives can be chosen to be $\setof{v_a}{0\neq a\in \Z_p}$, since the value $\Psi(v)$ is an invariant. Thus, we can assume that $G$ is conjugate to a group of the form
$$G=\langle \widetilde v_a \alpha_1,\ v_a\alpha_2,\ w\epsilon^{-2} f\rangle $$
for $1\leq a\leq p-1$.
We show that $G$ is conjugate to some $H_{a}$ by $h\in \langle\alpha_1,\alpha_2\rangle$. Indeed, this is equivalent to have that
$$hw\epsilon^{-2}fh^{-1}=w\sigma^{m\frac{F^{-2}-1}{F-1}}\tau^{(\xi m-n)\frac{F^{-2}-1}{F-1}}\alpha_1^{n+m}\alpha_2^{m(1+\xi)-n}=(\widetilde{v}_a \alpha_1)^{n+m} (v_a\alpha_2)^{(1+\xi)m-n}.$$
Working out this condition, it turns out to be equivalent that $w\in \langle x,y\rangle$, where
$$x=\widetilde v-v+z,\quad y=\widetilde v +(1+\xi)v+F(z)$$
with $z=(F+\xi -1)(\sigma)$. It is easy to check that
$$\mathrm{det}\begin{bmatrix} x & y\end{bmatrix}=-(\xi+2)\Psi(v_a)\neq 0,$$
i.e. $x,y$ is a basis and so $G$ is conjugate to $H_a$.
%
\end{proof}

We summarize the content of this subsection in the following table:
\begin{table}[H]
	\centering
	
	\small{    \begin{tabular}{c|c|c}
			$|\ker{\lambda}|$ & $\Z_p^2\times\Z_q$ &  $\G_F$ \\
			\hline
			$1$  & - & $p+q-3$\\
			$q$  & $1$ &- \\
			$p^2$  & $1$ & $q-2$ \\
			$p^2q$  &- & $1$
	\end{tabular}}
	\vs
	\caption{Enumeration of skew braces of $\G_F$-type for $p=-1\pmod{q}$.}
\end{table}

\section{Skew braces of size $p^2q$ with $q=1 \pmod{p}$ and $q\neq 1 \pmod{p^2}$}\label{section:q=1(p)}

In this section we assume that $p$ and $q$ are primes such that $q=1 \pmod{p}$ and $q\not=1 \pmod{p^2}$ (including the case $p=2$ unless stated otherwise). Let $\r$ be a fixed element of order $p$ in $\Z_q^\times$. Since $p^2q\neq 12$, the relevant non-abelian groups for this section are the following:

\begin{itemize}
    \item[(i)] $\mathbb{Z}_{q}\rtimes_{\r} \mathbb{Z}_{p^2}=\langle \sigma,\tau\, | \, \sigma^{p^2}=\tau^{q}=1,\, \sigma\tau\sigma^{-1}=\tau^{\r} \rangle$,
	\item[(ii)] $\mathbb{Z}_{p}\times (\mathbb{Z}_{q}\rtimes_{\r} \mathbb{Z}_{p})=\langle \sigma,\tau,\epsilon\,|\, \sigma^p=\tau^p=\epsilon^q=1,\ [\epsilon,\tau]=[\tau, \sigma]=1,\, \sigma\epsilon\sigma^{-1}=\epsilon^{\r} \rangle$.
\end{itemize}

Tables \ref{table:q=1_p>2_1} and \ref{table:4q_1} summarize the enumeration of skew braces according to the isomorphism class of their additive and multiplicative groups.

\begin{table}[H]
	\centering
	\small{
		\begin{tabular}{c|c|c|c|c}
			$+ \backslash \circ$ &  $\Z_{p^2q}$ & $\Z_q\rtimes_{\r}\Z_{p^2}$ & $\Z_p^2\times\Z_q$ & $\Z_p\times(\Z_q\rtimes_{\r}\Z_p)$  \\
			\hline
			$\Z_q\rtimes_{\r}\Z_{p^2}$  & $2p$& $2p(p-1)$&- &- \\
			$\Z_p\times(\Z_q\rtimes_{\r}\Z_p)$   &- &-  & $4$ & $6p-4$
	\end{tabular}}
	\vs
	\caption{Enumeration of skew braces of size $p^2q$ with $q=1\pmod{p}$, $q\neq 1\pmod{p^2}$ and $p>2$.}
	\label{table:q=1_p>2_1}

	\small{
		\begin{tabular}{c|c|c|c|c}
			$+\backslash \circ$ & $\Z_{4q}$ & $\Z_q\rtimes_{-1}\Z_4$ & $\Z_2^2\times\Z_q$ & $\Z_2\times(\Z_q\rtimes_{-1}\Z_2)$  \\
			\hline
			$\Z_q\rtimes_{-1}\Z_4$ & $2$ & $2$ & $2$ & $4$ \\
			$\Z_2\times (\Z_q\rtimes_{-1}\Z_2)$ & $2$ & $2$ & $2$ & $4$
	\end{tabular}}
	\vs
	\caption{Enumeration of skew braces of order $4q$ with $q=1\pmod{2}$ and $q\neq 1\pmod{4}$. Notice that in this case we can assume $r=-1$.}
	\label{table:4q_1}
\end{table}

\subsection{Skew braces of $\mathbb{Z}_{q}\rtimes_{\r} \mathbb{Z}_{p^2}$-type}\label{subsection:q=1(p)_non-abelian_p-Sylow_cyclic}

In this section we denote by $A$ the group $\mathbb{Z}_{q}\rtimes_{\r} \mathbb{Z}_{p^2}$. Such group has the following presentation:
$$A=\langle \sigma,\ \tau\ \lvert\ \sigma^{p^2}=\tau^q=1, \, \sigma\tau\sigma^{-1}=\tau^{\r}\rangle.$$
According to \cite[Subsection 4.5]{auto_pq}, the map 
\begin{equation*}
	\Z_p\times (\Z_q \rtimes \Z_q^\times)     \longrightarrow \aut{A}, \qquad(k,j,i)\mapsto  \varphi_{i,j}^k =\begin{cases} \tau\mapsto \tau^i \\ \sigma \mapsto \tau^j\sigma^{kp+1}\end{cases}
\end{equation*}
is an isomorphism. In particular, $p^2q$ divides $|\aut{A}|=pq(q-1)$ and so we need to check all the possible values for the size of the image of regular subgroups under $\pi_2$. Note that $Z(A)=\langle \sigma^p\rangle$ and so $\langle \tau,\sigma^p\rangle$ is an abelian characteristic subgroup of $A$ of order $pq$.

The conjugacy classes of subgroups of $\aut{A}$ are given in Table \ref{table:subgroups_semidirect_q=1(p^2)}.

\begin{table}[H] 
	\centering
	
	\small{
		\begin{tabular}{c|c|c|c}
			Size &    Group & Parameters  & Class \\
			\hline
			&&&\\[-1em]
			$p$ & $\langle \varphi_{r,0}^k\rangle$ & $0\leq k\leq p-1$ & $\Z_p$ \\
			&&&\\[-1em]
			& $\langle \varphi_{1,0}^1\rangle $ &- &  \\
			&&&\\[-1em]
			\hline
			&&&\\[-1em]
			$q$ & $\langle \varphi_{1,1}^0\rangle$ &- & $\Z_q$ \\
			&&&\\[-1em]
			\hline
			&&&\\[-1em]
			$p^2$    & $\langle \varphi_{1,0}^1,\ \varphi_{\r,0}^0\rangle $ &- & $\Z_p^2$\\
			&&&\\[-1em]
			\hline
			&&&\\[-1em]
			$pq$ & $\langle \varphi_{\r,0}^k,\ \varphi_{1,1}^0\rangle $ & $0\leq k\leq p-1$ & $\Z_q\rtimes_\r \Z_p$ \\
			&&&\\[-1em]
			& $\langle \varphi_{1,0}^1,\ \varphi_{1,1}^0\rangle$ &- & $\Z_{pq}$ \\
			&&&\\[-1em]
			\hline
			&&&\\[-1em]
			$p^2q$ & $\langle \varphi_{1,0}^1,\ \varphi_{\r,0}^0,\ \varphi_{1,1}^0\rangle $ &- & $\Z_p\times (\Z_q\rtimes_\r \Z_p)$
			
	\end{tabular}}
	\vs
	\caption{Conjugacy classes of subgroups of $\aut{A}$.}\label{table:subgroups_semidirect_q=1(p^2)}
\end{table}

\begin{lemma}\label{case q of Z_q semi Z_p2}
	The unique skew brace of $A$-type with $|\ker{\lambda}|=p^2$ is $(B,+,\circ)$ where 
	\begin{eqnarray*} \begin{pmatrix} x_1 \\ y_1 \end{pmatrix} + \begin{pmatrix} x_2 \\ y_2 \end{pmatrix} 
		=\begin{pmatrix}   x_1+ \r^{y_1}x_2\\ y_1+y_2\end{pmatrix},\qquad \begin{pmatrix} x_1 \\ y_1 \end{pmatrix} \circ \begin{pmatrix} x_2 \\ y_2 \end{pmatrix} 
		=\begin{pmatrix}   \r^{y_2} x_1+ \r^{y_1}x_2\\ y_1+y_2\end{pmatrix}
	\end{eqnarray*}
	for every $0\leq x_1,x_2\leq q-1$ and every $0\leq y_1,y_2  \leq p^2-1$. In particular, $(B,\circ)\cong \Z_{p^2q}$.
\end{lemma}

\begin{proof}

Let us consider the group
	\[
	G=\langle \sigma,\ \tau^{\frac{1}{r-1}}\varphi_{1,1}^0\rangle\cong \Z_{p^2q}.
	\]
	The subset $\pi_1(G)$ contains $\langle \sigma\rangle$ and $\tau^{\frac{1}{r-1}}$ and then $|\pi_1(G)|>p^2$ and it divides $p^2q$. So $\pi_1(G)=A$ and according to Lemma \ref{rem for regularity} we have that $G$ is regular.

Let $G$ be a regular subgroup of $\Hol(A)$ with $|\pi_2(G)|=q$. According to Table \ref{table:subgroups_semidirect_q=1(p^2)}, we have that $\pi_2(G)=\langle \varphi_{1,1}^0\rangle$, that is a normal subgroup of $\aut{A}$. The subgroups of size $p^2$ of $A$ are all conjugate to $\langle\sigma\rangle$, so we have that $G$ has the standard presentation:
	\[
	G=\langle \sigma,\ \tau^a\sigma^b\varphi_{1,1}^0\rangle = \langle \sigma,\ \tau^a\varphi_{1,1}^0\rangle.
	\]
	By condition \NC, we have that $a=\frac{1}{\r-1}$. The formula follows by the same argument of Lemma \ref{Zp2 rtimes Z_q q}.
\end{proof}

In the following we consider the group $\mathfrak G_1=\langle  \sigma^p,\ \tau,\ \varphi^0_{1,1} \rangle\unlhd \Hol(A)$ with $\Hol(A)/\mathfrak G_1\cong \Z_{p}\times \Z_p\times \Z_q^{\times}$.

\begin{lemma}\label{prop:semi_p}
	A set of representatives of conjugacy classes of regular subgroups $G$ of $\Hol(A)$ with $|\pi_2(G)|=p$ is
	
	\begin{center}
		\small{
			\begin{tabular}{c|l|c|l|l|c}
				$\pi_2(G)$ & Groups & Parameters  & Class for $p>2$ & Class for $p=2$ & \#\\
				\hline
				&&&&&\\[-1em]
				$\langle \varphi_{1,0}^1\rangle$  &  $G_a=\langle \tau,\ \sigma^p,\ \sigma^a\varphi_{1,0}^1 \rangle$ & $1\leq a\leq p-1$, & $A$ & $\Z_2\times (\Z_q\rtimes_{-1} \Z_2)$ & $p-1$\\
				&&&&&\\[-1em]
				\hline
				&&&&&\\[-1em]
				$\langle \varphi_{\r,0}^k    \rangle$ &   $ H_{a,k}=\langle \tau,\ \sigma^p, \, \sigma^a\varphi_{\r,0}^k \rangle$ & $1\leq a\leq p-1$, & $\Z_{p^2q}$, if $a=p-1$, & $\Z_{4p}$, if $k=0$, & $p(p-1)$\\
				&&&&&\\[-1em]
				& &  $0\leq k\leq p-1$ & $A$, otherwise & $\Z_{2}^2\times \Z_p$, if $k=1$ &
		\end{tabular}}
	\end{center}
	
\end{lemma}

\begin{proof}
	The groups $G$ of order $p^2q$ in the statement have the general form:
	\[
	\langle \tau,\ \sigma^p,\ \sigma^a\theta\rangle
	\]
	for $1\leq a\leq p-1$ and $\theta\in\{\varphi_{1,0}^1,\varphi_{r,0}^k\}$. So $\langle\tau,\sigma^p\rangle\subseteq\pi_1(G)$ and $\sigma^a\in\pi_1(G)$. Thus $|\pi_1(G)|>pq$ and it divides $p^2q$. So, $\pi_1(G)=A$ and according to Lemma \ref{rem for regularity} we have that $G$ is regular.
	
	If $G_a$ and $G_b$ are conjugate, then their images in $\Hol(A)/\mathfrak G_1$ coincide, i.e. $\langle \sigma^a\varphi_{1,0}^1\rangle=\langle \sigma^b\varphi_{1,0}^1\rangle$ and so $a=b$ (the same idea applies to the groups $H_{a,k}$). Also, $G_a$ and $H_{b,k}$ are not conjugate since their images under $\pi_2$ are not.
	
	The only subgroup of size $pq$ in $A$ is $\langle \tau, \sigma^p \rangle$. If $G$ is a regular subgroup of $\Hol (A)$, we have two cases to consider according to Table \ref{table:subgroups_semidirect_q=1(p^2)}. If $\pi_2(G)=\langle\varphi_{1,0}^1\rangle$, then $G$ has the following standard presentation:
	\[
	G=        \langle \tau,\ \sigma^p,\ \tau^b \sigma^a\varphi_{1,0}^1 \rangle = \langle \tau,\ \sigma^p,\ \sigma^a\varphi_{1,0}^1 \rangle
	\]
	where $1\leq a\leq p-1$, i.e. $G=G_a$. 
	
	If $\pi_2(G)=\langle\varphi_{\r,0}^k\rangle$ for $0\leq k\leq p-1$, a standard presentation for $G$ is
	\[
	G = \langle \tau,\ \sigma^p,\ \sigma^a\varphi_{\r,0}^k \rangle
	\]
	with $1\leq a\leq p-1$, i.e. $G=H_{a,k}$.
\end{proof}

\begin{lemma}\label{prop:semi_pq}
	A set of representatives of conjugacy classes of regular subgroups $G$ of $\Hol(A)$ with $|\pi_2(G)|=pq$ is
	
	\begin{center}
		\small{
			\begin{tabular}{c|l|c|l|l|c}
				$\pi_2(G)$ & Groups & Parameters  & Class for $p>2$ & Class for $p=2$ & \#\\
				\hline
				&&&&&\\[-1em]
				$\langle \varphi_{1,0}^1,\ \varphi_{1,1}^0\rangle$  &  $G_a=\langle \sigma^p,\ \sigma^a \varphi_{1,0}^1,\ \tau^{\frac{1}{\r-1}}\varphi_{1,1}^0 \rangle$ & $1\leq a\leq p-1$ & $\Z_{p^2q}$ & $\Z_2\times (\Z_q\rtimes_{-1}\Z_2) $ & $p-1$\\
				&&&&&\\[-1em]
				\hline
				&&&&&\\[-1em]
				$\langle \varphi_{\r,0}^k,\ \varphi_{1,1}^0 \rangle$ &   $ H_{a,k}=\langle \sigma^p,\ \sigma^a\varphi_{\r,0}^k,\ \tau^{\frac{1}{\r-1}}\varphi_{1,1}^0  \rangle$ & $1\leq a\leq p-1$, & $A$ & $\Z_{4q}$, if $k=0$, & $p(p-1)$\\
				&&&&&\\[-1em]
				& &  $0\leq k\leq p-1$ & & $\Z_2^2\times \Z_q$, if $k=1$ &
		\end{tabular}}
	\end{center}
\end{lemma}

\begin{proof}	
	The same argument of Lemma \ref{prop:semi_p} shows that the groups in the statement are regular and that they are not pairwise conjugate. 

	Let $G$ be a regular subgroup of $\Hol(A)$ such that $|\pi_2(G)|=pq$. According to Table \ref{table:subgroups_semidirect_q=1(p^2)} we have two cases, and in both cases the image of $\pi_2$ is normal. Up to conjugation, the unique subgroup of size $p$ of $A$ is $\langle \sigma^p\rangle$ and so $\ker{\pi_2}=\langle \sigma^p\rangle$.
	If $\pi_2(G)=\langle \varphi_{1,0}^1,\, \varphi_{1,1}^0\rangle$, then
	\[
	G=\langle \sigma^p,\ \tau^b \sigma^a\varphi_{1,0}^1,\ \tau^c \sigma^d\varphi_{1,1}^0 \rangle.
	\]
	
	From \nc conditions we have:
\begin{itemize}
    \item the third generator has order $q$ modulo $\langle\sigma^p\rangle$ and so it follows that $d=0\pmod{p}$. \item the last two generators commute modulo $\langle\sigma^p\rangle$ and so $c=\frac{1}{\r-1}$.
\end{itemize}
	If $a=0\pmod{p}$ then by Lemma \ref{pi_1 for fix}(1) we have $\pi_1(G)\subseteq \langle \sigma^p,\tau\rangle$, contradiction by regularity. Thus we can assume that $1\leq a\leq p-1$ and so
	\[
	G=\langle \sigma^p,\ \tau^b\sigma^a \varphi_{1,0}^1,\ \tau^{\frac{1}{\r-1}}\varphi_{1,1}^0 \rangle.
	\]
	Now, the group $G$ is conjugate to $G_a$ by $h=\varphi_{1,n}^0$ where $n=b\frac{1-\r}{\r^a-1}$.
	
	If $\pi_2(G)=\langle \varphi_{\r,0}^k, \, \varphi_{1,1}^0 \rangle$ then, by the same computations of the previous case the standard presentation of $G$ is
	\[
	G=\langle \sigma^p,\ \tau^b\sigma^a\varphi_{\r,0}^k,\ \tau^{\frac{1}{\r-1}}\varphi_{1,1}^0  \rangle.
	\]
	By the same argument, we can also assume that $1\leq a\leq p-1$. If $a=p-1$, from $(\tau^b\sigma^{p-1}\varphi_{\r,0}^k)^p=\tau^{bp}(\sigma^{p-1}\varphi_{\r,0}^k)^p \in \langle\sigma^p\rangle$ it follows that $b=0$. Otherwise, $G$ is conjugate to $H_{a,k}$ by a power of $\varphi_{1,1}^0$. To do this, notice that $\varphi_{1,n}^0=(\varphi_{1,1}^0)^n$ centralizes the first and the third generators of $G$ and that
	\begin{equation*}
	    (\varphi_{1,n}^0)(\tau^b\sigma^a\varphi_{r,0}^k)(\varphi_{1,n}^0)^{-1}=\tau^{b+n\frac{r^a-1}{r-1}}\sigma^a(\varphi_{1,1}^0)^{(1-r)n}\varphi_{r,0}^k=\tau^{b+n\frac{r^{a+1}-1}{r-1}}(\tau^{\frac{1}{r-1}}\varphi_{1,1}^0)^{(1-r)n}\sigma^a\varphi_{r,0}^k.
	\end{equation*}
holds. Thus, $n=b\frac{1-\r}{\r^{a+1}-1}$ will do.
%
\end{proof}

\begin{proposition}\label{prop:non_existence_p2}
	Let $p>2$ and let $G$ be a regular subgroup of $\Hol(A)$. Then $|\pi_2(G)|\not\in\{p^2,p^2q\}$.
\end{proposition}

\begin{proof}
	If $p>2$ then 
	\begin{equation}\label{powers}
	(\sigma^a \varphi_{1,0}^1)^n=\sigma^{a\left(p\frac{n(n-1)}{2}+n\right)}\left(\varphi_{1,0}^1\right)^n
	\end{equation}
	for every $n\in \mathbb{N}$.
	
	Let $G$ be a subgroup of $\Hol(A)$ of size $p^2q$. If $|\pi_2(G)|=p^2$, then according to Table \ref{table:subgroups_semidirect_q=1(p^2)} and the fact that the kernel of $\pi_2$ is the normal $q$-Sylow subgroup of $A$, we have the following standard presentation for $G$:
	\[
	G=\langle \tau,\ \tau^b\sigma^a\varphi_{1,0}^1,\ \tau^d \sigma^c\varphi_{\r,0}^0 \rangle = \langle \tau,\ \sigma^a\varphi_{1,0}^1,\ \sigma^c\varphi_{\r,0}^0 \rangle.
	\]
	By condition \nc, $(\sigma^a\varphi_{1,0}^1)^p=(\sigma^c\varphi_{\r,0}^0)^p\in \langle \tau\rangle$. Using \eqref{powers}, we have $a=c=0\pmod p$.
	
	Assume that $|\pi_2(G)|=p^2q$. According to Table \ref{table:subgroups_semidirect_q=1(p^2)}, $G$ has the standard presentation:
	\[
	G=\langle \tau^a\sigma^b\varphi_{1,0}^1,\ \tau^c\sigma^d\varphi_{\r,0}^0,\ \tau^e\sigma^f\varphi_{1,1}^0 \rangle
	\]
	where the first two generators have order $p$ and the third one has order $q$. So, $f=0$ and $b=d=0\pmod{p}$ where we are using formula \eqref{powers} again. The group $\langle \sigma^p,  \tau\rangle$ is characteristic in $A$. Hence in both cases, according to Lemma \ref{pi_1 for fix}(1) we have $\pi_1(G)\subseteq \langle \sigma^p,\tau\rangle$, and so $G$ is not regular.
\end{proof}

The last result is quite different when we consider the case $p=2$.

\begin{lemma}\label{lem:4p_i_Qp_4}
	Let $p=2$. There exists a unique conjugacy class of regular subgroups of $\Hol(A)$ with $|\pi_2(G)|=4$. A representative is given by
	\[
	H=	\langle \tau,\ \sigma\varphi_{1,0}^1,\ \sigma^2\varphi_{-1,0}^0 \rangle \cong \Z_2\times (\Z_q\rtimes_{-1}  \Z_{2}).
	\]
\end{lemma}

\begin{proof}
	For the group $H$, we have that $\{\tau^n,\tau^n\sigma: 0\leq n\leq q-1\}\subseteq\pi_1(H)$ and $\sigma^2\in\pi_1(H)$. So $|\pi_1(H)|>2q$ and it divides $4q$. Thus, $H$ is regular by Lemma \ref{rem for regularity}.
	
	Let $G$ be a regular subgroup with $|\pi_2(G)|=4$. According to Table \ref{table:subgroups_semidirect_q=1(p^2)}, we have that $G$ has the following standard presentation:
	\[
	G=	\langle \tau,\ \sigma^a\varphi_{1,0}^1,\ \sigma^b\varphi_{-1,0}^0 \rangle
	\]
	where $1\leq a,b\leq 3$.
	By the condition \nc, $(\sigma^b \varphi_{-1,0}^0)^2=\sigma^{2b}\in \langle \tau\rangle$, thus we have $b=2$. So $a=1$ or $a=3$. If $a=1$, we have $G=H$ and if $a=3$, the group $G$ is conjugate to $H$ by $\varphi_{1,0}^1$.
\end{proof}

\begin{lemma}\label{lem:4p_i_Qp_4p}
	Let $p=2$. There exists a unique conjugacy class of regular subgroups with $|\pi_2(G)|=4q$. A representative is given by
	\[
	H=\langle \sigma\varphi_{1,0}^1,\ \sigma^2\varphi_{-1,0}^0,\ \tau^{-\frac{1}{2}}\varphi_{1,1}^0 \rangle \cong \Z_2\times (\Z_q\rtimes_{-1}  \Z_{2}).
	\]
\end{lemma}

\begin{proof}

	Analogously to Lemma \ref{lem:4p_i_Qp_4}, it is straightforward to check that the group in the statement is a regular subgroup of $\Hol(A)$.
	A regular subgroup with $|\pi_2(G)|=4q$ has the following standard presentation:
	\[
	G=\langle \tau^a\sigma^b\varphi_{1,0}^1,\ \tau^c\sigma^d\varphi_{-1,0}^0,\ \tau^e\sigma^f\varphi_{1,1}^0 \rangle.
	\]
	By condition \nc we have $(\tau^a\sigma^b\varphi_{1,0}^1)^2=(\tau^c\sigma^d\varphi_{-1,0}^0)^2=(\tau^e\sigma^f\varphi_{1,1}^0)^q=1$, and so $b\in\{1,3\}$, $d\in\{0,2\}$ and $f=0$. Since $\varphi_{1,0}^1$ is a central element in $\pi_2(G)$, by condition \nc again, we have that $\tau^a\sigma^b\varphi_{1,0}^1$ is central in $G$, hence $a=c$ and $e=-\frac{1}{2}$ and so
	\[
	G=\langle \tau^a\sigma^b\varphi_{1,0}^1,\ \tau^c\sigma^d\varphi_{-1,0}^0,\ \tau^{-\frac{1}{2}}\varphi_{1,1}^0 \rangle
	\]
	for some $1\leq a\leq q-1$ and the constraints on $b$ and $d$ from above. We can assume $b=1$, otherwise we conjugate by $\varphi_{1,0}^1$. If we conjugate by $\varphi_{1,-a}^0$, we get
	\[
	G=\langle \sigma\varphi_{1,0}^1,\ \sigma^d\varphi_{-1,0}^0,\ \tau^{-\frac{1}{2}}\varphi_{1,1}^0 \rangle.
	\]
	Finally, from regularity, it follows that $d\neq 0$, so $d=2$ and we have the group in the statement.
\end{proof}

We summarize the content of this subsection in the following tables:

\begin{table}[H]
	\centering
	
	\small{
		\begin{tabular}[t]{c|c|c}
			$|\ker{\lambda}|$ &  $\mathbb{Z}_{p^2q}$  &  $\Z_q\rtimes_r\Z_{p^2}$  \\
			\hline
			$p$ & $p-1$ & $p(p-1)$ \\
			$pq$ & $p$ & $p^2-p-1$ \\
			$p^2$ & $1$ &- \\
			$p^2q$ &- & $1$
		\end{tabular}
		\qquad
		\begin{tabular}[t]{c|c|c|c|c}
			$|\ker{\lambda}|$ &  $\Z_{4q}$ & $\Z_2^2\times\Z_q$ & $\Z_2\times (\Z_q \rtimes_{-1} \Z_{2})$ & $\Z_q\rtimes_{-1} \Z_4$ \\
			\hline
			$1$ &- &- & $1$ &- \\
			$2$ &- & $1$ & $1$ & $1$ \\
			$4$ & $1$ &- &- &- \\
			$q$ &- &- & $1$ &- \\
			$2q$ & $1$ & $1$ & $1$ &- \\
			$4q$ &- &- &- & $1$
	\end{tabular}}
	\vs
	\caption{Enumeration of skew braces of $A$-type for $q=1\pmod{p}$,  $q\neq 1\pmod{p^2}$ and $p>2$ and $p=2$ respectively.}
\end{table}

\subsection{Skew braces of $\mathbb{Z}_{p}\times(\mathbb{Z}_q\rtimes_{\r} \mathbb{Z}_p)$-type}\label{subsection:q=1(p)_non-abelian_non-cyclic_p-Sylow}

In this section we denote by $A$ the group $\mathbb{Z}_{p}\times(\mathbb{Z}_q\rtimes_{\r} \mathbb{Z}_p)$. A presentation of this group is
$$A=\langle \sigma,\tau,\epsilon\,|\, \sigma^p=\tau^p=\epsilon^q=1,\ [\epsilon,\tau]=[\tau, \sigma]=1,\, \sigma \epsilon \sigma^{-1}=\epsilon^{\r} \rangle.$$
According to \cite[Subsection 4.6]{auto_pq} the mapping
$$\phi:(\Z_{p}\rtimes\Z_p^\times)\times(\Z_q\rtimes \Z_{q}^\times)\longrightarrow \aut{A}, \quad [(l,i),(s,j)]\mapsto \alpha_{l,i}\beta_{s,j} $$
where $$\alpha_{l,i}=\begin{cases}
\epsilon\mapsto  \epsilon\\
\tau\mapsto \tau^i
\\ \sigma\mapsto \tau^l \sigma
\end{cases} \quad\text{and}\quad \beta_{s,j}=\begin{cases}
\epsilon\mapsto  \epsilon^j\\
\tau\mapsto \tau
\\ \sigma\mapsto \epsilon^s \sigma
\end{cases}$$	
is an isomorphism of groups. In particular, note that $\alpha_{l,i}$ and $\beta_{s,j}$ commute.
In particular, $p^2q$ divides $|\aut{A}|=	pq(p-1)(q-1)$ and so we need to discuss all the possible values for the size of the image of regular subgroups under $\pi_2$.
The conjugacy classes of subgroups of $\aut{A}$ are collected in Table \ref{lem:subgroups_with center}.

\begin{table}[H] 
	\centering
	
	\small{
		\begin{tabular}{c|c|c}
			Size &    Group  & Class \\
			\hline
			&&\\[-1em]
			$p$    & $\langle\alpha_{1,1}\rangle$  & $\Z_p$\\
			&&\\[-1em]
			& $\langle\beta_{0,\r}\rangle$   & \\
			&&\\[-1em]
			& $\langle\alpha_{1,1}\beta_{0,\r}\rangle$ & \\
			&&\\[-1em]
			\hline
			&&\\[-1em]
			$q$    & $\langle\beta_{1,1}\rangle$   & $\Z_q$\\
			&&\\[-1em]
			\hline
			&&\\[-1em]
			$pq$    & $\langle\alpha_{1,1},\ \beta_{1,1}\rangle$   & $\Z_{pq} $\\
			&&\\[-1em]
			& $\langle\beta_{0,\r},\ \beta_{1,1}\rangle$   & $\Z_q\rtimes_{\r} \Z_p $\\
			&&\\[-1em]
			& $\langle\alpha_{1,1}\beta_{0,\r},\ \beta_{1,1}\rangle$   & $\Z_q\rtimes_{\r} \Z_p $\\
			&&\\[-1em]
			\hline
			&&\\[-1em]
			$p^2$    & $\langle\alpha_{1,1},\ \beta_{0,\r}\rangle$   & $\Z_{p}^2$\\
			&&\\[-1em]
			\hline
			&&\\[-1em]
			$p^2q$    & $\langle\alpha_{1,1},\ \beta_{1,1},\ \beta_{0,\r}\rangle$   & $A$
			
	\end{tabular}}
	\vs
	\caption{Conjugacy classes of subgroups of $\aut{A}$.}\label{lem:subgroups_with center}
\end{table}

The procedure to check the regularity of a subgroup has already been established in the previous sections, so we are not showing that part of the classification strategy explicitly in what remains of Section \ref{section:q=1(p)}.

\begin{proposition}\label{with center q 2}
	The unique skew brace of $A$-type with $|\ker{\lambda}|=p^2$ is $B=B_1\times B_2$, where $B_1$ is the trivial skew brace of size $p$ and $B_2$ is the unique skew brace of non-abelian type of size $pq$ with $|\ker{\lambda_{B_2}}|=p$. 
	In particular, $(B,\circ)\cong \Z_p^2\times \Z_q$.
\end{proposition}

\begin{proof}
	The unique subgroup of size $q$ of $\aut{A}$ is generated by $\beta_{1,1}$ and the kernel is a $p$-Sylow subgroup of $A$ which can be taken as $\langle\sigma,\tau\rangle$ up to conjugation. Therefore $I=\langle \tau\rangle_+\leq \ker{\lambda}\cap \Fix(B)\cap Z(B,+)$ and so $I$ is an ideal of $B$ contained in $Z(B,\circ)$. The subgroup $J=\langle \epsilon, \sigma\rangle_+\unlhd (B,+)$ is a left ideal and since $\tau\in Z(B,\circ)$ then $J$ is an ideal of $B$. Therefore $B=I+J$ and $I\cap J=0$ and so $B$ is a direct product of the trivial skew brace of size $p$ and a skew brace $B_2$ of size $pq$ with $|\ker{\lambda_{B_2}}|=p$. According to \cite[Theorem 3.6]{skew_pq}, there exists a unique such skew brace and $(B_2,\circ)\cong \Z_{pq}$.
\end{proof}

In the following we consider the subgroup $\mathfrak G_2=\langle \epsilon,\tau, \alpha_{1,1}, \beta_{1,1}\rangle\unlhd \Hol(A)$. In particular, $\Hol(A)/\mathfrak G_2 \cong \Z_p\times \Z_p^\times   \times \Z_q^\times$.

\begin{lemma}\label{with center p} 
	A set of representatives of conjugacy classes of regular subgroups $G$ of $\Hol(A)$ with $|\pi_2(G)|=p$ is

	\begin{center}
		\small{
			\begin{tabular}{c|l|c|l|l|c}
				$\pi_2(G)$ & Groups & Parameters  & Class for $p>2$ & Class for $p=2$ & \#\\
				\hline
				&&&&&\\[-1em]
				$\langle \alpha_{1,1}\rangle$  &  $K=\langle \epsilon,\ \tau,\ \sigma\alpha_{1,1}\rangle$ & - & $A$ & $\Z_q\rtimes_{-1} \Z_4$ & $1$\\
				&&&&&\\[-1em]
				\hline
				&&&&&\\[-1em]
				$\langle \beta_{0,r}    \rangle$ &   $ L=\langle \epsilon,\ \sigma,\ \tau\beta_{0,\r}\rangle$ & - & $A$ & $A$ & $1$\\
				&&&&&\\[-1em]
				\cline{2-6}
				&&&&&\\[-1em]
				&    $G_{c}=\langle  \epsilon,\ \tau,\ \sigma^c\beta_{0,r} \rangle$ & $1\leq c\leq p-1$  & $ \Z_p^2\times \Z_q$, if $c=-1$, & $\Z_2^2\times \Z_q$ & $p-1$\\
				& & & $A$, otherwise &\\
				&&&&&\\[-1em]
				\hline
				&&&&&\\[-1em]
				$\langle \alpha_{1,1}\beta_{0,\r}\rangle$ &    $H_{c}=\langle  \epsilon,\ \tau,\ \sigma^c\alpha_{1,1}\beta_{0,\r} \rangle$ & $1\leq c\leq p-1$  & $ \Z_p^2\times \Z_q$, if $c=-1$, & $\Z_{4q}$ & $p-1$\\
				&&&&&\\[-1em]
				& & & $A$, otherwise & & 
		\end{tabular}}
	\end{center}
\end{lemma}

\begin{proof}
	The groups $K,L$, $G_c$ and $H_{c}$ are not conjugate since their images under $\pi_2$ or their kernels are not. If $G_{c}$ and $G_{d}$ (resp. $H_c$ and $H_d$) are conjugate, then their images in $\Hol(A)/\mathfrak G_2$ coincide, i.e. $\langle \sigma^c \beta_{0,r}\mathfrak G_2 \rangle=\langle \sigma^d \beta_{0,r}\mathfrak G_2 \rangle$. Thus, $c=d$. 
	
	Let $G$ be a regular subgroup of $\Hol(A)$ such that $|\pi_2(G)|=p$. According to Table \ref{lem:subgroups_with center} we need to discuss three cases. Assume that $\pi_2(G)=\langle\alpha_{1,1}\rangle$. Then the kernel has order $pq$ and so $G$ has the form
	$$G=\langle \epsilon,\ \sigma^n\tau^m,\ \sigma^a\tau^b\alpha_{1,1}\rangle.$$
	By condition \NC we have $n=0$. So we can assume $b=0$ and $G=\langle \epsilon,\ \tau,\ \sigma^a\alpha_{1,1}\rangle$. By regularity, $a\neq 0$ and so $G$ is conjugate to $K$ by $\alpha_{0,a}$. 
	
	Assume that $\pi_2(G)=\langle\beta_{0,\r}\rangle$. Then $G$ has the following standard presentation:
	$$G=\langle \epsilon,\ \sigma^n\tau^m,\ \sigma^a\tau^b\beta_{0,\r}\rangle.$$
	If $n=0$ then $G=G_{a}$. If $n\neq 0$, we can assume $m=0$, otherwise, we conjugate by $\alpha_{1,1}^{-\frac{m}{n}}$. Thus we can assume $a=0$ and $n=1$. By regularity, $b\neq 0$ and so $G$ is conjugate to $L$ by $\alpha_{0,b^{-1}}$.
	
	Finally, assume that $\pi_2(G)=\langle \alpha_{1,1}\beta_{0,\r}\rangle$. Then
	$$G=\langle \epsilon,\ \sigma^n\tau^m,\ \sigma^a\tau^b\alpha_{1,1}\beta_{0,\r}\rangle.$$
	The condition \NC implies that $n=0$ and so $G=\langle \epsilon,\ \tau,\ \sigma^a\alpha_{1,1}\beta_{0,\r}\rangle=H_{a}.$
	
	Notice that for the case $p=2$, the groups $K$ and $H_1$ (which is the only $H_c$) have elements of order $4$ since $(\sigma\alpha_{1,1})^2 =(\sigma\alpha_{1,1}\beta_{0,r})^2=\tau$. For $p>2$, there are no elements of order $p^2$ and then we have a different isomorphism class for the corresponding groups in the table of the statement.
\end{proof}

\begin{lemma}\label{with center pq}
	A set of representatives of conjugacy classes of regular subgroups $G$ of $\Hol(A)$ with $|\pi_2(G)|=pq$ is

	\begin{center}
		\small{
			\begin{tabular}{c|l|c|c|c|c}
				$\pi_2(G)$ & Groups & Parameters  & Class for $p>2$ & Class for $p=2$ & \#\\
				\hline
				&&&&&\\[-1em]
				$\langle \alpha_{1,1}, \beta_{1,1}\rangle$  &  $K_1=\langle \tau,\ \sigma \alpha_{1,1},\ \epsilon^{\frac{1}{\r-1}}\beta_{1,1}\rangle$ & - & $\Z_p^2\times \Z_q$ & $\Z_{4q}$ & $1$\\
				&&&&&\\[-1em]
				\hline
				&&&&&\\[-1em]
				$\langle \beta_{1,1}, \beta_{0,r}    \rangle$ &   $ K_2=\langle\sigma,\ \tau\beta_{0,\r},\ \epsilon^{\frac{1}{\r-1}}\beta_{1,1}\rangle$ & - & $A$ & $A$ & $1$\\
				&&&&&\\[-1em]
				\cline{2-6}
				&&&&&\\[-1em]
				&  $G_{a}=\langle \tau,\ \sigma^a\beta_{0,r},\ \epsilon^{\frac{1}{\r-1}}\beta_{1,1}\rangle$ & $1\leq a\leq p-1$  & $ A$ & $A$  & $p-1$\\
				&&&&&\\[-1em]
				\hline
				&&&&&\\[-1em]
				$\langle \beta_{1,1}, \alpha_{1,1}\beta_{0,\r}\rangle$ &    $H_{a}=\langle \tau,\ \sigma^a\alpha_{1,1}\beta_{0,\r},\ \epsilon^{\frac{1}{\r-1}}\beta_{1,1}\rangle$ & $1\leq a\leq p-1$  & $ A$ & $\Z_q\rtimes_{-1} \Z_4$  & $p-1$
		\end{tabular}}
	\end{center}
\end{lemma}

\begin{proof}
	The same argument as in Lemma \ref{with center p} shows that the groups in the statement are not conjugate. Let $G$ be a regular subgroup of $\Hol(A)$ with $|\pi_2(G)|=pq$. According to Table \ref{lem:subgroups_with center} we have to check three cases. In all such cases, up to the action of the normalizer of $\pi_2(G)$, we can assume that $\ker{\pi_2|_G}$ is generated by $\tau$ or $\sigma$.
	
	Let $\pi_2(G)=\langle \alpha_{1,1},\ \beta_{1,1}\rangle$. By the condition \NC we have that $\ker{\pi_2|_G}=\langle\tau\rangle$. So, 
	$$G=\langle \tau,\ \epsilon^a\sigma^b\alpha_{1,1},\ \epsilon^c\sigma^d\beta_{1,1}\rangle.$$
	The condition \nc implies that $d=0$ and $c=\frac{1}{\r-1}$. If $b=0$ then, by Lemma \ref{pi_1 for fix}(1), $\pi_1(G)\subseteq \langle \epsilon,\tau\rangle$, a contradiction since $G$ is regular. So $b\neq 0$ and therefore $G$ is conjugate to $K_1$ by $\alpha_{0,b}\beta_{1,1}^n$ where $n=a\frac{1-\r}{\r^b-1}$.
		
	Let $\pi_2(G)=\langle \beta_{0,\r},\ \beta_{1,1}\rangle$. If $\ker{\pi_2|_G}=\langle \sigma\rangle$, then 
	$$G=\langle\sigma,\ \epsilon^a\tau^b\beta_{0,\r},\ \epsilon^c\tau^d\beta_{1,1}\rangle.$$
	According to the \NC condition we have $c=\frac{1}{\r-1}$ and $a=0$ (so, $b\neq 0$ by regularity) and by \nc we have $d=0$. Thus $G$ is conjugate to $K_2$ by $\alpha_{0,b^{-1}}$.
		
	If $\ker{\pi_2|_G}$ is generated by $\tau$ then
	$$G=\langle \tau,\ \epsilon^a\sigma^b\beta_{0,\r},\ \epsilon^c\sigma^d\beta_{1,1}\rangle.$$
	From the \nc conditions we have that $d=0$ and $c=\frac{1}{\r-1}$. So, $b\neq 0$ otherwise $\pi_1(G)\subseteq \langle \tau,\epsilon\rangle$. If $b+1=0$, then $\epsilon$ commutes with $\sigma^b\beta_{0,r}$ and since $(\epsilon^a \sigma^b\beta_{0,r})^p\in \langle \tau\rangle$ we have $a=0$, and so $G=G_{-1}$. Otherwise $G$ is conjugate to $G_b$ by $\beta_{1,1}^{n}$ where $n=a\frac{1-\r}{r^{b+1}-1}$.
		
	Let $\pi_2(G)=\langle \alpha_{1,1}\beta_{0,\r},\ \beta_{1,1}\rangle$. The condition \NC implies that $\ker{\pi_2|_G}=\langle\tau\rangle$. Then
	$$G=\langle\tau,\ \epsilon^a\sigma^b\alpha_{1,1}\beta_{0,\r},\ \epsilon^c\sigma^d\beta_{1,1}\rangle.$$
	The \nc conditions imply that $d=0$, $c=\frac{1}{\r-1}$. We also have $b\neq 0$ as above. In the same fashion, if $b+1=0$, then $a=0$ and thus $G=H_{-1}$. Otherwise, $G$ is conjugate to $H_{b}$ by $\beta_{1,1}^{n}$ where $n=a\frac{1-\r}{\r^{b+1}-1}$.
	
	Similar to Lemma \ref{with center p}, note that for the case $p=2$, the groups $K_1$ and $H_1$ (which is the only $H_c$) have elements of order 4 because $(\sigma\alpha_{1,1})^2=(\sigma\alpha_{1,1}\beta_{0,r})^2=\tau$. So, the isomorphism classes of these groups are different to the case $p>2$ in the table of the statement.
\end{proof}

As in the previous subsection, for the following results we have different consequences when considering $p=2$ and $p>2$. Recall that if $p=2$ then we can assume $r=-1$.

\begin{proposition}\label{prop:4p_i_dihedral_non-existence}
	Let $p=2$. If $G$ is a regular subgroup of $\Hol(A)$, then $|\pi_2(G)|\not\in\{4,4q\}$.
\end{proposition}

\begin{proof}
	Let us assume that $G$ is a subgroup of size $p^2q$ of $\Hol(A)$ such that $|\pi_2(G)|=4$. By Table \ref{lem:subgroups_with center} and the fact that $A$ has a unique $q$-Sylow subgroup, we have that $G$ has the following standard presentation:
	\[
	G=    \langle \epsilon,\ \sigma^a\tau^b\alpha_{1,1},\ \sigma^c\tau^d\beta_{0,-1} \rangle
	\]
	for some $0\leq a,b,c,d\leq 1$. By condition \nc, since $(\sigma^a\tau^b\alpha_{1,1})^2\in\langle\epsilon\rangle$ we have $a=0$. We also have that the second and the third generators must commute modulo the kernel, so $c=0$.
	
	Now, assume that $|\pi_2(G)|=4q$, so $G$ has the following standard presentation:
	\[
	G= \langle \epsilon^a u \alpha_{1,1},\ \epsilon^b v \beta_{0,-1},\ \epsilon^c w \beta_{1,1} \rangle
	\]
	where $u,v,w\in\langle\sigma,\tau\rangle$. By condition \NC, the first generator has order $2$ and we have $u\in \langle\tau\rangle$. Since $\alpha_{1,1}\beta_{0,-1}=\beta_{0,-1}\alpha_{1,1}$, condition \nc leads to $v\in \langle\tau\rangle$. For the same reason, $\alpha_{1,1}\beta_{1,1}=\beta_{1,1}\alpha_{1,1}$ implies $w\in \langle\tau\rangle$.
	
	Then, in both cases $\pi_1(G)\subseteq \langle \tau,\epsilon\rangle$ by Lemma \ref{pi_1 for fix}(1) and so $G$ is not regular.
\end{proof}

The following two lemmas hold just for odd primes. In the proofs we are using that if $p>2$ then 
\begin{equation}\label{powers2}
(\sigma^a \alpha_{1,1})^n=\tau^{a\frac{n(n-1)}{2}}\sigma^{na}\alpha_{1,1}^n
\end{equation}
for every $n\in\mathbb{N}$.

\begin{lemma}\label{with center pp}
	Let $p>2$. A set of representatives of conjugacy classes of regular subgroups $G$ of $\Hol(A)$ with $|\pi_2(G)|=p^2$ is
	$$G_a=\langle \epsilon,\ \sigma^a\alpha_{1,1},\ \tau\beta_{0,\r}\rangle \cong A$$
	for $1\leq a\leq p-1$.
\end{lemma}

\begin{proof}
	If $G_a$ is conjugate to $G_b$ by $h\in \aut{A}$ then $h\in N(\pi_2(G))$ and so $h=\alpha_{l,i}\beta_{0,j}$. Then 
	\begin{eqnarray*}
		h \tau \beta_{0,r} h^{-1}&=&\tau^i \beta_{0,r}\in G_b\\
		h\sigma^a\alpha_{1,1} h^{-1}&=& \tau^{la}\sigma^a\alpha_{1,1}^i\in G_b.
	\end{eqnarray*}
	Then $i=1$ and $l=0$ and so we have that $a=b$.
	
	Let $G$ be a regular subgroup with $|\pi_2(G)|=p^2$. Then up to conjugation $\pi_2(G)=\langle   \alpha_{1,1},\beta_{0,r}\rangle$ and the kernel of $\pi_2|_G$ is the subgroup generated by $\epsilon$. Hence, we can assume that
	$$G=\langle \epsilon,\ \sigma^a\tau^b\alpha_{1,1},\ \sigma^c\tau^d\beta_{0,\r}\rangle.$$
	From the \nc condition, the second and the third generators must commute modulo the kernel, so we have $c=0$ and then $d\neq 0$ by regularity. If $a=0$ then according to Lemma \ref{pi_1 for fix}(1) we have $\pi_1(G)\subseteq \langle \tau,\epsilon\rangle$, a contradiction. Therefore $a\neq 0$ and $G$ is conjugate to $G_a$ by the automorphism $\alpha_{0,d^{-1}}\alpha_{1,1}^{-\frac{b}{a}}$.
\end{proof}

\begin{lemma}\label{with center ppq}
	Let $p>2$. A set of representatives of conjugacy classes of regular subgroups $G$ of $\Hol(A)$ with $|\pi_2(G)|=p^2q$ is
	$$G_{a}=\langle \sigma\alpha_{1,1},\ \tau^a\beta_{0,\r},\ \epsilon^{\frac{1}{\r-1}}\beta_{1,1}\rangle\cong A$$
	for $1\leq a\leq p-1$.
\end{lemma}

\begin{proof}
	If $G_a$ and $G_b$ are conjugate by $h$, then $h$ normalizes $\langle\sigma\alpha_{1,1}\rangle$, the center of $G_a$ and $G_b$, and so $h=\beta_{0,j}$ for some $j$. Then $h\tau^a\beta_{0,\r}h^{-1}=\tau^{a}\beta_{0,\r}\in G_b$. Therefore $a=b$.
	
	Let $G$ be a regular subgroup of $\Hol(A)$ with $|\pi_2(G)|=p^2q$. The unique subgroup of order $p^2q$ up to conjugation is $\langle \alpha_{1,1},\beta_{0,g},\beta_{1,1}\rangle\cong A$. Hence a standard presentation of $G$ is the following:
	\[G=\langle \epsilon^a u\alpha_{1,1},\ \epsilon^b v\beta_{0,\r},\ \epsilon^c w\beta_{1,1}\rangle
	\]
	for some $u,v,w\in \langle \tau,\sigma\rangle$. Let $u=\sigma^x\tau^y$. Since $\alpha_{1,1}$ is central in $\pi_2(G)$, by condition \nc, we have that $\epsilon^a u \alpha_{1,1}\in Z(G)$ and so $v,w\in\langle\tau\rangle$ and $$a=b\frac{r^x-1}{r-1} \quad\text{and}\quad c(r^x-1)=\frac{r^x-1}{r-1}.$$ Now, lifting the relation $\beta_{0,r}\beta_{1,1}=\beta_{1,1}^r\beta_{0,r}$ we have $w=1$. If $x=0$, by Lemma \ref{pi_1 for fix}(1) we have $\pi_1(G)\subseteq \langle \epsilon,\tau\rangle$, a contradiction. So, we assume $x\neq 0$ and then $c=\frac{1}{\r-1}$. Thus, 
	\[
	G=\langle \epsilon^{ b\frac{r^x-1}{r-1} }\sigma^x\tau^y\alpha_{1,1},\ \epsilon^b\tau^t\beta_{0,\r},\ \epsilon^{\frac{1}{\r-1}}\beta_{1,1} \rangle
	\]
	for some $0\leq t\leq p-1$. Hence, using \eqref{powers2} we have that $G$ is conjugate to $G_{xt}$ by $h=\alpha_{n,x}\beta_{-b,1}$ where $n=\frac{x-1}{2}-y$.
\end{proof}

The following remark is analogous to Remark \ref{direct of G_0 type}.

\begin{remark}\label{direct of other type}
	The skew braces of $A$-type which decompose as direct products are the following: 
	\begin{itemize}
		\item[(i)] the skew brace in Proposition \ref{with center q 2}.		
		\item[(ii)] the skew brace associated to the group $G_{c}$ for $1\leq c\leq p-1$ as defined in Lemma \ref{with center p} is the direct product of the trivial skew brace of size $p$ and a skew brace of size $pq$ with $|\ker{\lambda}|=q$, see \cite[Theorem 3.9]{skew_pq}.		
		\item[(iii)] The skew brace associated to $G_{a}$ for $1\leq a\leq p-1$ as defined in Lemma \ref{with center pq} is the direct product of the trivial skew brace of size $p$ and a skew brace of size $pq$ with $|\ker{\lambda}|=1$, see \cite[Theorem 3.12]{skew_pq}.
	\end{itemize}
	According to the enumeration of skew braces of size $pq$ collected in \cite[Theorems 3.6, 3.9 and 3.12]{skew_pq}, a complete list of skew braces of non-abelian type of size $p^2 q$ that decompose as a direct product is given by the skew braces above together with those in Remark \ref{direct of G_0 type}.
\end{remark}

We summarize the content of this subsection in the following tables.

\begin{table}[H]
	\centering
	
	\small{
		\begin{tabular}[t]{c|c|c}
			$|\ker{\lambda}|$ &  $\mathbb{Z}_p^2\times\Z_q$  &  $\Z_p\times(\Z_q\rtimes_g\Z_p)$  \\
			\hline
			$1$ &- & $p-1$ \\
			$p$ & $1$ & $2p-1$ \\
			$q$ &- & $p-1$ \\
			$pq$ & $2$ & $2(p-1)$ \\
			$p^2$ & $1$ &- \\
			$p^2q$ &- & $1$
		\end{tabular}
		\qquad 
		\begin{tabular}[t]{c|c|c|c|c}
			$|\ker{\lambda}|$ &  $\Z_{4q}$ & $\Z_2^2\times\Z_q$ & $\Z_2\times(\Z_q\rtimes_{-1}\Z_2)$ & $\Z_q\rtimes_{-1} \Z_4$ \\
			\hline
			$2$ & $1$ &- & $2$ & $1$ \\
			$2q$ & $1$ & $1$ & $1$ & $1$ \\
			$4$ &- & $1$ &- &- \\
			$4q$ &- &- & $1$ &-
	\end{tabular}    }
	\caption{Enumeration of skew braces of $A$-type for $q=1\pmod{p}$, $q\neq 1\pmod{p^2}$: the left table assumes $p>2$ and the right table assumes $p=2$.}\label{table with center q=1 (p)}
\end{table}

 \section{Skew braces of size $p^2q$ with $q=1\pmod{p^2}$}\label{section:q=1(p^2)}

In this section we assume that $q=1\pmod{p^2}$ and we will denote by $h$ a fixed element of order $p^2$ in $\Z_q^\times$. Accordingly, we have the following non-abelian groups of size $p^2q$: 

\begin{itemize}
	\item[(i)] $\mathbb{Z}_{q}\rtimes_{h^p} \mathbb{Z}_{p^2}=\langle \sigma,\tau\, | \, \tau^q=\sigma^{p^2}=1,\, \sigma\tau\sigma^{-1}=\tau^{h^p} \rangle$,
	\item[(ii)] $\mathbb{Z}_{p}\times (\mathbb{Z}_{q}\rtimes_{h^p} \mathbb{Z}_{p})=\langle \sigma,\tau,\epsilon\,|\, \sigma^p=\tau^p=\epsilon^q=1,\ [\epsilon,\tau]=[\tau, \sigma]=1,\, \sigma \epsilon \sigma^{-1}=\epsilon^{h^p}\rangle$,
	\item[(iii)] $\mathbb{Z}_{q}\rtimes_{h} \mathbb{Z}_{p^2}=\langle \sigma,\tau\, | \, \tau^q=\sigma^{p^2}=1,\, \sigma\tau\sigma^{-1}=\tau^h \rangle$.
\end{itemize}

In Tables \ref{table:q=1_p>2_2} and \ref{table:4q_2} we have the enumeration of skew braces according to their additive and multiplicative groups.

\begin{table}[H]
	\centering
	\small{
		\begin{tabular}{c|c|c|c|c|c}
			$+ \backslash \circ$ & $\Z_{p^2q}$ & $\Z_q\rtimes_{h^p}\Z_{p^2}$ & $\Z_q\rtimes_h\Z_{p^2}$ & $\Z_p^2\times\Z_q$ & $\Z_p\times(\Z_q\rtimes_{h^p}\Z_p)$  \\
			\hline
			$\Z_q\rtimes_{h^p}\Z_{p^2}$ &$2p$ &$2p(p-1)$ & $2p(p-1)$ &- &- \\
			$\Z_q\rtimes_h\Z_{p^2}$  & $2$ & $2(p-1)$ & $2p(p-1)$&- &-\\
			$\Z_p\times(\Z_q\rtimes_{h^p}\Z_p)$  & -  & - & - & $4$ & $6p-4$\\
	\end{tabular}}
	\vs
	\caption{Enumeration of skew braces of size $p^2q$ with $q=1\pmod{p^2}$ for $p>2$.}
	\label{table:q=1_p>2_2}
		
	\small{
		\begin{tabular}{c|c|c|c|c|c}
			$+\backslash \circ$ & $\Z_{4q}$  & $\Z_q\rtimes_{-1}\Z_4$ & $\Z_q\rtimes_{h}\Z_4$ & $\Z_2^2\times\Z_q$ & $\Z_2\times (\Z_q\rtimes_{-1}\Z_2)$\\
			\hline
			$\Z_q\rtimes_{-1}\Z_4$ & $2$ & $2$ & $2$ & $2$ & $4$ \\
			$\Z_q\rtimes_{h}\Z_4$ & $2$ & $2$ & $4$ & - & - \\
			$\Z_2 \times (\Z_q\rtimes_{-1}\Z_2)$ & $2$ & $2$ & $2$ & $2$ & $4$
	\end{tabular}}
	
	\vs
	\caption{Enumeration of skew braces of order $4q$ with $q=1\pmod{4}$. }
	\label{table:4q_2}
\end{table}

As the strategy to check regularity has been widely discussed in the previous sections, we are not referring to it from now on.

\subsection{Skew braces of $\mathbb{Z}_{q}\rtimes_{h^p} \mathbb{Z}_{p^2}$-type}\label{subsection:q=1(p^2)_non-abelian_cyclic_p-Sylow1}

In this section we denote by $A$ the group $\mathbb{Z}_{q}\rtimes_{h^p} \mathbb{Z}_{p^2}$. The automorphism group of $A$ can be described as in Subsection \ref{subsection:q=1(p)_non-abelian_p-Sylow_cyclic} and we employ the same notation. In particular, the subgroups of order $p,q$ and $pq$ of $\aut{A}$ coincide with the subgroups in Table \ref{table:subgroups_semidirect_q=1(p^2)}.
Therefore, if $G$ is a regular subgroup with $|\pi_2(G)| \in \{ q, p, pq \}$ we can apply Lemma \ref{case q of Z_q semi Z_p2}, \ref{prop:semi_p} and \ref{prop:semi_pq}, respectively.

Up to conjugation, the elements of order dividing $p^2$ are contained in the subgroup of $\aut{A}$ given by $\langle \varphi_{1,0}^1,\ \varphi_{h,0}^0\rangle\cong \Z_p\times \Z_p^2$. So, according to \cite[Theorem 3.3]{subgroups_finite_p-group}, it has $p+1$ subgroups of order $p^2$, namely
\begin{equation}\label{lem:sub_aut_p2_section6}
\langle \varphi_{h,0}^k\rangle\cong\Z_{p^2} \quad\text{and}\quad \langle \varphi_{1,0}^1,\ \varphi_{h^p,0}^0\rangle\cong\Z_p\times\Z_p
\end{equation}
for $0\leq k\leq p-1$.

On the other hand, since $\langle \varphi_{1,1}^0\rangle$ is the unique $q$-Sylow subgroup of $\aut A$, we have that the subgroups of size $p^2q$ in $\aut A$ up to conjugation are the following $p+1$ subgroups:
\begin{equation}\label{lem:sub_aut_p2q_section6}
\langle\varphi_{h,0}^k,\ \varphi_{1,1}^0\rangle \quad\text{and}\quad \langle \varphi_{1,0}^1,\ \varphi_{h^p,0}^0,\ \varphi_{1,1}^0\rangle
\end{equation}
for $0\leq k\leq p-1$.

The following lemmas enumerate the regular subgroups of $\Hol(A)$ up to conjugation with $|\pi_2(G)|\in\{p^2,p^2q\}$. In order to do this, we separate $p=2$ and $p>2$.

\begin{lemma}\label{Sec 7 pp}
	Let $p>2$. A set of representatives of conjugacy classes of regular subgroups $G$ of $\Hol(A)$ with $|\pi_2(G)|=p^2$ is
	$$  G_{b,k}= \langle \tau,\ \sigma^b\varphi_{h,0}^k \rangle\cong \Z_q\rtimes_h \Z_{p^2}$$
	for $0\leq k\leq p-1$ and $1\leq b\leq p-1$.
\end{lemma}

\begin{proof}
	Arguing as in Lemma \ref{prop:semi_p} we can show that the groups in the statement are not conjugate. Let $G$ be a regular subgroup of $\Hol(A)$ with $|\pi_2(G)|=p^2$. The unique subgroup of order $q$ of $A$ is $\langle \tau \rangle$. According to \eqref{lem:sub_aut_p2_section6} we have two cases to consider. 
	\begin{itemize}
		\item[(i)] If $\pi_2(G)=\langle \varphi_{h,0}^k \rangle$, then a standard presentation is
		\[
		G= \langle \tau, \tau^a\sigma^b\varphi_{h,0}^k \rangle = \langle \tau, \sigma^b\varphi_{h,0}^k \rangle
		\]
		where $b\neq 0 \pmod{p}$. Up to conjugation by the normalizer of $\langle\varphi_{h,0}^k\rangle$ in $\aut{A}$, we can assume $1\leq b\leq p-1$, so $G=G_{b,k}$.
		\item[(ii)] If $\pi_2(G)=\langle \varphi_{1,0}^1,\ \varphi_{h^p,0}^0 \rangle$, then we can argue as in Proposition \ref{prop:non_existence_p2}, and prove that there are no such regular subgroups. \qedhere
	\end{itemize}
\end{proof}

\begin{lemma}\label{Sec 7 ppq}
	Let $p>2$. A set of representatives of conjugacy classes of regular subgroups $G$ of $\Hol(A)$ with $|\pi_2(G)|=p^2q$ is
	\[
	G_{b,k} = \langle \sigma^b\varphi_{h,0}^k,\ \tau^{\frac{1}{h^p-1}}\varphi_{1,1}^0 \rangle \cong \Z_q\rtimes_h \Z_{p^2}
	\]
	for $1\leq b\leq p-1$ and $0\leq k\leq p-1$.
\end{lemma}

\begin{proof}
	The groups in the statement are not conjugate by the same argument of Lemma \ref{prop:semi_p}.
	Let $G$ be a regular subgroup of $\Hol(A)$ with $|\pi_2(G)|=p^2q$. According to \eqref{lem:sub_aut_p2q_section6} we have two cases to consider. In the first one, $G$ has the standard presentation:
	\[
	G=\langle \tau^a\sigma^b\varphi_{h,0}^k,\  \tau^c\sigma^d\varphi_{1,1}^0 \rangle.
	\]
	From the \nc condition $(\tau^a\sigma^b\varphi_{h,0}^k)^q=1$, we have $d=0$ and so $b\neq 0\pmod{p}$ (otherwise by Lemma \ref{pi_1 for fix}(1) it follows that $\pi_1(G)\subseteq \langle \sigma^p, \tau\rangle$). Moreover, we can also assume that $1\leq b\leq p-1$, up to conjugation by a power of $\varphi_{1,0}^1$.  From the \nc condition, lifting the relation $\varphi_{h,0}^k\varphi_{1,1}^0=(\varphi_{1,1}^0)^h\varphi_{h,0}^k$ we have 
	that $c=\frac{1}{h^p-1}$. 
	
	Now, if $a=0$, we have one of the representatives in the statement. If not, we conjugate by $\varphi_{n,1}^0$ where $n=-\frac{h^{pb+1}-1}{a(h^p-1)}$.
	
	In the second case we have no regular subgroups by the same argument of Proposition \ref{prop:non_existence_p2}.
\end{proof}

Now, we deal with the case $p=2$. Recall that since $h$ is an element of order $4$ in $\Z_q^\times$, we can assume that $h^2=-1$.

\begin{lemma}
	Let $p=2$. A set of representatives of conjugacy classes of regular subgroups $G$ of $\Hol(A)$ with $|\pi_2(G)|=4$ is given by
	\[
	G_1=\langle \tau,\ \sigma\varphi_{1,0}^1,\ \sigma^2\varphi_{-1,0}^0 \rangle \cong \Z_2\times (\Z_q\rtimes_{-1} \Z_2), \qquad G_2=\langle \tau,\ \sigma\varphi_{h,0}^0 \rangle\cong \Z_q\rtimes_{h} \Z_4.
	\]
\end{lemma}

\begin{proof}
	The groups $G_1$ and $G_2$ are not conjugate since their images under $\pi_2$ are not. Let $G$ be a regular subgroup of $\Hol(A)$. If $\pi_2(G)=\langle \varphi_{1,0}^1, \varphi_{-1,0}^0 \rangle$
	we can argue as in Lemma \ref{lem:4p_i_Qp_4} to get $G_1$. If $\pi_2(G)=\langle \varphi_{h,0}^k \rangle$ for $ k=0,1$, arguing as in Lemma \ref{Sec 7 pp}, we have 
	\[
	G=\langle \tau,\ \sigma\varphi_{h,0}^k \rangle.
	\]
	If $k=1$ then $(\sigma\varphi_{h,0}^1)^2=(\varphi_{h,0}^1)^2=\varphi_{-1,1}^0 \in G$ and so $G$ is not regular. Hence $k=0$ and so $G=G_2$.
\end{proof}

\begin{lemma}
	Let $p=2$. There exists two regular subgroups $G$ of $\Hol(A)$ with $|\pi_2(G)|=4q$ up to conjugation. A set of representatives is given by
	\[
	G_1=\langle \sigma\varphi_{1,0}^1,\ \sigma^2\varphi_{-1,0}^0,\ \tau^{-\frac{1}{2}}\varphi_{1,1}^0 \rangle \cong \Z_2\times (\Z_q\rtimes_{-1} \Z_2),\qquad G_2=\langle \sigma\varphi_{h,0}^0,\ \tau^{-\frac{1}{2}} \varphi_{1,1}^0\rangle \cong \Z_q\rtimes_{h} \Z_4.
	\]
\end{lemma}

\begin{proof}
	The groups $G_1$ and $G_2$ are not conjugate since their images under $\pi_2$ are not. Let $G$ be a regular subgroup of $\Hol(A)$. If $\pi_2(G)=\langle\varphi_{1,0}^1,\varphi_{-1,0}^0, \varphi_{1,1}^0\rangle$, arguing as in Lemma \ref{lem:4p_i_Qp_4p} we get $G_1$. If $\pi_2(G)=\langle\varphi_{h,0}^k,\ \varphi_{1,1}^0\rangle$, by the same argument as in Lemma \ref{Sec 7 ppq}, we have that $G$ has the following standard presentation:
	\[
	G=\langle \sigma\varphi_{h,0}^k,\ \tau^{-\frac{1}{2}}\varphi_{1,1}^0\rangle.
	\]
	If $k=1$ then $(\sigma\varphi_{h,0}^1)^2=(\varphi_{h,0}^1)^2=\varphi_{-1,0}^0\in G$ and so $G$ is not regular. So $k=0$ and we get $G_2$.
\end{proof}

The results of this subsection are summarized in the following tables:

\begin{table}[H]
		
	\centering
	\small{  \begin{tabular}[t]{c|c|c|c}
			$|\ker{\lambda}|$ & $\Z_{p^2q}$ &  $\Z_q\rtimes_{h^p}\Z_{p^2}$  & $\Z_q\rtimes_h\Z_{p^2}$    \\
			\hline
			$1$ &  - & - & $p(p-1)$\\
			$p$ & $p-1$  & $p(p-1)$ & - \\
			$q$ & - &  - & $p(p-1)$\\
			$p^2$ & $1$ & - & -\\
			$pq$ &$p$ & $p^2-p-1$ & - \\
			$p^2q$  &- & $1$ &- 
		\end{tabular}
		\quad
		\begin{tabular}[t]{c|c|c|c|c|c}
			$|\ker{\lambda}|$ &  $\Z_{4q}$ & $\Z_2^2\times\Z_q$ & $\Z_{2}\times(\Z_q\rtimes_{-1}\Z_2$) & $\Z_q\rtimes_{-1}\Z_4$ & $\Z_q\rtimes_{h}\Z_4$\\
			\hline
			$1$ &- &- & $1$ &- & $1$ \\
			$2$ &- & $1$ & $1$ & $1$ &- \\
			$q$ &- &- & $1$ &- & $1$ \\
			$4$ & $1$ &- &- &- &- \\
			$2q$ & $1$ & $1$ & $1$ &- &- \\
			$4q$ &- &- &- & $1$ &-
	\end{tabular}}
	\vs
	\caption{Enumeration of skew braces of $\Z_q\rtimes_{h^p}\Z_{p^2}$-type for $q=1\pmod{p^2}$: the first table assumes $p>2$ and the second table assumes $p=2$.}
\end{table}

\subsection{Skew braces of $\mathbb{Z}_p\times(\mathbb{Z}_q\rtimes_{h^p} \mathbb{Z}_{p})$-type}\label{subsection:q=1(p^2)_non-abelian_non-cyclic_p-Sylow}

In this section we denote by $A$ the group $\mathbb{Z}_p\times(\mathbb{Z}_{q}\rtimes_{h^p} \mathbb{Z}_{p})$. The regular subgroups of $\Hol(A)$ with size of the image under $\pi_2$ equal to $q,p,pq$ are the same as the one described in Proposition \ref{with center q 2}, Lemma \ref{with center p} and Lemma \ref{with center pq}, respectively, since the subgroups of the automorphism group of $A$ of order $p,q$ and $pq$ coincide with the ones in the case $q=1\pmod{p}$ and $q\neq 1\pmod{p^2}$. Following the notation in Subsection \ref{subsection:q=1(p)_non-abelian_non-cyclic_p-Sylow}, the subgroups of order $p^2$ in $\aut{A}$ are
\begin{equation}\label{subgroups for this case}
\langle \alpha_{1,1},\ \beta_{0,h^p}\rangle, \quad \langle \alpha_{1,1}^k\beta_{0,h}\rangle
\end{equation}
for $0\leq k\leq p-1$, up to conjugation. As we did before, we consider the cases $p=2$ and $p>2$ separately. First, we start with the case $p>2$.

\begin{lemma}\label{with center pp q=1 pp}
	Let $p>2$. A set of representatives of conjugacy classes of regular subgroups $G$ of $\Hol(A)$ with $|\pi_2(G)|=p^2$ is
	$$G_a=\langle \epsilon,\ \sigma^a\alpha_{1,1},\ \tau\beta_{0,h^p}\rangle\cong A$$
	for $1\leq a\leq p-1$.
\end{lemma}

\begin{proof}
	If $\pi_2(G)=\langle \alpha_{1,1},\beta_{0,h^p}\rangle$ we can argue as in Lemma \ref{with center pp} to get $G_a$ as in the statement. We show that if $\pi_2(G)=\langle \alpha_{1,1}^k\beta_{0,h}\rangle$ for $0\leq k\leq p-1$ then $G$ is not regular. In such case we have
	$$G=\langle \epsilon,\ \sigma^a\tau^b \alpha_{1,1}^k\beta_{0,h}\rangle.$$
	If $k=0$ then $(\sigma^a\tau^b\beta_{0,h})^p=\beta_{0,h}^p\in G$. Otherwise, since $p>2$, we have $$(\sigma^a\tau^b\alpha_{1,1}^k\beta_{0,h})^p=\tau^{bp+k\frac{p(p-1)a}{2}}\sigma^{ap}\alpha_{1,1}^{pk}\beta_{0,h}^p=\beta_{0,h}^p\in G.$$ In both cases, $G$ is not regular according to Lemma \ref{rem for regularity}.
\end{proof}

Now, we deal with the case $p=2$. Note that in this case we have $h^2=-1$.

\begin{lemma}
	Let $p=2$.	There exists a unique conjugacy class of regular subgroups $G$ of $\Hol(A)$ with $|\pi_2(G)|=4$. A representative is given by
	\[
	H=\langle \epsilon,\ \sigma\alpha_{1,1}\beta_{0,h} \rangle \cong \Z_{q}  \rtimes_h \Z_4.	
	\]
\end{lemma}

\begin{proof}
	If $\pi_2(G)$ is generated by $\alpha_{1,1}$ and $\beta_{0,-1}$, then $G$ is not regular by the same computations we performed in Proposition \ref{prop:4p_i_dihedral_non-existence}. Assume that $\pi_2(G)=\langle \alpha_{1,1}^k\beta_{0,h} \rangle$ for $k=0,1$, then $G$ has the following standard presentation
	\[
	G=\langle \epsilon,\ \sigma^a\tau^b\alpha_{1,1}^k\beta_{0,h} \rangle
	\]
	for some $0\leq a,b\leq 1$. If $k=0$, then $(\sigma^a\tau^b\beta_{0,h})^2=\beta_{0,h}^2=\beta_{0,-1}\in G$ and so $G$ is not regular. Then $k=1$ and $a=1$, otherwise $\pi_1(G)\subseteq \langle \tau,\epsilon\rangle$. Finally, we can assume that $b=0$, up to conjugation by $\alpha_{1,1}$.
\end{proof}

The subgroups of order $p^2q$ in $\aut{A}$ are
\begin{equation*}\label{subgroups for this case_2}
\langle \alpha_{1,1},\ \beta_{0,h^p},\ \beta_{1,1} \rangle, \quad \langle\beta_{1,1},\ \alpha_{1,1}^k\beta_{0,h}\rangle
\end{equation*}
for $0\leq k\leq p-1$, up to conjugation. We also consider the cases $p>2$ and $p=2$ separately. We start with the case $p>2$.

\begin{lemma}\label{with center ppq 2}
	Let $p>2$. A set of representatives of conjugacy classes of regular subgroups $G$ of $\Hol(A)$ with $|\pi_2(G)|=p^2q$ is
	$$G_{a}=\langle \sigma\alpha_{1,1},\,\tau^a\beta_{0,h^p},\,\epsilon^{\frac{1}{h^p-1}}\beta_{1,1}\rangle\cong A$$
	for $1\leq a\leq p-1$.
\end{lemma}

\begin{proof}
	If the $p$-Sylow subgroup of the image of $\pi_2(G)$ is not cyclic we can conclude as in Lemma \ref{with center ppq}.
	We show that if $\pi_2(G)=\langle \beta_{1,1},\ \alpha_{1,1}^k\beta_{0,h}\rangle$, then $G$ is not regular. Indeed if
	$$G=\langle \epsilon^a u\beta_{1,1},\ \epsilon^b v\alpha_{1,1}^k\beta_{0,h}\rangle$$
	for some $u,v\in \langle\sigma,\tau\rangle$ then, from $(\epsilon^a u\beta_{1,1})^q=1$ by condition \nc, we have $u=1$, i.e. $G=\langle \epsilon^a\beta_{1,1},\  \epsilon^b v\alpha_{1,1}^k\beta_{0,h}\rangle$ where $v=\sigma^c\tau^d$. Moreover, we can assume $a\neq 0$ by regularity. Since $p>2$, we have $$(\epsilon^b\sigma^c\tau^d\alpha_{1,1}^k\beta_{0,h})^p=\epsilon^s\tau^{dp+ck\frac{p(p-1)}{2}}\sigma^{cp}\alpha_{1,1}^{kp}\beta_{0,h}^p=\epsilon^s\beta_{0,h}^p \in G$$	for some $s$. So, 
	$$(\epsilon^{a}\beta_{1,1})^{-\frac{s}{a}}(\epsilon^s\beta_{0,h}^p)=\beta_{1,1}^{-\frac{s}{a}}\beta_{0,h}^{p}\in G$$
	and so $G$ is not regular.
\end{proof}

\begin{lemma}
	Let $p=2$.	There exists a unique conjugacy class of regular subgroups $G$ of $\Hol(A)$ with $|\pi_2(G)|=4q$. A representative is given by
	\[
	H=\langle \sigma\alpha_{1,1}\beta_{0,h},\ \epsilon^{-\frac{1}{2}}\beta_{1,1} \rangle \cong \Z_q\rtimes_h \Z_4.	
	\]
\end{lemma}

\begin{proof}
	If $\pi_2(G)=\langle \alpha_{1,1},\ \beta_{0,-1},\ \beta_{1,1}\rangle$, then $G$ is not regular using the same computations we performed in Proposition \ref{prop:4p_i_dihedral_non-existence}. Assume that $\pi_2(G)=\langle \alpha_{1,1}^k\beta_{0,h},\ \beta_{1,1}\rangle$ for $k=0,1$, so $G$ has the following standard presentation
	\[
	\langle \tau^a\epsilon^b\sigma^c\alpha_{1,1}^k\beta_{0,h},\ \tau^d\epsilon^e\sigma^f\beta_{1,1} \rangle
	\]
	for $0\leq a,c,d,f\leq 1$ and $0\leq b,e\leq q-1$.
	By condition \nc, since $(\tau^d\epsilon^e\sigma^f\beta_{1,1})^q=1$ we have $f=d=0$. By regularity, $e\neq 0\pmod{q}$. If $c=0$, then $\pi_1(G)\subseteq \langle\tau,\epsilon\rangle$. Thus, $c=1$. By condition \nc and the fact that $\beta_{0,h}\beta_{1,1}=\beta_{1,1}^h\beta_{0,h}$, we have $e=-\frac{1}{2}$. Up to conjugation by $\alpha_{1,1}$ we can assume $a=0$.
	
	If $k=0$ then $(\epsilon^{-\frac{1}{2}}\beta_{1,1})^{-2b(h-1)}(\epsilon^b\sigma\beta_{0,h})^2=\beta_{1,1}^{-2b(h-1)}\beta_{0,h}^2\in G$ and then $G$ is not regular. Then $k=1$ and $G$ is conjugate to $H$ by $\beta_{1,1}^n$ where $n=-\frac{2b}{h+1}$. 
\end{proof}

According to Lemma \ref{with center pp q=1 pp} and Lemma \ref{with center ppq 2}, the enumeration of the skew braces of $\Z_p\times(\Z_q\rtimes_{h^p}\Z_p)$-type with $q=1\pmod{p^2}$ and $p>2$ is as in Table \ref{table with center q=1 (p)}. Table \ref{table:dihedral_4p_ii} collects the enumeration for the case $p=2$.

\begin{table}[H]
	\centering
	
	\small{
		\begin{tabular}{c|c|c|c|c|c}
			$\ker{\lambda}$ &  $\Z_{4q}$ & $\Z_2^2\times\Z_q$ & $\Z_2\times (\Z_q\rtimes_{-1} \Z_2)$ & $\Z_q\rtimes_{-1} \Z_4$ & $\Z_q\rtimes_{h} \Z_4$\\
			\hline
			$1$ &- &- &- &- & $1$ \\
			$2$ & $1$ &- & $2$ & $1$ &- \\
			$4$ &- & $1$ &- &- &- \\
			$q$ &- &- &- &- & $1$ \\
			$2q$ & $1$ & $1$ & $1$ & $1$ &- \\
			$4q$ &- &- & $1$ &- &-
	\end{tabular}}
	\vs
	\caption{Enumeration of skew braces of $\Z_2\times(\Z_q\rtimes_{-1}\Z_2)$-type with $q=1\pmod{4}$.}
	\label{table:dihedral_4p_ii}
\end{table}

\subsection{Skew braces of $\mathbb{Z}_{q}\rtimes_h \mathbb{Z}_{p^2}$-type}\label{subsection:q=1(p^2)_non-abelian_cyclic_p-Sylow2} 

In this section, we denote by $A$ the group $\Z_q\rtimes_h \Z_{p^2}$. 
A presentation of such group is
$$G=\langle \sigma,\tau\ \lvert\ \sigma^{p^2}=\tau^{q}=1, \, \sigma \tau \sigma^{-1}=\tau^h\rangle.$$

According to \cite[Theorem 3.4]{auto_pq}, the map
$$\phi: \Z_q\rtimes \Z_q^\times\longrightarrow\aut{A},\quad (i,j)\mapsto \varphi_{i,j}=\begin{cases} \tau\mapsto \tau^j \\ \sigma \mapsto \tau^i\sigma\end{cases}$$
is a group isomorphism. Since $q=1\pmod{p^2}$, then $p^2q$ divides  $|\aut{G}|=q(q-1)$ and so we need to discuss all the possible values of $|\pi_2(G)|$. 

A set of representatives of the conjugacy classes of subgroups of $\aut{A}$ is displayed in Table \ref{lem:subgroups_aut_6.5}.

\begin{table}[H]
	\centering
	
	\small{
		\begin{tabular}{c|c|l} 
			Size &    Group  & Class \\
			\hline
			&&\\[-1em]
			$p$    & $\langle\varphi_{0,h^p}\rangle$   & $\Z_p$\\
			&&\\[-1em]
			$q$    & $\langle\varphi_{1,1}\rangle$   & $\Z_q$\\
			&&\\[-1em]
			$pq$    & $\langle\varphi_{0,h^p},\, \varphi_{1,1}\rangle$   & $\Z_q\rtimes_{h^p} \Z_p $\\
			&&\\[-1em]
			$p^2$    & $\langle\varphi_{0,h}\rangle$   & $\Z_{p^2}$\\
			&&\\[-1em]
			$p^2q$    & $\langle\varphi_{1,1}, \, \varphi_{0,h}\rangle$   & $A$
	\end{tabular}}
	\vs
	\caption{Conjugacy classes of groups of $\aut{A}$.}\label{lem:subgroups_aut_6.5}
\end{table}

\begin{lemma}
	The unique skew brace of $A$-type with $|\ker{\lambda}|=p^2$ is $(B,+,\circ)$ where
	\begin{eqnarray*} 
	\begin{pmatrix} x_1 \\ x_2 \end{pmatrix} + \begin{pmatrix} y_1 \\ y_2 \end{pmatrix} 
		=\begin{pmatrix}    x_1+ h^{x_2}y_1\\ x_2+y_2\end{pmatrix},\qquad
	\begin{pmatrix} x_1 \\ x_2 \end{pmatrix} \circ \begin{pmatrix} y_1 \\ y_2 \end{pmatrix} 
		=\begin{pmatrix}   h^{y_2} x_1+ h^{x_2} y_1\\ x_2+y_2\end{pmatrix}
	\end{eqnarray*}
	for every $0\leq x_1,y_1\leq q-1$ and $0\leq x_2,y_2  \leq p^2-1$. In particular, $(B,\circ)\cong \Z_{p^2q}$.
\end{lemma}

\begin{proof}
	Let $G$ be a regular subgroup of $\Hol(A)$ with $|\pi_2(G)|=q$. Then we can assume that $\pi_2(G)=\langle \varphi_{1,1}\rangle$ which is normal in $\aut{A}$. Since all $p$-Sylow subgroups of $A$ are conjugate to each other and have order $p^2$, we can assume that $G$ has the standard presentation
	\[
	G=\langle \sigma,\ \tau^a\sigma^b\varphi_{1,1} \rangle = \langle \sigma,\ \tau^{a}\varphi_{1,1} \rangle,
	\]
	for some $a\neq 0$. The condition \NC is fulfilled if and only if $a=\frac{1}{h-1}$.
	
	The structure of the skew brace associated to $G$ can be obtained as in Lemma \ref{case q of Z_q semi Z_p2}.
\end{proof}

The group $\mathfrak G_3=\langle \tau,\varphi_{1,1}\rangle$ is normal in $\Hol(A)$ and $\Hol(A)/\mathfrak G_3\cong \Z_{p^2}\times \Z_q^\times$.

\begin{lemma}\label{6.5}
	The skew braces of $A$-type with $|\ker{\lambda}|=pq$ are $(B_c,+,\circ)$ where $(B_c,+)=\Z_q\rtimes_h \Z_{p^2}$ and $(B_c,\circ)=\Z_q\rtimes_{h^{\frac{p}{c}+1}}\Z_{p^2}\cong A$ for $1\leq c\leq p-1$. In particular, $B_c$ is a bi-skew brace.
\end{lemma}

\begin{proof}
	The subgroup 
	$$             G_c = \langle \tau,\ \sigma^p,\ \sigma^c\varphi_{0,h^p} \rangle\cong A$$
	is regular. 
	If $G_c$ and $G_d$ are conjugate, then their images in $\Hol(A)/\mathfrak G_3$ coincide and so $\langle \sigma^p, \sigma^c\varphi_{0,h^p}\rangle=\langle \sigma^p, \sigma^d\varphi_{0,h^p}\rangle$. Then $c=d$.
	
	The unique subgroup of $A$ of order $pq$ is $\langle \tau, \sigma^p\rangle$. Hence if $G$ is a regular subgroup of $\Hol(A)$, we can assume that $G=G_c$ for some $1\leq c\leq p-1$.
	
	Using the same argument as in Lemma \ref{Zp2 rtimes Z_q q} we can describe the structure of the skew brace $B_c$ associated to $G_c$ and according to \cite[Corollary 1.2]{skew_pq}, the associated skew brace is a bi-skew brace.
\end{proof}

\begin{lemma}
	A set of representatives of conjugacy classes of regular subgroups $G$ of $\Hol(A)$ such that $|\pi_2(G)|=pq$ is
	\[
	G_c = \langle \sigma^p,\ \sigma^c\varphi_{0,h^p},\ \tau^{\frac{1}{h-1}}\varphi_{1,1} \rangle\cong A
	\]
	for $1\leq c\leq p-1$.
\end{lemma}
\begin{proof}
	Arguing as in Lemma \ref{6.5} we can show that the groups $G_c$ are not pairwise conjugate. Let $G$ be a regular subgroup of $\Hol(A)$ with $|\pi_2(G)|=pq$.
	According to Table \ref{lem:subgroups_aut_6.5}, we can assume that $\pi_2(G)=\langle \varphi_{0,h^p},\ \varphi_{1,1}\rangle$. The subgroups of order $p$ of $A$ are $\langle \tau^a\sigma^p\rangle$ for some $a$. Since $\varphi_{1,1}$ is in the normalizer of $\pi_2(G)$, up to conjugation by a power of  $\varphi_{1,1}$ we can assume that $a=0$. So, $G$ has the standard presentation
	\[ G = \langle \sigma^p,\ \tau^b\sigma^c\varphi_{0,h^p},\ \tau^d\sigma^e\varphi_{1,1} \rangle\]
	for some $0\leq c,e\leq p-1$ and $0\leq b,d\leq q-1$. From the condition \nc and the fact that $\varphi_{0,h^p}\varphi_{1,1}=\varphi_{1,1}^{h^p}\varphi_{0,h^p}$ it follows that $e(1-h^p)=0\pmod{p^2}$, so $e=0\pmod{p}$, and $d(h^c-1)=\frac{h^c-1}{h-1}$. If $c=0$, then $\pi_1(G)\subseteq \langle\tau,\sigma^p\rangle$ by Lemma \ref{pi_1 for fix}(1), so we have $c\neq 0$ and then $d=\frac{1}{h-1}$. From $(\tau^b\sigma^c\varphi_{0,h^p})^p\in\langle\sigma^p\rangle$ we get $b=0$. Thus $G=G_c$.
\end{proof}

\begin{lemma}
	A set of representatives of conjugacy classes of regular subgroups $G$ of $\Hol(A)$ such that $|\pi_2(G)|=p^2$ is
	$$G_b=\langle \tau,\ \sigma^b\varphi_{0,h} \rangle\cong \begin{cases} \Z_{p^2q}, \text{ if } b=-1\pmod{p^2},\\ \Z_q\rtimes_{h^p} \Z_{p^2},\text{ if } b=-1\pmod{p} \text{ and } b\neq -1\pmod{p^2},\\
	\Z_q\rtimes_h \Z_{p^2},\text{ otherwise},\\
	\end{cases}$$ where $1\leq b\leq p^2-1$ and $b\neq 0\pmod{p}$. 
\end{lemma}

\begin{proof}
	By the same argument as in Lemma \ref{6.5}, the groups in the statement are not pairwise conjugate. Let $G$ be a regular subgroup of $\Hol(A)$ with $|\pi_2(G)|=p^2$. Since $A$ has a unique subgroup of order $q$ and according to Table \ref{lem:subgroups_aut_6.5}, we have that every regular subgroup $G$ has the standard presentation
	\[
	G_b = \langle \tau,\ \sigma^b\varphi_{0,h} \rangle
	\]
	for $b\neq 0\pmod{p}$ (otherwise, $\pi_1(G)\subseteq\langle\tau,\sigma^p\rangle$), i.e. $G=G_b$.
\end{proof}

\begin{lemma}
	A set of representatives of conjugacy classes of regular subgroups $G$ of $\Hol(A)$ with $|\pi_2(G)|=p^2q$ is
	$$ G_{d} =\langle \tau^{\frac{1}{h-1}}\varphi_{1,1},\ \sigma^d\varphi_{0,h} \rangle\cong A$$ for $1\leq d\leq p^2-1$ and $d\neq 0\pmod{p}$.
\end{lemma}   
\begin{proof}
	Let $G$ be a regular subgroup with $|\pi_2(G)|=p^2q$. According to Table \ref{lem:subgroups_aut_6.5}, we can assume that a standard presentation is
	\[
	G=\langle \tau^a\sigma^b\varphi_{1,1},\ \tau^c\sigma^d\varphi_{0,h} \rangle.
	\]
	According to the \nc conditions, from $(\tau^a\sigma^b\varphi_{1,1})^q=1$ we have $b=0$. Lifting the relation $\varphi_{0,h}\varphi_{1,1}=\varphi_{1,1}^h\varphi_{0,h}$, we have the equation $a(h^d-1)=\frac{h^d-1}{h-1}$. If $d=0\pmod{p}$, then $\pi_1(G)\subseteq \langle \tau,\sigma^p\rangle$. So $d\neq 0\pmod{p}$ and $a=\frac{1}{h-1}$. Hence,
	\[
	G =\langle \tau^{\frac{1}{h-1}}\varphi_{1,1},\ \tau^c\sigma^d\varphi_{0,h} \rangle
	\]
	for $d\neq 0\pmod{p}$. If $d=-1$, then $c=0$ since $\tau^c\sigma^d\varphi_{0,h}$ has order $p^2$. Otherwise, $G$ is conjugate to $G_{d}$ by $\varphi_{-c\frac{h-1}{h^{d+1}-1},1}$.
\end{proof}

We summarize the content of this subsection in the following table:     

\begin{table}[H]
	\centering
	
	\small{
		\begin{tabular}{c|c|c|c}
			$|\ker{\lambda}|$ &  $\Z_{p^2q}$ & $\Z_q\rtimes_{h^p}\Z_{p^2}$ &  $\Z_q\rtimes_h\Z_{p^2}$ \\
			\hline
			$1$ &- &- & $p(p-1)$  \\
			$p$ &- & $p-1$ &- \\
			$q$ & $1$ & $p-1$ & $p(p-2)$ \\
			$p^2$ & $1$ &- &- \\
			$pq$ &- &- & $p-1$ \\
			$p^2q$ &- &- & $1$
	\end{tabular}}
	\vs
	\caption{Enumeration of skew braces of $\Z_q\rtimes_h\Z_{p^2}$-type for $q=1\pmod{p^2}$.}
\end{table}

\section{A proof to a conjecture on skew braces of size $p^2q$}\label{section:conjecture}

In this section, we give a proof to Conjecture 4.1 of \cite{skew_trick}. To do this, we need the tables we obtained in the present work and also in \cite{abelian_case} for the skew braces of abelian type (also known as \emph{classical} braces).

For an integer $n$, denote by $A(n)$ the number of non-isomorphic skew braces of abelian type of size $n$ (i.e. classical braces) and by $B(n)$ the number of non-isomorphic skew braces of non-abelian type of size $n$. So the total number of non-isomorphic skew braces $s(n)$ is given by:
\begin{equation}\label{def}
	s(n)=A(n)+B(n).
\end{equation}

From \cite{Dietzel}, it is known the value of $A(4q)$ for a prime $q\geq 5$ and also the value of $A(p^2q)$ for primes $p,q$ such that $q>p+1>3$. As a particular case, we reprove these results.

\begin{theorem}\cite[Conjecture 4.1]{skew_trick}
Let $p$ and $q$ be prime integers. If $q\geq 5$, then
\begin{equation*}
	s(4q)=\begin{cases} 29, & \text{if } q= 3 \pmod 4 \\ 43, & \text{if } q= 1 \pmod 4 \end{cases}
\end{equation*}
and if $q>p+1>3$, then
\begin{equation*}
s(p^2q)=\begin{cases} 4, & \text{if } q\neq 1 \pmod{p} \\ 2p^2+7p+8, & \text{if } q=1\pmod{p} \text{ and } q\neq 1\pmod{p^2} \\ 6p^2+6p+8, & \text{if } q=1\pmod{p^2}. \end{cases}
\end{equation*}

\end{theorem}

\begin{proof} For the first part of the statement, according to \cite[Table 1]{abelian_case} we have that
	\begin{equation*}
	A(4q)=\begin{cases} 9, & \text{if } q= 3 \pmod 4 \\ 11, & \text{if } q= 1 \pmod 4 \end{cases}
	\end{equation*}
	and adding up all the entries on Tables \ref{table:4q_1} and \ref{table:4q_2}, we have:
	\begin{equation*}
	B(4q)=\begin{cases} 20, & \text{if } q= 3 \pmod 4 \\ 32, & \text{if } q= 1 \pmod 4. \end{cases}
	\end{equation*}
	So, by \eqref{def} we have $s(4q)$ as desired.

    For the second part, according to \cite[p. $237$]{EnumerationGroups}, we have that both conditions $q>p+1>3$ and $q\neq 1\pmod{p}$ can only be fulfilled if $p$ and $q$ are arithmetically independent (that is $p\neq \pm 1\pmod q$ and $q\neq 1 \pmod p$). So, by \cite[Table 1]{abelian_case} we have:
	\begin{equation*}
	A(p^2q)=\begin{cases} 4, & \text{if } q\neq 1\pmod{p} \\ p+8, & \text{if } q=1\pmod{p} \text{ and } q\neq 1\pmod{p^2} \\ 2p+8, & \text{if } q=1\pmod{p^2}. \end{cases}
	\end{equation*}
	
	For $B(n)$, we have that the case $q=1\pmod{p}$ and $q\neq 1\pmod{p^2}$ comes from Table \ref{table:q=1_p>2_1} and the condition $q=1\pmod{p^2}$ comes from Table \ref{table:q=1_p>2_2}. Finally, since every group of order $p^2q$ with $p$ and $q$ arithmetically independent must be abelian, we only have braces of abelian type. Summarizing:
	 \begin{equation*}
	 B(p^2q)=\begin{cases} 0, & \text{if } q\neq 1\pmod{p} \\ 2p^2+6p, & \text{if } q=1\pmod{p} \text{ and } q\neq 1\pmod{p^2} \\ 6p^2+4p , & \text{if } q=1\pmod{p^2} \end{cases}
	 \end{equation*}
	 and then the second part of the conjecture follows. Notice that the numbers in the last equation were obtained adding up all the entries in Table \ref{table:q=1_p>2_1} and \ref{table:q=1_p>2_2}.
\end{proof}

\section*{Acknowledgement}
This work was partially supported by UBACyT 20020171000256BA and PICT 2016-2481. The authors want to thank Manoj Yadav and Leandro Vendramin for their comments on an earlier version of this work. We also thank the anonymous referee for many suggestions that improved the presentation of our work.

\bibliographystyle{abbrv}
\bibliography{refs}

\def\cprime{$'$}
\begin{thebibliography}{10}

\bibitem{abelian_case}
E.~Acri and M.~Bonatto.
\newblock Skew braces of size $p^2q$ {I}: abelian type.
\newblock {\em arXiv e-prints:2004.04291}, Apr 2020.

\bibitem{skew_pq}
E.~Acri and M.~Bonatto.
\newblock Skew braces of size $pq$.
\newblock {\em Communications in Algebra}, 48(5):1872--1881, 2020.

\bibitem{squarefree}
A.~A. Alabdali and N.~P. Byott.
\newblock Skew braces of squarefree order.
\newblock {\em Journal of Algebra and Its Applications},
  doi:10.1142/S0219498821501280.

\bibitem{p_cube}
D.~Bachiller.
\newblock Classification of braces of order {$p^3$}.
\newblock {\em J. Pure Appl. Algebra}, 219(8):3568--3603, 2015.

\bibitem{MR3835326}
D.~Bachiller.
\newblock Solutions of the {Y}ang-{B}axter equation associated to skew left
  braces, with applications to racks.
\newblock {\em J. Knot Theory Ramifications}, 27(8):1850055, 36, 2018.

\bibitem{skew_trick}
V.~G. Bardakov, M.~V. Neshchadim, and M.~K. Yadav.
\newblock Computing skew left braces of small orders.
\newblock {\em International Journal of Algebra and Computation},
  30(04):839--851, 2020.

\bibitem{EnumerationGroups}
S.~R. Blackburn, P.~M. Neumann, and G.~Venkataraman.
\newblock {\em Enumeration of finite groups}, volume 173 of {\em Cambridge
  Tracts in Mathematics}.
\newblock Cambridge University Press, Cambridge, 2007.

\bibitem{Caranti}
E.~Campedel, A.~Caranti, and I.~D. Corso.
\newblock {H}opf-{G}alois structures on extensions of degree $p^2q$ and skew
  braces of order $p^2q$: The cyclic {S}ylow $p$-subgroup case.
\newblock {\em Journal of Algebra}, 556:1165 -- 1210, 2020.

\bibitem{auto_pq}
E.~{Campedel}, A.~{Caranti}, and I.~{Del Corso}.
\newblock {The automorphism groups of groups of order $p^{2} q$}.
\newblock {\em arXiv e-prints:1911.11567}, Nov 2019.

\bibitem{biskew}
L.~N. {Childs}.
\newblock {Bi-skew braces and Hopf Galois structures.}
\newblock {\em {New York J. Math.}}, 25:574--588, 2019.

\bibitem{Crespo_2p2}
T.~Crespo.
\newblock {H}opf {G}alois structures on field extensions of degree twice an odd
  prime square and their associated skew left braces.
\newblock {\em Journal of Algebra}, 565:282 -- 308, 2021.

\bibitem{Dietzel}
C.~Dietzel.
\newblock Braces of order $p^2q$.
\newblock {\em Journal of Algebra and Its Applications},
  doi:10.1142/S0219498821501401.

\bibitem{MR1183474}
V.~G. Drinfel{\cprime}d.
\newblock On some unsolved problems in quantum group theory.
\newblock In {\em Quantum groups ({L}eningrad, 1990)}, volume 1510 of {\em
  Lecture Notes in Math.}, pages 1--8. Springer, Berlin, 1992.

\bibitem{MR1722951}
P.~Etingof, T.~Schedler, and A.~Soloviev.
\newblock Set-theoretical solutions to the quantum {Y}ang-{B}axter equation.
\newblock {\em Duke Math. J.}, 100(2):169--209, 1999.

\bibitem{MR1637256}
T.~Gateva-Ivanova and M.~Van~den Bergh.
\newblock Semigroups of {$I$}-type.
\newblock {\em J. Algebra}, 206(1):97--112, 1998.

\bibitem{MR3647970}
L.~Guarnieri and L.~Vendramin.
\newblock Skew braces and the {Y}ang--{B}axter equation.
\newblock {\em Math. Comp.}, 86(307):2519--2534, 2017.

\bibitem{kohl_4q}
T.~Kohl.
\newblock Groups of order $4p$, twisted wreath products and {H}opf--{G}alois
  theory.
\newblock {\em Journal of Algebra}, 314, 04 2007.

\bibitem{MR1769723}
J.-H. Lu, M.~Yan, and Y.-C. Zhu.
\newblock On the set-theoretical {Y}ang-{B}axter equation.
\newblock {\em Duke Math. J.}, 104(1):1--18, 2000.

\bibitem{NZ}
K.~{Nejabati Zenouz}.
\newblock {\em On Hopf-Galois Structures and Skew Braces of Order $p^{3}$}.
\newblock PhD thesis, The University of Exeter,
  https://ore.exeter.ac.uk/repository/handle/10871/32248, 2018.

\bibitem{NZpaper}
K.~Nejabati~Zenouz.
\newblock Skew braces and {H}opf-{G}alois structures of {H}eisenberg type.
\newblock {\em J. Algebra}, 524:187--225, 2019.

\bibitem{MR2278047}
W.~Rump.
\newblock Braces, radical rings, and the quantum {Y}ang-{B}axter equation.
\newblock {\em J. Algebra}, 307(1):153--170, 2007.

\bibitem{MR2298848}
W.~Rump.
\newblock Classification of cyclic braces.
\newblock {\em J. Pure Appl. Algebra}, 209(3):671--685, 2007.

\bibitem{Rump}
W.~Rump.
\newblock Classification of cyclic braces, {I}{I}.
\newblock {\em Trans. Amer. Math. Soc.}, 372:305--328, 2019.

\bibitem{Leandro-Byott}
A.~{Smoktunowicz} and L.~{Vendramin}.
\newblock {On skew braces (with an appendix by N. Byott and L. Vendramin).}
\newblock {\em {J. Comb. Algebra}}, 2(1):47--86, 2018.

\bibitem{MR1809284}
A.~Soloviev.
\newblock Non-unitary set-theoretical solutions to the quantum {Y}ang-{B}axter
  equation.
\newblock {\em Math. Res. Lett.}, 7(5-6):577--596, 2000.

\bibitem{subgroups_finite_p-group}
M.~T\u{a}rn\u{a}uceanu.
\newblock An arithmetic method of counting the subgroups of a finite abelian
  group.
\newblock {\em Bulletin math\'ematique de la Soci\'et\'e des Sciences
  Math\'ematiques de Roumanie}, 53 (101)(4):373--386, 2010.

\bibitem{YBE}
L.~Vendramin and A.~Konovalov.
\newblock Yang{B}axter, combinatorial solutions for the {Y}ang-{B}axter
  equation, version 0.9.0 (2019).
\newblock Available at https://gap-packages.github.io/YangBaxter.

\end{thebibliography}
 
 \end{document}